\documentclass[a4paper,english,10pt]{amsart}
\usepackage[applemac]{inputenc}
\usepackage{babel}
\usepackage{amsfonts}
\usepackage{amsmath}
\usepackage{amssymb}
\usepackage{amsthm}
\usepackage{graphicx}
\usepackage[all]{xy}
\usepackage{textcomp}
\usepackage{tikz-cd}
\usepackage{calligra,mathrsfs}
\usepackage{hyperref}
\hypersetup{
  colorlinks = true,
     urlcolor =   blue,
    linkcolor =  blue,
     citecolor = purple}

\usepackage[shortlabels]{enumitem}

\newtheorem{thm}{Theorem}[section]  
\newtheorem{cor}[thm]{Corollary}
\newtheorem{lem}[thm]{Lemma}
\newtheorem{defi}[thm]{Definition}
\newtheorem{prop}[thm]{Proposition}
\newtheorem{es}[thm]{Example}
\newtheorem{rem}[thm]{Remark}

\DeclareMathOperator{\Homs}{\mathscr{H}\text{\kern -3pt {\calligra\large om}}\,}

\title{A study of compatible deformations in non-Archimedean geometry}
\author{John Welliaveetil}

\begin{document}

\maketitle

\begin{abstract} 
In 2010, Hrushovski--Loeser showed that the Berkovich
 analytification of a quasi-projective variety over a non-Archimedean valued field admits a 
 deformation retraction onto a finite simplicial complex. 
In this article, we adapt the tools and methods developed by Hrushovski--Loeser 
to study if such deformation retractions can be obtained to be compatible with 
respect to a given morphism. Amongst other results, 
we show that compatible deformation retractions exist over 
a constructible partition of the base and prove the general statement in the case of a morphism 
of relative dimension 1 where the target is a smooth connected curve. 
\end{abstract}

 {\hypersetup {linkcolor = black} 
\tableofcontents
}
   \section{Introduction} 
    
          Over the course of the twentieth century there have been 
          several approaches, each with its merits, towards developing a
          theory of geometry over non-Archimedean valued fields similar to the theory over 
          the complex numbers. 
          However it was only
          in the late 1980's that Berkovich proposed a theory          that
          provided for the first time 
          non-Archimedean analytic spaces with good topological properties similar 
          to those that one takes for granted in complex geometry. 
          
              Berkovich proved several important results that shed light on 
              the topological nature of such analytic spaces. 
           In \cite{berk90}, he showed that Berkovich analytic spaces 
           are locally compact and locally path connected. Furthermore, there 
           is an analytification functor that takes 
           varieties over a non-Archimedean complete field $K$ to Berkovich $K$-analytic spaces and 
            several GAGA type results that serve to connect the topology 
            of the analytification of the variety to 
            scheme theoretic properties of the variety. 
            In \cite{berk93}, he showed that every 
            smooth $K$-analytic space is locally contractible. More generally, 
            if $X$ is locally isomorphic to a strictly $K$-analytic domain 
            of a smooth $K$-analytic space then $X$ is locally contractible. 
            The question of local contractibility and more general questions regarding the nature of the 
            homotopy type of a general $K$-analytic space remained open until 2010, when 
            Hrushovski and Loeser used techniques from Model theory to 
            study the homotopy types of the Berkovich analytification of quasi-projective varieties 
            over a valued field. 
            In \cite{HL}, they showed that the homotopy types of such analytifications were 
            determined by finite simplicial complexes embedded in the analytifications. 
            Furthermore, they proved that these spaces are locally contractible and 
             there can be only finitely many distinct homotopy 
            types of the analytifications of quasi-projective varieties 
            that vary in a family.  
                     It is important to note that they did not require any assumption of smoothness 
            on the underlying varieties. 
                          
                The work of Hrushovski and Loeser in \cite{HL} offers a 
             theory of non-Archimedean geometry which enables one to use powerful techniques and tools from 
             Model theory such as the notion of definability and model theoretic compactness. 
              A brief overview of \cite{HL} can be found in the excellent survey article
              \cite{DUCHL}.
             Inspired by the results and ideas in \cite{HL}, we study 
           the extent to which 
           the deformations of the analytifications of quasi-projective varieties onto finite 
           simplicial complexes is functorial. 
             More precisely, we study the validity of Statement \ref{statement}.  
             
             By a variety over a field $F$, we mean a reduced and separated scheme of finite type over $F$. We fix a field $K$ that is non-Archimedean non-trivially real valued, algebraically closed and
             complete.
\\
           
\noindent $\mathbf{Statement}$ $\mathbf{1:}$ \label{statement} 
\emph{Let $\phi \colon V' \to V$ be a flat
 surjective morphism between quasi-projective $K$-varieties of finite type. 
    There exist deformation retractions
    $H \colon I \times V^{\mathrm{an}} \to V^{\mathrm{an}}$ and 
    $H' \colon I \times V'^{\mathrm{an}} \to V'^{\mathrm{an}}$ 
    which are compatible with the morphism $\phi^{\mathrm{an}}$ i.e.
  the following diagram commutes.}
$$
 \begin{tikzcd} [row sep = large, column sep = large] 
I \times V'^{\mathrm{an}} \arrow[r,"H'"] \arrow[d, "id \times \phi^{\mathrm{an}}"] &
V'^{\mathrm{an}}  \arrow[d, "\phi^{\mathrm{an}}"] \\
I \times V^{\mathrm{an}} \arrow[r,"H"] &
V^{\mathrm{an}} 
\end{tikzcd}
$$
\emph{Furthermore, if $e$ denotes the end point of the interval $I$ then  
      the images of the deformations $\Upsilon := H(e,V^{\mathrm{an}})$
       and $\Upsilon' = H'(e,V'^{\mathrm{an}})$ are homeomorphic to
      finite simplicial complexes.}

 \subsection{Hrushovski-Loeser spaces}     
   
          As mentioned above, in order to study the homotopy types of certain non-Archimedean 
          analytic spaces, Hrushovski and Loeser introduced a model theoretic 
          analogue of the Berkovich space. 
          The first order properties of non-Archimedean valued fields can be described within
          the theory ACVF of algebraically closed valued fields.  
          Any model of ACVF is an algebraically closed valued field $L$
          whose value group we denote $\Gamma(L)$ and residue field 
          we denote $k(L)$.
          Let $\mathrm{val} \colon L \to \Gamma(L)$ denote the 
          valuation.
           Note that we write the group structure on $\Gamma(L)$ 
          additively.   
          
          Let $F$ be a valued field. 
         Analogous to the analytification functor, 
         to any variety $V$ over $F$,
          Hrushovski--Loeser associate 
          a functor $\widehat{V}$ that takes a model $L$ of ACVF extending 
          $F$ to the set $\widehat{V}(L)$ of $L$-definable stably dominated 
          types that concentrate on $V$.  
           The set 
           $\widehat{V}(L)$ can be endowed with a topology 
           whose pre-basic open sets are 
           of the form $\{p \in \widehat{U}| \mathrm{val}(f)_*(p) \in O\}$
            where 
           $O$ is an open subspace of the value group 
           $\Gamma_\infty(L)$, \footnote{$\Gamma_\infty(L) := \Gamma(L) \cup \{\infty\}$.}
            $U$ is a Zariski open subset of $V \times_F L$, 
            $f$ is regular on $U$, $\mathrm{val}(f) \colon U \to \Gamma_\infty$ 
            is the function $u \mapsto \mathrm{val}(f(u))$ and $\mathrm{val}(f)_*$ which is induced by $\mathrm{val}(f)$
            maps types on $U$ to types on $\Gamma_\infty$. 
            A map $f \colon V' \to V$ of $F$-varieties
            induces a continuous map
           $\widehat{f} \colon \widehat{V'} \to \widehat{V}$.
            Furthermore, we have an injection $V \hookrightarrow \widehat{V}$ such that the topology 
           induced on $V$ is the valuative topology.  
           More generally, this formalism allows us to define $\widehat{S}$ 
           when $S$ is any definable subset of $V$ and endow it with the induced topology. 
           
           The space $\widehat{V}$ is closely related to 
           the associated Berkovich space and homotopies constructed on 
           $\widehat{V}$ induce homotopies on $V^{\mathrm{an}}$ (cf. \cite[\S 14]{HL}). Furthermore,
           if $L$ is an extension of $F$
            that is maximally complete and $\Gamma(L) = \mathbb{R}$ then 
           $\widehat{V}(L) = V_L^{\mathrm{an}}$. 
           Lastly, $\widehat{V}$ is almost always not 
           definable, it is however strictly pro-$F$-definable i.e.
           a pro-object in the category of $F$-definable sets.  
           
             As stated previously,
             one of the goals
              of \cite{HL} is to study the homotopy type of the space $\widehat{V}$. 
              Theorem 11.1.1 in loc.cit. implies in particular that given a quasi-projective $F$-variety
               $V$ there exists a deformation 
            retraction 
            $H \colon I \times \widehat{V} \to \widehat{V}$ such that the image of 
            $H$ is an iso-definable subset $\Upsilon \subset \widehat{V}$ that
            admits an $F$-definable homeomorphism
            with a definable subset of $\Gamma_\infty^w$ where $w$ is a
            finite $F$-definable set.
            The interval $I$ is a generalized interval which means that it is obtained by glueing copies of
            $[0,\infty]$ end to end.

\subsection{Compatible deformations exist generically} 

    In \S \ref{finite obstruction}, we prove a version
     of Statement \ref{statement}
     for Hrushovski--Loeser spaces
      that holds generically over the base. 
    We show that if $\phi \colon V' \to V$ is a morphism between quasi-projective varieties then there exists an 
    open dense subspace $U \subset V$ such that we have deformations 
    $H'$ and $H$ of $\widehat{\phi^{-1}(U)}$ and $\widehat{U}$ respectively which are
    compatible with the morphism $\widehat{\phi}$.  
  Since our result holds only generically, we do not require that the morphism 
  $\phi \colon V' \to V$ be flat. 
  The precise statement is as follows. 
    \begin{thm} \label{generic cdr}
   Let $\phi \colon V' \to V$ be a morphism between quasi-projective $K$-varieties whose image is dense. Let $G$ be a finite algebraic group acting on $V'$
   which restricts to a well defined action along the fibres of the morphism $\phi$
   and 
   $\xi_i \colon V'\to \Gamma_{\infty}$ be a finite collection of $K$-definable functions.
   There exists an open dense subset $U \subseteq V$ such that if $V'_U := \phi^{-1}(U)$ then
   there exists a generalized interval $I$, deformation retractions $H \colon I \times \widehat{U} \to \widehat{U}$        
     and $H' \colon I \times \widehat{V'_U} \to \widehat{V'_U}$ which satisfy the following properties. 
       \begin{enumerate} 
      \item  The images of $H'$ and $H$ are iso-definable $\Gamma$-internal subsets of $\widehat{V'_U}$ and 
      $\widehat{U}$ respectively. Let $\Upsilon'$ and $\Upsilon$ denote the images of $H'$ and $H$ respectively. 
      \item The homotopy $H'$ is invariant for the action of the group $G$, 
    respects the levels of the functions $\xi_i$ and is compatible with the homotopy 
    $H$. 
    \item The homotopy $H'$ is Zariski generalizing. 
    \item  If the fibres of the morphism $\phi$ are pure over an open dense subset of $V$ and of dimension $n$ then
    for every $z \in \Upsilon$, $\Upsilon'_z$ is pure of dimension $n$ where 
   the notion of dimension for iso-definable $\Gamma$-internal 
    sets is as introduced in \cite[\S 8.3]{HL}. 
    If $Y \subset V$ is an irreducible component then 
    $\Upsilon \cap Y$ is pure of dimension $\mathrm{dim}(Y)$. 
 \item If $Z \subset V$ is a Zariski closed subset such that for every 
 irreducible component $Y$ of $V$, $Z \cap Y$ is strictly contained in $Y$ then 
 $\widehat{Z} \cap \Upsilon = \emptyset$. 
 \item If the variety $V$ is integral, projective and normal then 
       the deformation retraction $H$ extends to a deformation retraction 
      $I \times \widehat{V} \to \widehat{V}$ which preserves the closed subset 
      $\widehat{V \smallsetminus U}$. 
      Furthermore, if $d_{bord} \colon V \to \Gamma_\infty$ denotes the schematic
      distance \cite[\S 3.12]{HL} to $V \smallsetminus U$ then $H$ respects $d_{bord}$. 
      (In the event that $V$ satisfies assertions (1)-(5), we assume that (6) is vacuously true.)
         \end{enumerate}
     \end{thm} 
    
 Theorem \ref{generic cdr} implies the following theorem concerning 
 the analytifications of morphisms between quasi-projective varieties. 
 It can be easily deduced using 
 \cite[Corollary 14.1.6]{HL}. 
 
 \begin{cor} 
    Let $\phi \colon V' \to V$ be a surjective morphism between quasi-projective $K$-varieties. 
    There exists a finite partition $\mathcal{V}$
    of $V$ into locally closed sub-varieties such that 
    for every $W \in \mathcal{V}$, there exists a
     generalized real interval $I_W$ and a
      pair of deformation retractions 
     $$H'_{W} \colon I_W \times {V'_W}^{\mathrm{an}} \to {V'_W}^{\mathrm{an}}$$ 
     and 
     $$H_{W} \colon I_W \times W^{\mathrm{an}} \to W^{\mathrm{an}}$$
     which are compatible with respect to the morphism $(\phi_{|V'_W})^{\mathrm{an}}$
     and whose
      images are
     homeomorphic to finite simplicial complexes. 
 \end{cor}    
    
  
\subsubsection{Sketch of the proof of Theorem \ref{generic cdr}}
  
 If $V = \mathrm{Spec}(K)$ then Theorem \ref{generic cdr} is essentially 
 \cite[Theorem 11.1.1]{HL}. 
  Recall that in this case, the construction of the homotopy 
  $H'$ is the composition of inflation and curve homotopies 
  built using the following commutative diagram
   
 $$
\begin{tikzcd}[row sep = large, column sep = large]
V'_{1} \arrow[r,"g'"] \arrow[d,"b'"]  & E \arrow[r,"p"] \arrow[d,"b"] & F \\
V' \arrow[r,"g"] \arrow[d,"\phi"] & \mathbb{P}^m \\
\mathrm{Spec}(K)
\end{tikzcd}
$$   
  where the square is cartesian, $b$ is the blow up of $\mathbb{P}^m$ at a point, 
  $F := \mathbb{P}^{m-1} \subset \mathbb{P}^m$
   and 
  $p$ is the projection.  
  
  To prove Theorem \ref{generic cdr}, we make use of the tools developed in 
  \cite{HL}. As a first step, we use Lemma \ref{prove the simpler case}
  to reduce to the case where $V$ is integral, normal and projective, 
  $V'$ is projective and there exists a Zariski open dense subset 
  $U \subseteq V$ such that the fibres of the morphism $\phi$ are pure and 
  equidimensional over $U$. 
  In this situation, after shrinking $U$, Lemma \ref{Local factorization} implies the following diagram 
  $$
\begin{tikzcd}[row sep = large, column sep = large]
V'_{1U} \arrow[r,"g'_{U}"] \arrow[d,"b'_{U}"]  & E \times U \arrow[r,"p_{U}"] \arrow[d,"b_U"] & F \times U \\
V'_{U} \arrow[r,"g_{U}"] \arrow[d,"\phi_{U}"] & \mathbb{P}^m \times U \\
U
\end{tikzcd}
$$
  where the square is cartesian, $b_U$ is of the form $b \times \mathrm{id}_U$
  where $b \colon E \to \mathbb{P}^m$ is the blow up at a point and 
  $p_U := p \times \mathrm{id}_U$ where 
  $p \colon E \to F$ is the projection map. 
  
  The deformation retraction $H'$ as required by Theorem \ref{generic cdr} is the composition
  of homotopies 
  $H'_{\Gamma f} \circ H'_{bf} \circ H'_{curvesf} \circ H'_{inff} \circ H'_{inff-primary}$. 
  The homotopy $H'_{curvesf}$ is constructed on the
  Hrushovski--Loeser space associated to
   a suitable constructible subset 
  of $\widehat{V'_{1U}}$ while making use of the fact that 
  $p'_U := p_U \circ g'_U$ is a fibration of curves. We can extend to a
  homotopy on the entire space $\widehat{V'_{1U}}$ by composing with an inflation homotopy
  that makes use of the finite morphism $g'_U$ and by choosing an appropriate horizontal divisor 
  over $U$
  using Lemma \ref{existence of D} that guarantees that outside of this divisor, 
  for every $u \in U$, the morphism $\widehat{V'_{1u}} \to \widehat{E_u}$ 
  is a homeomorphism locally around the simple points. 
  We hence get a homotopy on $\widehat{V'_{1U}}$ 
  such that the fibres of the morphism $\widehat{p'_U}$ 
  restricted to this image are $\Gamma$-internal. 
  We then show as in \cite[Theorem 6.4.4]{HL} that 
  we can reduce to proving Theorem \ref{generic cdr}
  for a certain pseudo-Galois cover of $F \times U$. 
  Note that it is not automatic that a homotopy
  of $V'_{1U}$ that restricts to a well defined homotopy on 
  the exceptional divisor descends to a homotopy on $V'_U$. 
  To rectify this, we begin with an inflation homotopy 
  $H'_{inff-primary}$ on $\widehat{V'_U}$ that 
  enables us to escape the exceptional divisor. 
  The composition of the sequence of inflation and curve homotopies does not 
  fix their image. We construct as in \cite[\S 11.5]{HL} a 
  tropical homotopy $H'_{\Gamma f}$
   on an iso-definable $\Gamma$-internal subset of 
  $\widehat{V'_{1U}}$ such that the composition 
  $H'$ is indeed a deformation retraction. 
   At every stage, if necessary, we shrink the open set $U$. 

  An important point to note is that the homotopies 
  ($H'_{inff-primary}$), $H'_{inff}$, $H'_{curvesf}$ and $H'_{\Gamma f}$ 
  move along the fibres of the morphism 
  $(\widehat{\phi_U})$, $\widehat{\phi_U \circ b_U}$.
  Hence, we see that $H'$ firsts seeks
   to deform $\widehat{V'_U}$ into a subset 
  whose fibres are $\Gamma$-internal over $U$ followed by a deformation induced
  by an 
  appropriate deformation retraction of the base. 
  This is the content of Proposition \ref{generic relative cdr}.

\subsection{When the base is a curve}        
 In \S \ref{base is a curve}, we study Statement \ref{statement} when the base 
     is a smooth connected curve. While Theorem \ref{generic cdr} shows the existence of 
     deformations generically over the base, Theorem \ref{weak cdr for curves} proves that 
     if the base is a smooth connected curve then 
     we can get compatible homotopies over Zariski open neighbourhoods of an 
     arbitrary point. The precise statement is 
     Theorem \ref{weak cdr for curves}.
As before, \cite[Corollary 14.1.6]{HL} implies the following analogous statement 
for Berkovich spaces. 

\begin{cor}
    Let $S$ be a smooth connected $K$-curve and $X$ be a quasi-projective $K$-variety. 
    Let $\phi \colon X \to S$ be a surjective morphism such that every irreducible component of 
    $X$ dominates $S$. 
    Let $s \in S(K)$. There exists a Zariski open subset $U \subset S$ containing $s$ 
    and a
      pair of deformation retractions 
     $$H' \colon I \times X_U^{\mathrm{an}} \to X_U^{\mathrm{an}}$$ 
     and 
     $$H \colon I \times U^{\mathrm{an}} \to U^{\mathrm{an}}$$
     which are compatible with respect to the morphism $\phi^{\mathrm{an}}$
     and whose
      images are
     homeomorphic to finite simplicial complexes. 
     The interval $I$ is a generalized real interval. 
    \end{cor} 

    One observes from the 
    proof of Theorem \ref{generic cdr} that the existence of the 
    curve and inflation homotopies relies on the existence of divisors 
     which are flat over the base. 
      The reason we can prove the stronger result when the base is of dimension $1$ is a
      consequence of the following fact. Let $\phi \colon X \to S$ be as in the statement 
      of Theorem \ref{weak cdr for curves} and assume in addition that
       the fibres of $\phi$ are pure. 
       Let $D \subset X$ be Zariski closed and generically of codimension 1
       and pure 
      in the fibres of $\phi$ i.e. there exists an open set $U \subset X$ such that 
      if $s \in U$ then  $D \cap X_s$ is of codimension $1$ in $X_s$. 
      We suppose in addition that if $\eta \in S$ is the generic point of $S$ then 
      $\overline{D_\eta} = D$ where $\overline{D_{\eta}}$ denotes the Zariski closure 
      of $D_\eta$ in $X$. 
      It then follows that 
      $D$ is of codimension 1 and pure
       in every fibre of $\phi$. This follows from dimensionality concerns.

    \subsubsection{When $\phi$ is of relative dimension $1$.}
      
     One of the advantages of working within the framework of 
     Hrushovski--Loeser is the flexibility it allows us when 
     constructing our homotopies.  
     For instance, the construction of the inflation homotopy in
     \cite[Lemma 10.3.2]{HL} makes use of a suitable \emph{cut-off} 
     to extend a homotopy to the entire space. 
     This is the motivation behind Lemma \ref{foreshortening}
     which shows how we may extend 
     homotopies that move along the fibres of a projective morphism 
     in such a way that the extensions coincide with the original 
     homotopy over a reasonably large subspace of the base. 
     This result is the key ingredient in our proof of 
     Theorem \ref{statement 1 wbc} which verifies 
     Statement \ref{statement} 
     when $S$ is a smooth connected curve,
       the morphism $\phi$ is flat and the fibres of $\phi$ are of dimension 
       $1$.  
       Since the fibres are of dimension $1$, we go through 
       the steps in the proof of Theorem \ref{weak cdr for curves} to find that 
      for every $s \in S$, there exists a neighbourhood $U$ of $s$ and a homotopy 
      $h_U  \colon [0,\infty] \times X_U \to \widehat{X_U/U}$ 
      whose image is relatively $\Gamma$-internal. 
       We then extend 
       suitable cut-offs of 
       the homotopies $h_U$ to the whole of $X$ using Lemma \ref{foreshortening}.  
      As before, we have the following analogous result for Berkovich spaces. 
      
     \begin{cor} \label{statement 1 wbc}
    Let $S$ be a smooth connected $K$-curve and $X$ be a quasi-projective $K$-variety. 
    Let $\phi \colon X \to S$ be a surjective morphism such that every irreducible component of 
    $X$ dominates $S$. We assume in addition that the fibres of $\phi$ are of dimension $1$. 
     There exists a pair of deformation retractions 
     $$H' \colon I \times X^{\mathrm{an}} \to X^{\mathrm{an}}$$ 
     and 
     $$H \colon I \times S^{\mathrm{an}} \to S^{\mathrm{an}}$$
     which are compatible with respect to the morphism $\phi^{\mathrm{an}}$
     and whose
      images are
     homeomorphic to finite simplicial complexes. 
     The interval $I$ is a generalized real interval. 
    \end{cor} 
      

\subsection{Locally trivial morphisms}

       In \S \ref{rank 1 projective bundles}, we study the validity of Statement 
       \ref{statement} in the context of \emph{locally trivial morphisms} of relative dimension $1$. 
       A morphism $\phi \colon X \to S$ is said to be locally trivial if for every 
       point $s \in S(K)$, there exists a Zariski open neighbourhood 
       $U$ of $s$ and a $U$-isomorphism 
       $f_U \colon X_U \to V \times U$ where $V$ is a quasi-projective $K$-variety that is independent of $s$ and $U$. 
       In the event that $\widehat{V}$ admits a homotopy whose image is iso-definable $\Gamma$-internal and the generalized interval over which the homotopy runs is $[0,\infty]$, we can 
       employ 
       Lemma \ref{foreshortening} to verify the existence of 
       compatible homotopies 
       on $\widehat{X}$ and $\widehat{S}$ whose images 
       are iso-definable and $\Gamma$-internal. 
       This is the content of 
      Corollary \ref{locally trivial morphisms}.
       
       In the case when $V = \mathbb{P}^1_K$, we refer to the morphism as a $\mathbb{P}^1$-bundle. 
       The class of $\mathbb{P}^1$-bundles allows us examples of objects for which relative homotopies can be constructed explicitly. 
       We choose a system of coordinates on $\mathbb{P}^1$. 
       Recall from \cite[\S 7.5]{HL} that if we are 
       given a divisor $D$ on $\mathbb{P}^1$ then we have a canonical 
       deformation retraction of $\widehat{\mathbb{P}^1}$ onto the convex hull 
       of $D$. 
       In the case of a trivial family i.e. if $\phi \colon X \to S$ is 
       isomorphic over $S$ to $S \times \mathbb{P}^1$ and 
       $D \subset X$ is a divisor that is finite and surjective over $S$,
        Hrushovski and Loeser construct 
       deformations of 
       $\widehat{X}$ that move along the fibres of the morphism $\widehat{\phi}$ and 
       whose image over a simple point $s \in S$ is the convex hull of the divisor 
       $D_s$ in $\widehat{\mathbb{P}^1} \times \{s\}$. 
        In Theorem \ref{horizontal divisor implies deformation}, we show that 
        the existence of a suitable horizontal divisor $D$
        on $X$ enables us to glue the standard homotopies 
          on each trivializing chart with respect to the restriction of the divisor 
          $D$ to that chart. 
          
 \subsection{Related work} 
 
          An interesting variation of Statement \ref{statement} can 
          be found in the tropical geometry literature. 
          In \cite{MU17}, Ulirsch constructs a functorial tropicalization morphism 
          for fine and saturated log schemes. 
          More precisely, 
          let $k$ be a trivally valued field. 
          Given a fine and saturated log scheme $X$ which is locally of finite type 
          over $k$, we can define a generalized cone complex $\Sigma_X$ and 
          construct a continuous tropicalization morphism 
          $\mathrm{trop}_X \colon X^{\beth} \to \overline{\Sigma}_X$ where 
          $\overline{\Sigma}_X$ is the canonical extension of $\Sigma_X$.
         Theorem 1.1 of loc.cit. proves an analogue of Statement \ref{statement} 
          in this context by showing that 
          if $f \colon X' \to X$ is a morphism of 
          fine and saturated log schemes locally of finite type 
          over $k$ then we have an induced morphism of generalized cone complexes 
          $\Sigma(f) \colon \Sigma_{X'} \to \Sigma_X$ such that 
          $\Sigma(f) \circ \mathrm{trop}_{X'} = \mathrm{trop}_X \circ f^{\beth}$. \\ 
            
    \noindent \textbf{Acknowledgements:} The work presented here 
    is inspired and motivated 
    by the theory of non-Archimedean geometry introduced in \cite{HL}. 
    We are greatly indebted to the tools and techniques developed in this paper. 
     We are grateful to Fran\c{c}ois Loeser for his patient explanations of 
     several difficult parts of loc.cit. and his advice regarding this project. 
    We must also thank Yimu yin and Andreas Gross for the many discussions.  
   We are grateful as well to Kavli, IPMU for the support during the writing of 
   the article.

 \section{Hrushovski-Loeser spaces}
  
     
     Model theory seeks to understand mathematical structures, the 
     first order sentences that hold true in these structures and those sets which are defined by first order formulae. 
     Thus, model theory gives us tools applicable in a variety of settings alongside providing 
     us a framework which permits us to relate different mathematical structures. 
     The Hrushovski-Loeser space allows us to make use of these powerful tools 
     to study the topology of objects that occur in non-Archimedean geometry. 
     This approach which was first introduced in 2010 opens up a completely new perspective 
     when thinking of certain Berkovich spaces.
      In what follows, we provide a brief introduction of the Hrushovski-Loeser space.
      The details can be found in \cite{HL}.
      We assume that the reader is familiar with the fundamental notions of
      model theory such as \emph{languages, structures, formulae, theories} and \emph{types.}
      Chapters 1 - 4 of \cite{M02} explain these concepts perfectly.
      The following two sections 
      are more or less reproductions 
      of parts of the paper \cite{JW20}.

\subsection{The theory ACVF}
 
 Hrushovski-Loeser spaces can be thought of as a model theoretic analogue of 
  Berkovich spaces, developed within the framework of ACVF - the theory of 
  \emph{algebraically closed non-trivially valued fields}. 
  Note that working within ACVF does not prevent us from 
  studying objects which defined over valued fields which are not algebraically closed. 
  Let us now introduce the language $\mathcal{L}_{k,\Gamma}$ to describe
  ACVF and then explain why we extend this language 
  to $\mathcal{L}_{\mathcal{G}}$. 
  
 \begin{defi} 
       {The multi-sorted language} $\mathfrak{L}_{k,\Gamma}$
                         \emph{is given by specifying the following set of data.}
         \begin{enumerate}  
           \item  \emph{A set $\mathfrak{S}$ of sorts  - consisting of a valued field sort $\mathrm{VF}$, a residue field sort $k$ and 
                    a value group sort $\Gamma$.} 
    \item \emph{A set of function symbols $\mathfrak{F} = \{+_{\mathrm{VF}},\times_{\mathrm{VF}},-_{\mathrm{VF}},+_k,\times_k,-_{k}, +_{\Gamma}, \mathrm{val}, \mathrm{Res} \}$
                   which are defined on appropriate sorts. For instance, 
                   $+_{\mathrm{VF}} \colon \mathrm{VF} \times \mathrm{VF} \to \mathrm{VF}$, 
                 $\mathrm{val} \colon \mathrm{VF}^* \to \Gamma$, $\mathrm{Res} \colon \mathrm{VF}^2 \to k$ and so on.} 
         \item \emph{A set of relation symbols 
                           $\mathfrak{R} = \{<\}$ defined on specific sorts. 
                          For instance, $< \subset \Gamma \times \Gamma$. }
   \item \emph{A set of constant symbols
              $\mathfrak{C} = \{0_{\mathrm{VF}}, 1_{\mathrm{VF}},0_{k}, 1_{k},0_{\Gamma}\} $ 
              which are sort specific.} 
                        \end{enumerate}                 
        \end{defi} 
        
        ACVF is that $\mathcal{L}_{k,\Gamma}$-theory whose models are algebraically closed non-trivially valued fields. 
      Formally, we provide the following definition. 
                
        \begin{defi}
       \emph{The $\mathcal{L}_{k,\Gamma}$-theory} $\mathrm{ACVF}$ \emph{consists of the set of sentences
       such that if $M$ is a model of $\mathrm{ACVF}$
        then $\mathrm{VF}(M)$ is 
    an algebraically closed valued field with valuation $\mathrm{val}$,
   value group $\Gamma(M)$ and residue field $k(M)$.  
    It follows that
     the value group $\Gamma(M)$ is a non-trivial dense, linear ordered abelian group,  
    the residue field $k(M)$ is an algebraically closed field, the map $\mathrm{val} \colon \mathrm{VF}(M)^* \to \Gamma(M)$ is 
    a surjective homomorphism that satisfies the strong triangle inequality i.e.
     if $x,y \in \mathrm{VF}(M)^*$ then
    $$\mathrm{val}(x + y) \geq \mathrm{min}\{\mathrm{val}(x),\mathrm{val}(y)\}.$$
     Lastly, the function $\mathrm{Res}$ maps $(x,y) \in \mathrm{VF}^2$ to 
    the residue in $k(M)$ of $xy^{-1}$ if $\mathrm{val}(xy^{-1}) \geq 0$ and $0$ otherwise, such that 
    $\mathrm{Res}$ is surjective, $\mathrm{Res}(\_ ,1)$ is a homomorphism etc.}
        \end{defi} 
        
       \noindent \emph{Notation} : In this section, let $K$ be a model of ACVF.
       We often abuse notation and use $K$ itself to denote the valued field $\mathrm{VF}(K)$ and 
       $k$ to denote the residue field $k(K)$.
 \\
 
  \begin{rem} 
  \begin{enumerate} 
  \item 
  \emph{A classical result of A. Robinson states that the completions of ACVF are the theories $\mathrm{ACVF}_{p,q}$
  where the residue field has characteristic $q$, the valued field has characteristic $p$ 
  and either $p = 0$ or $p \neq 0$ and $q = p$. Complete theories are defined in \cite[Definition 2.2.1]{M02}.}
   \item \emph{An important point to note is that ACVF admits elimination of quantifiers  
  in the language $\mathcal{L}_{k,\Gamma}$. Quantifier elimination is defined in \cite[Definition 3.1.1]{M02}.}
  \end{enumerate}
  \end{rem}
  Although the language $\mathcal{L}_{k,\Gamma}$ describes ACVF, it does not eliminate imaginaries.
  To rectify this problem we expand the language $\mathcal{L}_{k,\Gamma}$ by adding 
  certain sorts which we refer to as 
 \emph{geometric} sorts. We provide more details below. 
  
  \subsection{Definable sets}   
                 
                              We introduce the notion of a definable set as in \cite[Section 2.1]{HL} which is 
                              of a more geometric flavour than that 
              presented in \cite{M02}. Let $\mathcal{L}$ be a multi-sorted language and $T$ be a complete $\mathcal{L}$-theory which admits quantifier elimination. 
              We fix a large saturated model $\mathbb{U}$ of $T$ and assume that any model of $T$ which is of any interest 
              to us will be contained in $\mathbb{U}$. Furthermore, we can assume that if $A \subset \mathbb{U}$ is 
              of cardinality strictly less than the cardinality of $\mathbb{U}$ then any type consisting of formulae defined 
              over $A$ has a realization in $\mathbb{U}$. This technical condition will be of use to us when we 
              discuss the notion of definable types in \S \ref{section : definable types}.
               By a set of parameters, we will mean a \emph{small} subset of $\mathbb{U}$.
               Recall that a subset is said to be small if it is of cardinality strictly smaller than the 
               degree of saturation of $\mathbb{U}$ which we assume to be $\mathrm{card}(\mathbb{U})$.
             
                Let $C \subset \mathbb{U}$ be a parameter set.
                We extend the language $\mathcal{L}$ to $\mathcal{L}_C$ by adding constant 
                symbols corresponding to the 
                 elements of $C$. We can expand the theory $T$ to $T_C$ whose models are 
                 those models of $T$ that contain the set $C$.          
                 An $\mathcal{L}_C$-formula $\phi$ can be used to define a
                 functor $Z_\phi$ from the category whose objects are
                 models of $T_C$ 
                 and morphisms
                      are elementary embeddings to the category of sets. 
                      Suppose that the formula $\phi$ involves the 
                      variables $x_1,\ldots,x_m$ where for every $i \in \{1,\ldots,m\}$, $x_i$ is specific to a sort $S_i$
                      respectively. 
                      Then
                      given an $\mathfrak{L}_C$-model ${M}$ of T i.e. a model of T that contains $C$, we set 
                     \begin{align*}
                        Z_{\phi}({M}) = \{\bar{a} \in S_1(M) \times \ldots \times S_m(M) | {M} \models \phi(\bar{a}) \} .  
                     \end{align*}  
                           Clearly, $Z_{\phi}$ is well defined. 
                  
                \begin{defi}
                \begin{itemize}
                \item                  \emph{
                      A} $C$-definable set  $Z$
                     \emph{is a functor from the category whose objects are models of 
                     $T_C$ and morphisms are elementary embeddings to the category of sets such that there exists an $\mathfrak{L}_C$-formula 
                     $\phi$ and $Z = Z_{\phi}$. } 
                \item \emph{Let $X$ and $Y$ be $C$-definable sets. A morphism $f \colon X \to Y$ is said to be \emph{$C$-definable} if 
   its graph is a $C$-definable set. }
         \end{itemize}
                \end{defi}  
                The definable set $Z_{\phi}$ is fully determined by evaluating it on $\mathbb{U}$ i.e. by the set 
                $Z_\phi(\mathbb{U})$. 
                       \begin{es} 
      \emph{Consider the following examples of three different classes of definable sets in ACVF using 
      the language $\mathcal{L}_{k,\Gamma}$.
       Observe that these objects 
        appear naturally in the study of Berkovich spaces, tropical geometry and algebraic geometry. 
      \begin{enumerate}
      \item  Let $K$ be a model of ACVF and $A \subset K$ be a set of parameters.
       Let $n \in \mathbb{N}$. 
       A semi-algebraic subset of $K^n$ is a finite boolean combination of sets of the form 
       $$\{x \in K^n | \mathrm{val}(f(x)) \geq \mathrm{val}(g(x))\}$$
       where $f,g$ are polynomials with coefficients in $A$. One verifies that semi-algebraic sets 
       extend naturally to define $A$-definable sets in ACVF. 
      \item Let $G = \Gamma(K)$ where $K$ is as above. A $G$-rational polyhedron in $G^n$
      is a finite Boolean combination of subsets of the form 
      $$\{(a_1,\ldots,a_n) \in G^n | \sum_i z_ia_i \leq c \}$$
      where $z_i \in \mathbb{Z}$ and $c \in G$. Such objects extend
      naturally to define a $K$-definable subset of $\Gamma^n$ where $\Gamma$ is the value group sort. 
      \item  Any constructible subset of $k^n$ gives a definable set in a natural way. 
       \end{enumerate} }
       \end{es}         
       
     \begin{rem}\label{varieties as definable sets}
      \emph{Let $F$ be a non-trivially valued field and let $V$ be an $F$-variety. 
     The variety $V$ defines in a natural way a functor $Z_V$ as follows. 
     Let $K$ be a model of ACVF that extends $F$. 
     We then set $$Z_V(K) := V \times_F K(K).$$
      We add appropriate sorts to ACVF so that any such functor associated to a variety is a definable set.       
      }
          \end{rem}

%

 \subsection{The language $\mathcal{L}_{\mathcal{G}}$.}
  
  As mentioned earlier, the theory ACVF does not eliminate imaginaries in the language $\mathcal{L}_{k,\Gamma}$.
   In this section, we briefly introduce an extension $\mathcal{L}_{\mathcal{G}}$ of
   $\mathcal{L}_{k,\Gamma}$   
  within which ACVF eliminates imaginaries. 
  In fact, ACVF also admits quantifier elimination in the language $\mathcal{L}_{\mathcal{G}}$. 
    The Hrushovski-Loeser spaces are defined in the language
  $\mathcal{L}_{\mathcal{G}}$ and require elimination of imaginaries for some of their fundamental properties,
  for instance pro-definability (cf. \cite[Lemma 2.5.1, Theorem 3.1.1]{HL}). 
  
    \begin{defi}
  \emph{ A theory $T$ is said to} eliminate imaginaries \emph{if 
  for any model $M \models T$, any collection of sorts 
  $S_1,\ldots,S_m$ and any $\emptyset$-definable
   equivalence relation 
  $E$ on $S_1(M) \times \ldots \times S_m(M)$ there exists a definable function $f$ on 
  $S_1(M) \times \ldots \times S_m(M)$ whose codomain is a product of sorts and is such that 
  $aEb$ if and only if $f(a) = f(b)$.}
  \end{defi}
  
 Suppose a complete theory $T$ in a language $\mathcal{L}$ does not eliminate imaginaries. We can then extend 
 $T$ to a complete theory $T^{eq}$ over a language $\mathcal{L}^{eq}$ so that $T^{eq}$ eliminates imaginaries. 
 Indeed, for every $\emptyset$-definable equivalence relation $E$ on a product of sorts 
 $S_1 \times \ldots \times S_m$, we add a 
 sort to $\mathcal{L}$
 corresponding to the quotient $(S_1 \times \ldots \times S_m)/E$ and a function symbol 
 representing the map 
 $\overline{a} \mapsto \overline{a}/E$.
 This gives the language $\mathcal{L}^{eq}$. 
 Every model $M$ of $T$ extends canonically to a model 
 $M^{eq}$ of $T^{eq}$. 
 The new sorts that were added to $\mathcal{L}$ are referred to as the imaginary sorts and 
 their elements are called imaginaries. 
 
 Given a definable set $X$ in $T$, we can associate to it an element in $\mathbb{U}^{eq}$ called its \emph{code} as follows.  
 Suppose $$X(\mathbb{U}) = \{{x} \in \mathbb{U}^n | \mathbb{U} \models \phi({x},a)\}$$ where $x$ and $a$ in the definition 
 are tuples. 
 We define an equivalence 
 relation by setting $y_1Ey_2$ if $\forall x (\phi(x,y_1) \leftrightarrow \phi(x,y_2))$. 
 The element $a/E$ belongs to $\mathbb{U}^{eq}$ and we refer to it as the code of $X$.
 
 The language $\mathcal{L}_{\mathcal{G}}$ is obtained from 
 $\mathcal{L}_{k,\Gamma}$ by adjoining to it the \emph{geometric sorts}
 $S_n$ and $T_n$ for $n \geq 1$.  
  The sort $S_n$ is the collection of codes for all free rank $R$-submodules of 
  $K^n$ where $R$ is the valuation ring given by 
  $\{x \in K |\mathrm{val}(x) \geq 0\}$. 
  Given $s \in S_n$ for some $n$, let $\Lambda(s)$ denote the corresponding free
  rank $n$ $R$-submodule of $K^n$. 
  If $\mathcal{M}$ denotes the maximal ideal of $R$ then
  let $T_n$ be the set of codes for the 
  elements in $\bigcup_{s \in S_n} \Lambda(s)/\mathcal{M}\Lambda(s)$.   

\subsection{Stably dominated types} 
       
       Let $F$ be a valued field and let $V$ be a $F$-variety.
        The Hrushovski-Loeser space $\widehat{V}$
       associated to $V$ is the space of stably dominated types that concentrate on $V$. 
       We begin with a discussion of the notion of a 
       type followed by that of a 
       definable 
       type which is central to this story. We closely follow the treatment in 
       \S 2.3 of \cite{HL}. 
       
       \subsubsection{Types} \label{types}
       Let $\mathcal{L}$ be a language and $T$ be a complete $\mathcal{L}$-theory. 
       If $z$ is a set of variables, we use $\mathcal{F}_z$ to denote the set of $\mathcal{L}$-formulae 
       with variables in $z$ up to equivalence in the theory $T$. 
       
       \begin{defi}
          An $n$-type $p = p(x_1,\ldots,x_n)$ \emph{is a subset of $\mathcal{F}_{\{x_1,\ldots,x_n\}}$ such that
          $p(x_1,\ldots,x_n)$ is satisfiable in $T$ i.e.
           there exists a model $M$ of 
          $T$ and $(a_1,\ldots,a_n) \in M^n$ such that 
          for every $f \in p$, 
          $M \models f(a_1,\ldots,a_n)$. 
          Furthermore, we say that} the type $p$ is complete \emph{if 
          for every $\phi \in \mathcal{F}_{\{x_1,\ldots,x_n\}}$ either 
          $\phi \in p$ or $\neg \phi \in p$.}    
       \end{defi}  
    
    \begin{rem} \label{realizations} 
    \begin{enumerate}
    \item \emph{Let $p$ be a complete $n$-type. Let $M$ be a model of $T$ such that there exists $a := (a_1,\ldots,a_n)$ 
    that satisfies $p$. Since $p$ is complete, observe that} 
      $$p = \{f \in \mathcal{F}_{\{x_1,\ldots,x_n\}}|M \models f(a)\}.$$
    \emph{In this case, we say that $a$ \emph{realizes} the type $p$. In general, given an $n$-tuple $\alpha \in \mathbb{U}^n$, we 
    write $\mathrm{tp}(\alpha)$ to denote the type
    generated by $\alpha$ i.e.}
      $$\mathrm{tp}(\alpha) := \{f \in \mathcal{F}_{\{x_1,\ldots,x_n\}} | \mathbb{U} \models f(\alpha)\}.$$  
    \item \emph{The definition above concerns types consisting of formulae without parameters. 
   However, it is natural and necessary to work with types defined over a set of parameters. Suppose $A \subset \mathbb{U}$ is a small set of parameters
     then a complete $n$-type defined over $A$ will be a complete $n$-type in the language $\mathcal{L}_A$. 
     Using this formalism and the notion of realizations, one sees that types provide us a tool with which to probe the first order properties 
     of elements in elementary extensions.}
    \end{enumerate}
        \end{rem}
         
    \begin{es} \label{types in ACF}
      \emph{Let $\mathcal{L}_{r} := \{+,-,.,0,1\}$ denote the language of rings where 
      $+,-,.$ are the binary function symbols and $0,1$ are the constant symbols.
      Let $q$ be a prime number or zero. 
      Let $\mathrm{ACF}_q$ denote the theory of algebraically closed fields of characteristic $q$ i.e. the $\mathcal{L}_r$-theory
      whose models are algebraically closed fields of characteristic $q$.
      In this context, the complete $n$-types in $\mathrm{ACF}_q$ take on a recognizable form from algebraic geometry. 
      } 
      
      \emph{Let $L$ be a model of $\mathrm{ACF}_q$. 
      Let $L' \subset L$ be a subfield.  
      Let $p = p(x_1,\ldots,x_n)$ be a type defined over $L'$. 
      We can associate to $p$ an ideal $I_p \subset L'[x_1,\ldots,x_n]$ 
      by setting 
      $$I_p := \{f \in L'[x_1,\ldots,x_n] |  f(x_1,\ldots,x_n) = 0 \in p\}.$$
      It can be shown that the association 
      $p \mapsto I_p$ defines a bijection between the 
      set of complete $n$-types defined over $F$ and the prime ideals 
      of $L'[x_1,\ldots,x_n]$ (cf. \cite[Example 4.1.14]{M02}).
      Given an ideal $I \subset L'[x_1,\ldots,x_n]$, let
        $Z(I) \subset \mathrm{Spec}(L'[x_1,\ldots,x_n])$ denote its vanishing locus. 
        We say that $p$ is the}  {generic type} \emph{of the irreducible closed subvariety 
        $Z(I_p)$.}

    \end{es}  
    \subsubsection{Definable types}   \label{section : definable types}
    As in the previous section, we work in the complete $\mathcal{L}$-theory $T$.  
     Let $A \subset \mathbb{U}$ be a small set of parameters. Let $z$ be a set of variables and 
     $\mathcal{F}^A_z$ denote the set of $\mathcal{L}_A$ formulae with free variables in $z$ upto equivalence in $T_A$.  
     Let $p$ be a complete $n$-type defined over $A$. 
     Since for every $\phi \in \mathcal{F}^A_z$, either $\phi \in p$ or 
     $\phi \notin p$, one may think of  
     $p$ as a Boolean retraction
     from  $\mathcal{F}^A_z$ to the two element Boolean algebra.  
    Equivalently,
     if $U$ is an $A$-definable set whose formula belongs to $p$, then one can 
     see $p$ as a uniform decision to include or exclude 
    $A$-definable subsets $V$ of $U$ according to whether $a \in V(\mathbb{U})$ or $a \notin V(\mathbb{U})$ where 
    $a$ is a realization of $p$.
     Note however that these decisions are restricted to 
     only those subsets of $U$ that are $A$-definable. To broaden the scope of this definition 
     to express decisions on subsets which are not necessarily $A$-definable, we introduce the notion of a
     definable type. Before doing so, we discuss an example we hope will be illuminating. 
     
     \begin{es}
      \emph{As in Example \ref{types in ACF}, let $q$ be a prime number or zero. Let $\mathcal{L}_r$ be the language of rings and $\mathrm{ACF}_q$ denote the theory of 
      algebraically closed fields of characteristic $q$. 
      Let $L$ be a field of characteristic $q$. 
      In Example \ref{types in ACF}, we showed that there exists a bijection between 
      the complete $n$-types defined over $L$ and the irreducible sub-varieties 
      of $\mathbb{A}^n_L$.
      Let $p$ denote the complete $1$-type defined over $L$ that corresponds to $\mathbb{A}^1_L$. 
      By this we mean that $Z(I_p) = \mathbb{A}^1_L$ where $I_p \subset L[x]$ is 
      the ideal associated to $p$ and defined in Example \ref{types in ACF}.}
      
            \emph{Let $L'$ be an extension of $L$ and $p'$ denote the 
      complete $1$-type corresponding to $\mathbb{A}^1_{L'}$. 
      One sees that if one were to restrict $p'$ to $L$ i.e. 
      consider only those formulae in $p'$ with parameters in $L$ 
      then we get $p$. 
      In other words, 
      $$p'_{|L} = p.$$
      The geometric object $\mathbb{A}^1$ can be defined over any field and a type cannot fully express this flexibility. 
      We see that the correct notion is that of a definable type which 
      provides us with a compatible family of types, each element of which 
      is associated to a 
      model of ACF. 
      } 
      \end{es}
      
      \begin{defi} 
       \emph{Let $x = \{x_1,\ldots,x_m\}$. An} $\emptyset$-definable type \emph{is a function 
       $$d_p x \colon \mathcal{F}_{x,y_1,\ldots,} \to \mathcal{F}_{y_1,\ldots,}$$ 
       such that for any $y = \{y_1,\ldots,y_n\}$, 
       $d_p x$ restricts to a Boolean retraction 
       $\mathcal{F}_{x,y} \to \mathcal{F}_y$. 
       If $A \subset \mathbb{U}$ is a set of parameters
       then an $A$-definable type is a $\emptyset$-definable type in the theory 
       $T_A$.}       
      \end{defi} 

  \begin{rem} 
  \emph{Let $A \subset \mathbb{U}$ be a small subset and let $p$ be an $A$-definable type.}
 \begin{enumerate}
 \item    \emph{The definable type $p$ provides us with a compatible family of types. Indeed, if $M$ is a model 
 of $T$ that contains $A$, then $p_{|M}$ is
  the type defined over $M$ consisting of those formulae 
 $\phi(x,b_1,\ldots,b_r)$ such that $M \models d_p x(\phi)(b_1,\ldots,b_r)$.}
 \item \emph{Let $X$ be a definable set. We say that $p$ \emph{concentrates} on $X$ if all the realizations 
 of the type $p_{|\mathbb{U}}$ belong to $X$.} 
 \item \emph{Let $X$ and $Y$ be $A$-definable sets and $f \colon X \to Y$ be an $A$-definable map. 
 Let $S^A_{def,X}$ denote the set of $A$-definable types on $X$. 
 The map $f$ induces a map 
 $f_* \colon S^A_{def,X} \to S^A_{def,Y}$ such that $d_{f_*(p)}y (\phi(y,z)) = d_p x (\phi(f(x),z))$.} 
 \end{enumerate}
    \end{rem}    
    
 \subsubsection{Stably dominated types} 
 
     The definition of a stably dominated type for a general theory $T$ is slightly involved and 
     not necessary for our purposes. We may hence
        restrict our attention to working within the theory ACVF in the language $\mathcal{L}_{\mathcal{G}}$. 
     
  \begin{defi} \label{almost orthogonal definition} 
   \emph{Let $A \subset \mathbb{U}$ be a set of parameters. A type $p$ over $A$ is said to be} almost orthogonal to $\Gamma$ 
   \emph{to $\Gamma$ if for any realization $a$ of $p$, $\Gamma(A(a)) = \Gamma(A)$ where 
   $A(a)$ denotes the definable closure of the set $A \cup \{a\}$ and $\Gamma(A) := \Gamma(\mathbb{U}) \cap dcl(A)$.
    An $A$-definable type 
   $p$ is said to be orthogonal to $\Gamma$ if for every structure $B$ that contains $A$, 
   the type $p_{|B}$ is almost orthogonal to $\Gamma$.}
    \end{defi}    
     
     A stably dominated type is a definable type which in some sense
     does not enlarge the value group sort. More precisely, we mean the following.      
          
    \begin{prop}\cite[Proposition 2.9.1]{HL} 
     In $\mathrm{ACVF}$, an $A$-definable type $p$ is \emph{stably dominated} if and only if it is orthogonal to $\Gamma$. 
     \end{prop} 
       
     \begin{rem} 
     \emph{Let $A \subset \mathbb{U}$ be a small set of parameters and
    let $X$ and $Y$ be $A$-definable sets.
    Let $g \colon X \to Y$ be an $A$-definable morphism. 
    Recall that we have a map 
    $g_* \colon S^A_{def,X} \to S^A_{def,Y}$. One checks that 
    the map $g_*$ restricts to a well defined function from the set of $A$-definable 
    stably dominated types that concentrate on $X$ to the set of $A$-definable
     stably dominated types that 
    concentrate on $Y$.} 
    
     \emph{Let $f \colon X \to \Gamma$ be an $A$-definable function. 
     Suppose $p$ is a stably dominated type that 
    concentrates on $X$ and is defined over $A$. 
    One verifies from the definition that $f_*(p)$ concentrates at a point i.e.
    if $M$ is a model of ACVF then $(f_*(p))_{|M}$ contains the formula 
    $x = a$ for some $a \in \Gamma(A)$.} 
     \end{rem}   
       
       Stably dominated types in ACVF are controlled by the residue field 
       which is the \emph{stable} part of the theory. A precise formulation of this can be found in 
       \cite[\S 2.2]{hashrushmac}. 
       For an introduction to stability theory see the chapter by M. Ziegler in 
       \cite{P}.
       
   \begin{es} \label{generic type of ball example} 
   \emph{Let $a \in \mathrm{VF}(\mathbb{U})$ and $\alpha \in \Gamma(\mathbb{U})$. 
      Let $B(a,\alpha)$ denote the definable set 
      such that if $M$ is a model of ACVF that contains $\{a,\alpha\}$
      then
           $$B(a,\alpha)(M) := \{x \in M | \mathrm{val}(x-a) \geq \alpha\}.$$ 
     In other words, $B(a,\alpha)$ is the closed ball around $a$ of radius $\alpha$.}   
    
     \emph{We can associate to $B(a,\alpha)$ a stably dominated type called 
     its generic type which we denote $p_{B(a,\alpha)}$.
     The definable type $p_{B(a,\alpha)}$ is determined by a definable type concentrated
      on a geometric object over the residue field sort in the following sense. 
     Let us suppose $a = 0$ and $\alpha = 0$ which corresponds to the closed unit ball.  
     Let $\mathrm{red} \colon B(0,0) \to k$ be the reduction map i.e. 
     $\mathrm{red}(x) := \mathrm{Res}(x,1)$. 
     Let $L$ be a model of ACVF.
     We then have that 
     $b$ is a realization of $(p_{B(a,\alpha)})_{|L}$ if and only if 
     $\mathrm{red}(b)$ is a realization of the generic type of 
     $\mathbb{A}^1_{k(L)}$.}
   \end{es}           
        
\subsection{The Hrushovski-Loeser space} \label{HL spaces} 
  
      We define the Hrushovski-Loeser space and provide a fundamental example of such an object. 

\begin{defi} (The space $\widehat{X}$) \label{hat space}
\emph{ Let $C \subset \mathbb{U}$ be a small set of parameters. Let $X$ be a $C$-definable set. We define $\widehat{X}$ to be the functor 
             from the category whose objects are models of $\mathrm{ACVF}$ that contain $C$ and morphisms are elementary
             embeddings
              to the category of sets such that if  
            $M$ is a model that contains $C$ then 
             $\widehat{X}(M)$ is the set of $M$-definable stably dominated types that concentrate on $X$.} 
\end{defi}

Let $F$ be a valued field. Given a quasi-projective $F$-variety $V$,  
we can endow the space $\widehat{V}$ with a topology in a natural fashion which is
analogous to the topology of the Berkovich space $V^{\mathrm{an}}$. 
We first define the topology on the space $\widehat{V}(\mathbb{U})$ and then restrict this 
definition to $\widehat{V}(M)$ where $M$ is any model of $\widehat{V}$. 

\begin{defi} 
\begin{enumerate} 
\item \emph{ A pre-basic open set of $\widehat{V}(\mathbb{U})$ is of the form 
$$\{p \in \widehat{O}|\mathrm{val}(f)_*(p) \in W\}$$ where $O \subset V(\mathbb{U})$ is a Zariski open subset of $V$
with parameters in $\mathbb{U}$,
 $f$ a regular function on $O$ defined over $\mathbb{U}$ and 
$W$ an open subset of $\Gamma_\infty(\mathbb{U})$. Here $\mathrm{val}(f)$ denotes the definable map given by
$O(\mathbb{U}) \to \mathrm{VF}(\mathbb{U}) \to \Gamma_\infty(\mathbb{U})$. (As $\mathrm{val}(f)_*(p)$ is a stably dominated type on 
$\Gamma_\infty$, it is constant.)}
\item \emph{ The set $\widehat{V}(\mathbb{U}) \times \Gamma_\infty^l(\mathbb{U})$ is given the product topology and
if $X(\mathbb{U}) \subset \widehat{V}(\mathbb{U}) \times \Gamma_\infty^l(\mathbb{U})$ then 
 we let $X(\mathbb{U})$ have the subspace topology.}
 \end{enumerate} 
 \end{defi}
 
 \begin{rem} \label{functoriality of HL spaces} 
\emph{Let $C \subset \mathbb{U}$ be a small set of parameters.
 Let $f \colon X \to Y$ be a $C$-definable map of $C$-definable sets. 
 The map $f_* \colon S^C_{def,X} \to S^C_{def,Y}$ restricts to a well defined map 
 $\widehat{f} \colon \widehat{X} \to \widehat{Y}$. Furthermore, 
 if $X(\mathbb{U}) \subset {V}(\mathbb{U})$ 
 for some variety $V$,
 $Y(\mathbb{U}) \subset {V'}(\mathbb{U})$ 
 for some variety $V'$ and $r \in \mathbb{N}$ and $f$ is the restriction of a regular map then 
 the map $\widehat{f}$ is continuous.}
\end{rem} 
  
  Let $X$ be an $F$-definable subset
 contained in $\mathrm{VF}^n \times \Gamma_\infty^l$ for some $n,l \in \mathbb{N}$.
 If $M$ is a structure that contains $F$ then we endow $\widehat{X}(M)$ with the topology whose open 
 sets are the $M$-definable open sets of $\widehat{X}$.
 Note that it does not make sense to provide $\widehat{X}(M)$  
 with the subspace topology via the inclusion 
 $\widehat{X}(M) \subset \widehat{X}(\mathbb{U})$ since the resulting space 
 would be discrete. 
 
 Note that if $V$ is an algebraic variety defined over $F$ then the space 
 $\widehat{V}(M)$ is Hausdorff for every model $M$ of ACVF that extends $F$.

 \begin{es} \label{the affine hat space} 
 \emph{Observe that the affine line defines an $\emptyset$-definable set which 
 we denote $\mathbb{A}^1$ (cf. Remark \ref{varieties as definable sets}). 
 As a set $\widehat{\mathbb{A}^1}(\mathbb{U})$ consists of the generic types of balls 
 $p_{B(a,\alpha)}$ (cf. Example \ref{generic type of ball example}) 
 where $a \in \mathbb{U}$ and $\alpha \in \Gamma(\mathbb{U})$ as well 
 as the \emph{simple} points $b \in \mathbb{U}$.}
 
 \emph{Let $M$ be a model of ACVF. The topology of the space 
 $\widehat{\mathbb{A}^1}(M)$ 
 can be described as follows. 
 Let $m \in M$. 
 By an open ball around $m$ of radius $\alpha$, we mean 
 the set $O(m,\alpha) := \{x \in \mathbb{U}|\mathrm{val}(x-m) > \alpha\}$. 
 A fundamental system of open neighbourhoods 
 in $\widehat{\mathbb{A}^1}(M)$ 
 of $m$ is given by the family $\{\widehat{O(m,\alpha)}(M)\}_{\alpha \in \Gamma(M)}$.
 Let $B := B(a,\alpha)$ be a closed ball with $\alpha \in \Gamma(M)$. 
 A fundamental system of open neighbourhoods of $p_B$ is given by sets of the form 
 $\widehat{O}$ where $O$ is a $M$-definable open ball from which finitely many 
 $M$-definable closed balls are removed and 
 $\widehat{O}$ contains the point $p_B$. This description can be deduced from Holly's theorem \cite{holly95} which effectively says 
 that a definable subset of $\mathrm{VF}$
 admits a swiss cheese decomposition i.e. it is the  
  disjoint union of balls from which finitely many sub-balls are removed. 
 By a ball in $\mathrm{VF}$, we mean a set of the form 
 $\{x \in \mathrm{VF} | \mathrm{val}(x - a) \mbox{ } \kappa \mbox{ } \alpha\}$ where $\kappa \in \{>,\geq\}$ and $\alpha \in \Gamma_\infty$.
 }
 
 \emph{After fixing coordinates, we have the following equality of sets
 $$\widehat{\mathbb{P}^1}(\mathbb{U}) = \widehat{\mathbb{A}^1}(\mathbb{U}) \cup \{ \infty\}$$ where 
 $\infty$ is an $\emptyset$-definable point.
 The space $\widehat{\mathbb{A}^1}(M)$ is an open subspace of $\widehat{\mathbb{P}^1}(M)$. 
 The open neighbourhoods of $\infty$ in $\widehat{\mathbb{P}^1}(M)$ consist of the complements in 
 $\widehat{\mathbb{P}^1}(M)$ of the spaces 
 $\widehat{B(0,\alpha)}(M)$ where $\alpha \in \Gamma(M)$.  
 }
  \end{es}

 \begin{rem}
 \emph{Let $C \subset \mathbb{U}$ be a set of parameters. Let $\mathrm{Def}_C$ be the category whose objects are those
 sets which are definable with parameters in $C$ and morphisms are $C$-definable maps.  
 By a $C$-pro-definable set, we mean a pro-object in $\mathrm{Def}_C$ indexed by a small partially ordered set. 
 Let $X$ be a $C$-definable set. By \cite[Theorem 3.1.1]{HL}, there exists a $C$-pro-definable set $E$ such that 
 for every model $M$ of ACVF that contains $C$, we have a canonical identification 
 $$\widehat{X}(M) = E(M).$$}
  \end{rem} 
 
    The main theorem  -  \cite[Theorem 11.1.1]{HL} proved by Hrushovski-Loeser in \cite{HL} implies in particular 
       that the homotopy type of the Hrushovski-Loeser space associated to a quasi-projective variety over a valued field
        is determined by a relatively simple object - a definable subset of $\Gamma_\infty^n$
       for some $n \in \mathbb{N}$.  
       The precise statement is as in Theorem \ref{main theorem}. 
       
  \begin{rem} 
    \emph{Note that when using the machinery of Hrushovski-Loeser to construct \emph{pro-definable} deformation retractions
    $H \colon I \times \widehat{V} \to \widehat{V}$ 
     onto a $\Gamma$-internal set, where $V$ is a quasi-projective 
    variety of dimension greater than 1, we can no longer suppose that the interval 
   $I$
    is $[0,\infty]$. 
    It is usually a} generalized interval \emph{which consists of glueing 
    end to end copies of $[0,\infty]$. 
     A more in depth, explanation can be found in \cite[\S 3.9]{HL}.}  
  \end{rem}

\begin{thm} \label{main theorem} Let $V$ be a quasi-projective variety over a valued field $F$ and let $X$ be a definable subset of $V \times \Gamma^l_\infty$
 over some base set $A \subset \mathrm{VF} \cup \Gamma$, with $F = \mathrm{VF}(A)$.
  Then there exists an $A$-definable deformation retraction 
  $$H \colon I \times \widehat{X} \to \widehat{X}$$
   with image an iso-definable subset $\Upsilon$ which is definably homeomorphic to a definable subset of $\Gamma_\infty^w$, for some finite $A$-definable set $w$.
One can furthermore require the following additional properties for $H$ to hold simultaneously.
\begin{enumerate} 
\item Given finitely many $A$-definable functions $\xi_i \colon X \to \Gamma_\infty$ with canonical 
extension $\widehat{\xi_i} \colon \widehat{X} \to \Gamma_\infty$, one can choose $H$ to respect the $\widehat{\xi_i}$, 
i.e. to satisfy $\widehat{\xi_i}(H(t,x)) = \widehat{\xi_i}(x)$ for all $(t, x) \in I \times \widehat{X}$. In particular, finitely many definable subsets $U$ of $X$
 can be preserved, in the sense that $H$ restricts to a homotopy of $\widehat{U}$.
\item Assume given, in addition, a finite algebraic group $G$ acting on $V$ and leaving $X$ globally invariant. Then the retraction 
$H$ can be chosen to be equivariant with respect to the $G$-action.
\item Assume $l = 0$. The homotopy $H$ is Zariski generalizing, i.e. for any Zariski open subset $U$ of $V$, $\widehat{U} \cap X$ is invariant under $H$.
\item The homotopy $H$ is such that for every $x \in \widehat{X}$, $H(e_I,H(t, x)) = H(e_I,x)$ for every $t$ and $x$.
\item One has $H(e_I,X) = \Upsilon$, i.e. $\Upsilon$ is the image of the simple points.
\item Assume $l = 0$ and $X = V$. Given a
 finite number of closed irreducible subvarieties $W_i$ of V , one can demand that $\Upsilon \cap \widehat{W_i}$ has
  pure dimension $\mathrm{dim}(W_i)$.
  \end{enumerate} 
\end{thm}

 Observe that the theorem not only describes precisely the homotopy type of the space $\widehat{X}$ but also allows us to construct 
 pro-definable deformation retractions with considerable flexibility. It is this fact that we will exploit in \S 5. 
 Note that all deformation retractions considered and constructed in this paper satisfy assertions (3), (4) and (5) above. 


\subsubsection{Canonical Extensions} \label{canonical extensions}
            
                    Let
                    $C \subset \mathbb{U}$ be a small set of parameters and 
                     let $X$ be a $C$-definable set. 
            Let $Y$ be a $C$-definable subset 
             and 
            $f \colon X \to \widehat{Y}$ be a $C$-definable map. 
            In this situation, we can extend the map $f$ to a well defined 
            map $\widehat{f} \colon \widehat{X} \to \widehat{Y}$ which 
            we call the \emph{canonical extension of $f$}.
            We do so as follows.  
                    
                    Let $p \in \widehat{X}(M)$
                    where $M$ is a model of ACVF that contains $C$.
                    Let $c \models p_{|M}$ and $d \models f(c)_{|M(c)}$ (cf. Remark \ref{realizations}). 
              By \cite[Proposition 2.6.5]{HL}, the type $\mathrm{tp}(cd|M)$ is stably dominated. 
              It follows that $\mathrm{tp}(d|M)$ is stably dominated as well. We 
              set $\widehat{f}(p) := \mathrm{tp}(d|M)$. The map $\widehat{f}$ is well defined and
              can be shown to be 
              pro-$C$-definable.
              
                    When $X \subset \mathbb{P}^m \times \Gamma_\infty^n$,
                     it is natural to ask what hypothesis must be placed on the definable map $f$ to ensure that 
                     the induced map  
                     $\widehat{f}$ is continuous. To this end, we introduce the $v$ and $g$-topologies on $X$. 
                     The $v$-topology is
                     the topology induced by the valuation on the ground field and is hence a well defined topology.
                     The $g$-topology on the other hand is a Grothendieck topology and if $U \subset X$ is both $v$ and $g$
                     open then $\widehat{U}$ is open in $\widehat{X}$.

            \begin{defi} \label{v+g topology} ($v$ and $g$ - open sets)
             \emph{Let $V$ be an algebraic variety defined over a valued field $F$. A definable set $U \subset V$ is v-open 
                  if it is open for the valuation topology on $V$. A definable set $G$ is g-open if it 
                  is a positive Boolean combination of Zariski closed and open subsets and 
                  sets of the form 
                  $\{u : \mathrm{val}(f)(u) < \mathrm{val}(g)(u) \}$ where $f,g$ are regular functions on some Zariski open set.
                  More generally, if $U$ is a definable subset of the variety $V$ then a set $W \subset U$ is 
                  v-open (g-open) if it is of the form $U \cap O$ where $O$ is v-open (g-open). If 
                 $X \subset V \times \Gamma_{\infty}^n$ is a definable set then 
                 $X$ is v-open (g-open) if its pullback via $\mathrm{id} \times \mathrm{val}$
                  to $V \times \mathbb{A}^n$ is v-open (g-open). }  
                  \end{defi}         
      
      \begin{rem} 
          \emph{
          Observe that the $v$-topology on $\Gamma$ is discrete while the neighborhoods of $\infty$ are the same as those 
          defined by the order topology. The $g$-topology of $\Gamma_\infty$ when restricted to $\Gamma$ coincides with the order topology while
          the point $\infty$ is isolated. It follows that the $v + g$ topology which is the topology generated by the class of sets which are 
          both $v$ and $g$-open, induces the order topology on $\Gamma_{\infty}$.
         In general the $v,g$ and $v + g$-open sets are definable. 
          Observe that in the case of a variety $V$ over a valued field $F$, 
          the collection of $v$-open sets definable over $F$
         generate the valuation topology on $V$. The $g$-topology however does not necessarily generate a topology.}
         \end{rem}    
            
     \begin{defi} (v-continuity and g-continuity) \emph{ Let $V$ be an algebraic variety over 
     a valued field $F$ or a definable subset of such 
     a variety. A definable function $h \colon V \to \Gamma_\infty$ is called v-continuous (resp. g-continuous) if the pullback of 
     any v-open (resp. g-open) set is v-open (resp. g-open). A function $h \colon V \to W$ with $W$ an affine $F$-variety is 
     called v-continuous (resp. g-continuous) if, for any regular function $f \colon W \to \mathbb{A}^1$ , $\mathrm{val} \circ f \circ h$ is v-continuous (resp. g-continuous). }
          \end{defi}        
          
           Let $V$ be an algebraic variety and $W \subset \mathbb{P}^m \times \Gamma_\infty^n$ be a definable set. We endow 
  $\widehat{W}$ with the subspace topology from $\widehat{\mathbb{P}^m} \times \Gamma_\infty^n$. 
  Let $f \colon V \to \widehat{W}$ be a well defined pro-definable function such that for every open subset $O \subset \widehat{W}$,
    $f^{-1}(O)$ is $v+g$-open. In this case, we say that the map $f$ is $v+g$-continuous. 
          By \cite[Lemma 3.8.2]{HL}, $f$ induces a continuous pro-definable map 
       $\widehat{f} \colon \widehat{V} \to \widehat{W}$. 
%
          
      \subsubsection{Simple points}   Let $V$ be an $A$-definable set. For $x \in V$, the definable 
         type $\mathrm{tp}(x/\mathbb{U})$ which concentrates on the point $x$ is stably dominated. It follows that 
         $\mathrm{tp}(x/\mathbb{U})$ is an element of $\widehat{V}(\mathbb{U})$. We can thus view 
         $V$ as a subset of $\widehat{V}$. This subset of points in $\widehat{V}$ is called the set of \emph{simple points}.    
         
    \begin{lem} (\cite{HL}, Lemma 3.6.1) Let $X$ be a definable subset of $\mathrm{VF}^n$.
 \begin{enumerate}
    \item  The set of simple points of $\widehat{X}$ (which we identify with $X$) is an iso-definable \cite[Definition 2.2.2]{HL}
     and relatively definable \cite[\S 2.2]{HL} dense subset of $\widehat{X}$.
     If $M$ is a model of $\mathrm{ACVF}$ then $X(M)$ is dense in $\widehat{X}(M)$.
    \item    The induced topology on $X$ agrees with the valuation topology on $X$.
\end{enumerate}
\end{lem} 

\begin{rem} \label{unique extensions} 
  \emph{Let $V$ be an algebraic variety and $W \subset \mathbb{P}^m \times \Gamma_\infty^n$. We endow 
  $\widehat{W}$ with the subspace topology from $\widehat{\mathbb{P}^m} \times \Gamma_\infty^n$. 
  Let $f \colon V \to \widehat{W}$ be a $v+g$-continuous map. Since the simple points are dense in $V$, we see that there exists 
  exactly one morphism $\widehat{V} \to \widehat{W}$ that extends $f$.}
\end{rem} 

\section{The Berkovich space $B_{\mathbf{F}}(X)$} \label{model theoretic berkovich space} 
%
  
      We provide a model theoretic reinterpretation of the Berkovich space, one for which a connection 
      with the space of stably dominated types can be easily made. 
      Our presentation follows \cite[Chapter 14]{HL}.
   
          Let $F$ be a real valued field
        and let $\mathbb{R}_{\infty} := \mathbb{R} \cup \{\infty\}$. 
           Let $\mathbf{F}$ denote the structure defined by the pair $(F,\mathbb{R}_{\infty})$. Let 
           $V$ be a quasi-projective $F$-variety. As a set, the Berkovich space $B_{\mathbf{F}}(V)$ is defined as follows. 

\begin{defi}
   \emph{Let $X$ be an $\mathbf{F}$-definable subset of $V \times \Gamma_\infty^l$ for some $l \in \mathbb{N}$. 
   Let} $B_{\mathbf{F}}(X)$ \emph{ be the set of almost orthogonal to $\Gamma$, $\mathbf{F}$-types which concentrate on $X$.
   The notion of an almost orthogonal to $\Gamma$ type was introduced in Definition \ref{almost orthogonal definition}.}
\end{defi} 
 
     Let $f : X \to \Gamma_{\infty}$ be an $\mathbf{F}$-definable function. 
     Observe that if 
     $p \in B_{\mathbf{F}}(X)$ is such that $a \models p$ then $f(a) \in \mathbb{R}_{\infty}$ depends only on 
     the type $p$ i.e. if $a_1 \models p$ and $a_2 \models p$ then $f(a_1) = f(a_2)$.  
    We set $f(p) := f(a)$. Thus we have a well defined function $f \colon B_{\mathbf{F}}(X) \to \mathbb{R}_{\infty}$. 
    
    \begin{defi} (Topology on $B_{\mathbf{F}}(X)$)
      \emph{Let $X$ be an $\mathbf{F}$-definable subset of the $F$-variety $V$. The set $B_{\mathbf{F}}(X)$ is endowed 
      with the topology generated by pre-basic open sets 
      of the form $\{q \in B_{\mathbf{F}}(X \cap U) | \mathrm{val}(f)_*(q) \in W \}$ where 
      $U \subset V$ is an open affine subspace, $f$ is a regular function on $U$ and 
      $W \subset \mathbb{R}_{\infty}$ is an open interval.}
     \end{defi}

 \begin{lem}
    The spaces $V^{\mathrm{an}}$ and $B_{\mathbf{F}}(V)$ are canonically homeomorphic.
   \end{lem} 
   \begin{proof} 
    This is proved in \cite[14.1]{HL}. 
   \end{proof} 

    We relate the space $B_{\mathbf{F}}(V)$ to the space $\widehat{V}$. Let $L$ be an 
    algebraically closed, spherically complete
    valued field which contains $\mathbf{F}$ as a substructure and whose residue field is the algebraic closure of the residue field
    $k(F)$ of $F$ and 
    value group $\Gamma(L)$ is $\mathbb{R}$. Such a field is unique up to isomorphism over 
   the structure $\mathbf{F}$. We fix one such copy and call it $F^{max}$.   
   
   \begin{lem} \label{HL to Berkovich}
        There exists a surjective continuous function $$\pi_{F,V} \colon \widehat{V}(F^{max}) \to B_{\mathbf{F}}(V)$$ such that if 
    $X$ is an $\mathbf{F}$-definable subset of $V$ then $\pi_{F,V}^{-1}(B_\mathbf{F}(X)) = \widehat{X}(F^{max})$.     
   \end{lem} 
   \begin{proof}
        The map is constructed in  \cite[\S 14.1]{HL}. The surjectivity of $\pi_{F,V}$ 
        and the assertion regarding $\pi_{F,V}^{-1}(B_\mathbf{F}(X))$ is proved in \cite[Lemma 14.1.1]{HL}
         and 
        \cite[Proposition 14.1.2]{HL}. Nonetheless, we repeat the construction here since 
        the map $\pi_{F,V}$ figures prominently in \S 6.

       Let $p$ be a stably dominated type defined over $F^{max}$ that concentrates on $V$. Then 
       $p_{|F^{max}}$ is an $F^{max}$-type. Let $\pi_{F,V}(p)$ denote the $\mathbf{F}$-type defined 
       by those formulae with parameters in $\mathbf{F}$ that are contained in $p_{|F^{max}}$.
      Let $a \in V$ be such that $\pi_{F,V}(p) = \mathrm{tp}(a/\mathbf{F})$. 
      We must have that $\Gamma(\mathbf{F}(a)) \subseteq \Gamma(F^{max}(a)) = \Gamma(F^{max}) = \Gamma(\mathbf{F})$.
     It follows that $\pi_{F,V}$ is a well defined function. 
        It can be checked that it is continuous as well.  
        
        We show that $\pi_{F,V}$ is surjective. Let $p \in B_{\mathbf{F}}(V)$ and 
        $a$ be a realization of $p$. 
         By \cite[Theorem 12.18 (ii)]{hashrushmac}, the type 
        $\mathrm{tp}(a|F^{max})$ extends to an $F^{max}$-stably dominated type 
        which concentrates on $V$, thus defining an element of $\widehat{V}(F^{max})$. 
   \end{proof} 

\begin{rem}
   \emph{The purpose of this section was to emphasize the extent to which the Berkovich space and the Hrushovski-Loeser 
   space are closely related.  
   In fact, when working over certain models of ACVF, the two spaces coincide.
   Indeed, using the notation from Lemma \ref{HL to Berkovich},
   we see that if $F = F^{max}$ then by \cite[Lemma 14.1.1]{HL} $\pi$ is a homeomorphism.}
\end{rem}

\section{Required tools}  
     Our goal in this section is to develop the tools required to prove the principal results of Sections 
     \S \ref{finite obstruction} and \S \ref{base is a curve}.

 \subsection{Preliminary simplifications}

          We use the following lemmas to show that 
          in most situations under consideration,
          we may suppose that 
          we have a morphism between projective varieties whose fibres are pure over some open 
           dense
            subset of the base. 
            
            \begin{lem} \label{simplifications}
            Let $V$ be a 
                         quasi-projective $K$-variety. 
   Let $\phi \colon V' \to V$ be a morphism between quasi-projective $K$-varieties whose image is dense. Let $G$ be a finite algebraic group acting on $V'$
   that restricts to a well defined action along the fibres of the morphism $\phi$
   and 
   $\xi_i \colon V'\to \Gamma_{\infty}$ be a finite collection of definable functions.
  There exists projective $K$-varieties $\overline{V}$, $\overline{V'}$ and a finite type surjective 
  morphism 
  $\overline{\phi} \colon \overline{V'} \to \overline{V}$ with the following properties.  
  \begin{enumerate} 
  \item The varieties $\overline{V'}$ and $\overline{V}$ are pure of dimension $\mathrm{dim}(V')$ and $\mathrm{dim}(V)$ respectively. 
  \begin{itemize}
  \item   In the event that $V$ is a smooth connected $K$-curve, we can suppose $\overline{V}$ is the unique smooth projective $K$-curve
  that contains $V$ as a Zariski open dense subset. 
  \item If $V$ is integral and normal then we can take $\overline{V}$ to be integral and normal as well. 
  \end{itemize}
    \item There exists embeddings 
  $i \colon V \hookrightarrow \overline{V}$ and $i' \colon V' \hookrightarrow \overline{V'}$ whose images are locally closed subspaces. 
  \item The morphism $\overline{\phi}$ extends the induced morphism $\phi \colon i'(V') \to i(V)$.
  \item The fibres of the morphism $\overline{\phi}$ are pure and projective over some open dense subset of 
  $\overline{V}$. When $V$ is a smooth connected curve and $\phi$ is flat, the morphism $\overline{\phi}$ is flat over $\overline{V}$ and 
  all fibres of $\overline{\phi}$ are pure. 
  \item The variety $\overline{V'}$ admits an action of the group $G$ 
   that extends the action of $G$ on $V'$. 
   The action of the group $G$ restricts to a well defined morphism on the fibres of 
   $\overline{\phi}$. 
  \item The functions $\xi_i$ extend to definable functions $\xi'_i \colon \overline{V'} \to \Gamma_{\infty}$.   
  \end{enumerate} 
\end{lem} 
\begin{proof} 
As the varieties $V$ and $V'$ are quasi-projective, there exists $n,m \in \mathbb{N}$ such that
we can identify $V$ and $V'$ with locally closed subspaces of $\mathbb{P}^n_K$ and 
$\mathbb{P}^m_K$ respectively. 
Let $a \colon V \hookrightarrow \mathbb{P}^n_K$ and $a' \colon V' \hookrightarrow \mathbb{P}^m_K$ denote the 
respective immersions. 
Let $\overline{V}$ denote a projective variety contained in $\mathbb{P}^n_K$ that is pure, of dimension
$\mathrm{dim}(V)$ and contains 
$V$ (cf. \cite[\S 11.2]{HL}).
In the event that $V$ is a smooth connected $K$-curve, let $\overline{V}$ be the unique smooth projective $K$-curve
  that contains $V$ as a Zariski open dense subset. 
 If $V$ is integral and normal, then we can replace $\overline{V}$ with its normalization. 
 Note that this normalization is a projective variety and contains $V$ as an open dense subset. 
 After increasing $n$ if necesssary, we abuse
  notation and assume that 
  we have a closed immersion
  $\overline{V} \hookrightarrow \mathbb{P}^n_K$. 

  Let $A' := (\mathbb{P}_K^m)^G$. 
The group $G$ acts on the projective variety $A'$ 
in the following fashion. 
Let $\mathbf{x} := (x_h)_{h \in G} \in A'$ and $g \in G$ then we set $g(\mathbf{x}) := (x_{hg})_{h \in G}$.
 Let $b' \colon  V' \hookrightarrow A'$ denote the immersion 
defined by $v \mapsto (a'(hv))_{h \in G}$.
  Observe that $b'$
   is a $G$-equivariant embedding.
  

  We have an embedding $c \colon V' \to A' \times \overline{V}$ given by $c(v) := (b'(v),a(\phi(v)))$.
  Let $p_2 \colon A' \times \overline{V} \to \overline{V}$ denote the projection morphism onto the second coordinate.
  Observe that $p_2$ is a $G$-equivariant morphism where $\overline{V}$ is endowed with the trivial action.
   We identify $V'$ with its image 
  in $A' \times \overline{V}$ via the embedding $c$. Likewise, we identify the morphism $\phi$ with the restriction of $p_2$ to $c(V')$.  

Let $S \subset \overline{V}$ denote the set of images of the generic points of $V'$ for the morphism 
$\phi$ and the generic points of $\overline{V}$. 
Given a point $x \in \overline{V}$, 
let $A'_x$ denote the fibre over $x$ for the morphism $p_2$. 
Recall from \cite[Exercise II.3.10]{hart} that $A'_x$ is homeomorphic to 
the subspace $p_2^{-1}(x) \subset A' \times \overline{V}$.
For every point $\eta \in S$, we choose a $G$-equivariant pure 
projective variety $X(\eta) \subset A'_\eta$ that contains $V' \cap A'_{\eta}$ and whose closure 
 in $A' \times \overline{V}$ is of dimension $\mathrm{dim}(V')$ (cf. \cite[\S 11.2]{HL}).
 It follows that 
$\bigcup_{\eta \in S} X(\eta)$ is a $G$-invariant subset of 
$A' \times \overline{V}$.  
Let $\mathrm{cl}({X(\eta)})$ denote the Zariski closure of $X(\eta)$ in $A' \times \overline{V}$ 
and 
we define $\overline{V'} := \bigcup_{\eta \in S} \mathrm{cl}({X(\eta)})$ endowed with the reduced induced closed subscheme structure. 
Since $G$ stabilizes $\bigcup_{\eta \in S} \mathrm{cl}({X(\eta)})$, we see that 
$G$ must stabilize $\overline{V'}$. 
Hence, $\overline{V'}$ is a projective 
variety which is pure of $\mathrm{dim}(V')$ that contains $V'$ and extends the action of the group $G$. 

 Observe that if $V$ is a smooth connected curve and $\phi$ is flat, then $S$ as defined above consists of 
 exactly the generic point of $V$. In this situation, our construction implies in addition that every irreducible component
 of $\overline{V'}$ dominates $\overline{V}$. It follows that the
  morphism 
 $\overline{V'} \to \overline{V}$ is flat. 
The purity of the fibres follows from \cite[Corollary III 9.6]{hart}.

Let us now return to the general situation.
Using \cite[Corollary III.9.6]{hart} and \cite[Tag 052B]{stacks-project}, we deduce that
the fibres of the restriction ${p_2}_{|\overline{V'}} \colon \overline{V'} \to \overline{V}$
are pure and equidimensional over an open dense subset of $\overline{V}$. 
Observe that
the morphism ${p_2}_{|\overline{V'}}$ is $G$-equivariant. 
 We define $\overline{\phi} := {p_2}_{|\overline{V'}}$.  
The functions $\xi_i \colon V' \to \Gamma_{\infty}$ extend to definable functions on 
$\overline{V'}$ by setting $\xi'_i(v) := \xi_i(v)$ for $v \in V'$ and $\infty$ otherwise. 
\end{proof}

     \begin{lem} \label{prove the simpler case}
       Suppose that Theorem \ref{generic cdr} and Proposition \ref{generic relative cdr} are true 
      in the following situation. 
      \begin{enumerate}
      \item Let $G$ be a finite algebraic group acting on projective $K$-varieties $V'$ and $V$
      where the action on $V$ is trivial. 
      \item The variety $V$ is integral and normal. 
      \item We are given a $G$-equivariant morphism $\phi \colon V' \to V$ whose fibres are pure over some 
      dense open subset $U \subset V$. 
      \item We are given a finite family $\{\xi_i \colon V' \to \Gamma_\infty\}_{i \in I}$ of $K$-definable functions.  
      \end{enumerate}
        Then Theorem \ref{generic cdr} and Proposition \ref{generic relative cdr} are true in general. 
     \end{lem}        
         \begin{proof}
          We first consider the case of Theorem \ref{generic cdr}. 
          Let the data be as given in Theorem \ref{generic cdr}. That is to say, 
          let $V$ be a pure quasi-projective $K$-variety and 
          let $\phi \colon V' \to V$ be a morphism between quasi-projective varieties with dense image. 
          Let $G$ be a finite algebraic group acting on $V'$
   which restricts to a well defined action along the fibres of the morphism $\phi$
   and 
   $\xi_i \colon V'\to \Gamma_{\infty}$ be a finite collection of $K$-definable functions.
   
             We deduce without difficulty that we can
              assume $V$ is integral and normal.          
          By Lemma \ref{simplifications}, 
          there exists projective 
          $K$-varieties $\overline{V}$, $\overline{V'}$ and a finite type surjective 
  morphism 
  $\overline{\phi} \colon \overline{V'} \to \overline{V}$ satisfying conditions (1)-(6) of \ref{simplifications}. 
     Observe from our construction in the proof of Lemma \ref{simplifications}
      that the variety $V'$ is not necessarily dense in $\overline{V'}$. Let $V'_1$ denote the closure 
     of $V'$ in $\overline{V'}$. 
     In our situation, since $V'$ is locally closed, it is open in $V'_1$. 
     To the family of definable functions $\{\xi'_i \colon \overline{V'} \to \Gamma_{\infty}\}_i$ which extend the functions $\xi_i$, we add 
     the valuations of the characteristic functions of the closed subvariety $V'_1 \subset \overline{V'}$ which we denote 
     $\epsilon_j$ i.e. the family $\{\epsilon_j\}_j$ is such that $\bigcap_j \epsilon_j^{-1}(\infty) = V'_1$. 
     The hypothesis of the lemma 
     implies the existence of a 
      Zariski dense open subset $W \subset \overline{V}$ such that if  
     $W' := \phi^{-1}(W)$ then there exists 
     compatible deformation pairs $(H',\Upsilon')$ on $\widehat{W'}$ and 
     $(H,\Upsilon)$ on $\widehat{W}$
     satisfying assertions (1)-(6) of Theorem \ref{generic cdr}. 
       This implies in particular that the deformation $H'$ restricts to a  
      deformation  of 
      $\widehat{W' \cap V'_1}$.
      The set $W' \cap V'$ is open in $V'_1$. 
       Theorem \ref{generic cdr} asserts that the deformation  $H'$ is Zariski generalizing.
       Every Zariski open subset of $W' \cap V'_1$ is of the form $O \cap W' \cap V'_1$ where 
      $O$ is a Zariski open subset of $W'$. It follows that the restriction of the deformation  
      $H'$ to $\widehat{W' \cap V'_1}$ is Zariski generalizing as well. This
       implies that $H'$ restricts to a well defined deformation  
      of $\widehat{W' \cap V'}$.
      Recall that the map $V' \to V$ is a well defined morphism of quasi-projective varieties whose image is dense.
      We can hence shrink $W$ so that $W' \cap V' \to W \cap V$ is surjective. 
       It follows
       that $H$ restricts to a  well defined deformation of $\widehat{W \cap V}$. 
      The restrictions of the deformations $H'$ and $H$ to $W' \cap V'$ and 
      $W \cap V$ must be compatible since $H'$ and $H$ are compatible. 
      The proof in the case of Proposition \ref{generic relative cdr} is an easy adaptation of the arguments above.           
         \end{proof}

   \begin{lem} \label{Local factorization}
    Let $f \colon V' \to V$ be a projective morphism
    of $K$-varieties 
     such that the fibres of $f$ are pure of dimension $m$. 
    Let $G$ be a finite algebraic group acting on $V'$ such that the 
    morphism $f$ is $G$-equivariant when $V$ is endowed with the trivial action. 
    For every $v \in V(K)$, there exists a Zariski open neighbourhood $U \subset V$ of $v$ such that 
    the morphism $f \colon f^{-1}(U) \to U$ factors through a finite surjective
    $G$-equivariant morphism
    $p \colon f^{-1}(U) \to \mathbb{P}^m \times U$ over $U$  where 
    the $G$-action on $\mathbb{P}^m \times U$ is taken to be trivial. 
    
         Suppose we are given a horizontal divisor
         \footnote{We say that $D$ is a horizontal divisor if it does not contain any irreducible component of any fibre of the morphism $f$.}
          $D \subset V'_U$. 
         There then exists a 
          $K$-point $z \in \mathbb{P}^m$ 
          such that after shrinking $U$ if necessary but maintaining that it contains $v$,
           we have the following commutative diagram
          
          $$
\begin{tikzcd}[row sep = large, column sep = large]
V'_{1U} \arrow[r, "{p'}"] \arrow[d,"b'"] & E \times U \arrow[d, "b \times id"] \arrow[r,"q"] & \mathbb{P}^{m-1} \times U \\
V'_U  \arrow[d,"f"] \arrow[r,"p"]  & \mathbb{P}^m \times U \\
U 
\end{tikzcd}
$$ 
    where $b \colon E \to \mathbb{P}^m$ is the blow up at the point $z$, the restriction 
    of $(q \circ p')$ to $D$ is finite surjective onto $\mathbb{P}^{m-1} \times U$ and the 
    square in the diagram is cartesian. 
         
    \end{lem}      
     \begin{proof} 
     This is a relative version of \cite[Lemma 11.2.1]{HL}. 
     We may replace $V'$ with $V'/G$ and assume 
     that the action of $G$ is trivial. 
     Let $v \in V$.
        Let 
      $n \in \mathbb{N}$ be the smallest natural number greater than or equal to $m$ such that there exists a 
      Zariski open neighbourhood $U$ of $v$ and the 
       morphism 
      $f$ factors through
       a finite morphism
      $i \colon V'_U \rightarrow \mathbb{P}^n \times U$.
      The fact that there exists such an $n$ is because 
      $f$ is projective. 
       Let $V'_v$ denote the fibre over $v$. If 
      $m = n$ then we have nothing to prove. Suppose $n > m$. Let 
      $(z,v) \in (\mathbb{P}^n \times U)(K)$ be a point that is not contained in $i(V'_v)$.
      Let $C := i(V'_U) \cap (\{z\} \times U) \subset \mathbb{P}^n \times U$.
      Observe that $C$ is a closed subset of $\mathbb{P}^n \times U$ and hence $p_2(C)$ 
      is a closed subset of $U$ 
      where 
      $p_2$ is the projection $\mathbb{P}^n \times U \to U$.
       Furthermore, our choice of $z \in \mathbb{P}^n$ implies that $v \notin p_2(C)$. We abuse notation
       and call the complement of $C$ in $U$, $U$ as well. 
       By construction, for every $u \in U$, $(z,u) \notin i(V'_u)$.
       Let $p \colon \mathbb{P}^n \smallsetminus \{z\} \to \mathbb{P}^{n-1}$
       denote the projection through the point $z$. It follows that the map 
       $p \times \mathrm{id} \colon (\mathbb{P}^n \smallsetminus \{z\}) \times U \to \mathbb{P}^{n-1} \times U$
       restricts to a finite morphism $i_1 \colon V'_U \to \mathbb{P}^{n-1} \times U$. 
       This contradicts our assumption that $n$ was minimally chosen. 
        
       We now prove the second part of the Lemma. 
       Let $D \subset V'_U$ be a horizontal divisor. 
       Let $z \in \mathbb{P}^m(K) \times \{v\}$ be a point which 
       is not contained in the image $p(D)$. We now argue as before. 
        Let $C' := p(D) \cap (\{z\} \times U) \subset \mathbb{P}^m \times U$. 
       We have that $C'$ is closed and hence $p_2(C') \subset U$ is 
       a closed subspace that does not contain $v$. 
       We shrink $U$ so that it does not intersect $p_2(C')$. 
      Let $b \colon E \to \mathbb{P}^m$ be the blow up at the point $z$. 
      Let $V'_{1U} := V'_U \times_{(\mathbb{P}^m \times U)} (E \times U)$. 
     We thus have the diagram above and the remaining assertions can be checked 
      from the construction. 
        \end{proof}      
        
     \begin{rem} 
   \emph{Observe in Lemma \ref{Local factorization} that 
      $V'_{1U}$ comes equipped with a natural action 
      by the group $G$ and the morphism 
      $b' \colon V'_{1U} \to V'_U$ is $G$-equivariant.}     
     \end{rem}

  \begin{lem} \label{no need for etale}
   Let $F$ be a valued field. 
    Let $f \colon V' \to V$ be a finite surjective morphism between pure $K$-varieties where $V$ is assumed to 
    be normal. 
    Assume there exists a proper closed subset $D \subset V$ which satisfies the following property.

    \begin{itemize} 
    \item If $n$ denotes the supremum of the values $\mathbf{card}\{f^{-1}(v)\}$ 
    as $v$ varies along $V$ then 
    for every $u \in U := V \smallsetminus D$, 
    $n = \mathbf{card}\{f^{-1}(u)\}$. 
   \end{itemize} 
   We then have that for every $v \in U$
   and every point $v' \in f^{-1}(v)$, there exists
   definable  
   sets
   $W \subset {V}$ containing $v$ and 
   $W' \subset V'$ containing $v'$ such that 
   $\widehat{W}$ and $\widehat{W'}$ are open in 
   $\widehat{V}$ and $\widehat{V'}$ respectively and 
   $\widehat{f}$ restricts to a homeomorphism 
   from $\widehat{W'}$ onto $\widehat{W}$. 
     \end{lem}    
   \begin{proof}
   Let $F$ be a model of ACVF such that $v \in U(F)$. 
    Let $U' := f^{-1}(U)$. 
    For every $v' \in U'$ such that $v' \mapsto v$ we choose 
    a suitably small $F$-definable $v+g$-open subset $W''_{v'} \subset {U'}$ such that 
    if $v'_1 \neq v'_2$ are distinct preimages of $v$ then
    $W''_{v'_1} \cap W''_{v'_2} = \emptyset$. 
    Let $W$ be the $F$-definable set $\bigcap_{v' \in f^{-1}(v)} {f}(W''_{v'})$.
     By \cite[Corollary 9.7.4]{HL}, 
    the morphism $\widehat{f}$ is open. 
    Furthermore, since $f$ is proper, \cite[Lemma 4.2.26]{HL} implies that
    $\widehat{f}$ is a definably closed map.  
    It follows that $\widehat{W}$ is an pro-$F$-definable open neighbourhood of $v$.
     For any $v' \mapsto v$, let $W'_{v'} := {f}^{-1}(W) \cap W''_{v'}$. 
     
    We claim that 
    ${f}^{-1}(W) = \bigcup_{v' \in f^{-1}(v)} W'_{v'}$. 
    By definition of $W'_{v'}$, we see that 
    $\bigcup_{v' \in f^{-1}(v)} W'_{v'} \subseteq {f}^{-1}(W)$. 
    Hence, we are left to show that 
    if $w' \in {U'}$ is such that ${f}(w') \in W$ then 
    for some $v' \in f^{-1}(v)$, $w' \in W'_{v'}$. 
    Let $w := {f}(w')$. 
    Since $w \in U$, 
    $\mathbf{card}({f}^{-1}(w)) = n$ and by construction
    for every $v' \in f^{-1}(v)$, there exists $w'_{v'} \in W'_{v'}$ 
    such that ${f}(w'_{v'}) = w$.  
         Since the $W'_{v'}$ are mutually disjoint, the points
     $w'_{v'}$ must account for all the preimages of 
     $w$. This implies that for some $v'$, 
     $w'_{v'}  = w$. 
     We have thus verified the claim. 
     
     Observe that
     for every $v' \in f^{-1}(v)$,
      $f$ restricts to a bijection from 
      $W'_{v'}$ to $W$. 
    By \cite[Lemma 4.2.6]{HL}, this implies that $\widehat{f}$ restricts to a bijection 
    $\widehat{W'_{v'}} \to \widehat{W}$. 
    The morphism 
    $\widehat{f}$ is clopen when restricted to 
    $\bigcup_{v' \in f^{-1}(v')} \widehat{W'_{v'}}$. 
    By construction, for every $v'$, $\widehat{W'_{v'}}$ is open. 
    It follows that $\widehat{f}$ restricts to a homeomorphism from each 
    $\widehat{W'_{v'}}$ onto $\widehat{W}$.  
   \end{proof} 
   
   In the following lemma, when we write $v \in V$ where $V$ is a 
   quasi-projective variety defined over a field $F$, we mean 
   $v \in V(\mathbb{U}')$ where $\mathbb{U}'$
   where $\mathbb{U}'$ 
   is chosen to be a suitable universal domain for the theory ACF. 
      
   \begin{lem} \label{existence of D}
   Let $F$ be a field and
    let $\phi \colon V' \to V$ be a morphism of quasi-projective $F$-varieties whose fibres 
    are pure of dimension $m$ for some $m \in \mathbb{N}$.   
    We assume that the morphism $\phi$ factors through a finite surjective morphism 
   $f \colon V' \to P \times V$ via the
    projection $p \colon P \times V \to V$ where $P$ is an irreducible $F$-variety.
    There then exists a 
    Zariski open dense subset $U \subseteq V$ and a
    Zariski closed subset $T \subset P \times U$ satisfying the following properties.
    \begin{enumerate} 
   \item The restriction $p_{|T} \colon T \to U$ is flat.
   \item  Given $u \in U$, let $n(u)$ denote the supremum of the set 
   $\{\mathbf{card}\{f^{-1}(x,u)\}| x \in P\}$. 
   We then have that 
   $$\{(x,u) \in P \times U | \mathbf{card}\{f^{-1}(x,u)\} < n(u)\} \subseteq T.$$
  \item When $V$ is a smooth curve and the morphism $\phi$ is 
  flat, we can take $U = V$.    
   \end{enumerate}
%
   \end{lem}
     \begin{proof} 
     We can assume at the outset that $V$ is integral. 
     We begin by defining an ACF-definable set $E \subset P \times V$ as follows. 
     
     Let $M$ denote the supremum of the set $\{\mathbf{card}\{f^{-1}(x,u)\} | (x,u) \in P \times V\}$.
     There exists an ACF-definable partition $S_1,\ldots,S_M$ of $V$ such that 
     if $s \in S_i$ then 
     $$\mathrm{sup} \{\mathbf{card}\{f^{-1}(x,s)\} | x \in P\} = i.$$ 
     Let $E' = \bigcup_{1 \leq i \leq M} E'_i$ where $E'_i$ is 
     the ACF-definable subset of $P \times S_i$ consisting 
     of the set of pairs $(x,s)$ such that 
     $$\mathbf{card}\{f^{-1}(x,s)\} = i.$$
      
     Observe that $E'$ is $F$-definable. 
      Let $E$ be the complement of $E'$ in $P \times V$. 
      Hence $E$ is a constructible subset of $P \times V$ and 
      for every $v \in V$, $E_v$ is of dimension strictly smaller than $m$.
      Hence 
      $\mathrm{dim}(E) < \mathrm{dim}(V) + m$. 
      Let $\overline{E}$ denote the Zariski closure of $E$ in 
      $P \times V$.
      
      We now show how to choose $U$ as required by the lemma. Let $Z \subset V$ be 
      the set of points $v \in V$ such that 
      $\overline{E}$ contains $P \times \{v\}$. 
      The set $Z$ is ACF-definable. 
      Since the dimension of $E$ is strictly smaller than 
      $\mathrm{dim}(V) + m$, there exists a Zariski open subset $U$ of $V$ that is disjoint from 
      $Z$. Let $T := \overline{E} \cap (P \times U)$.
      We now shrink $U$ further if necessary so that the restriction ${p}_{|T} \colon T \to U$ is flat. 
      This verifies assertions (1) and (2) of the lemma. 
      
       Suppose that $V$ is a smooth curve and the morphism $\phi$ is flat. 
       Since $\mathrm{dim}(E) \leq m$,
       we see that $\overline{E}$ cannot contain any fibre of the form 
       $P \times \{v\}$ 
       for some $v \in V$. 
       Indeed, if $\overline{E}$ contains $P \times \{v\}$ 
       for some $v \in V$ then 
       $P \times \{v\}$ must be an irreducible component of $\overline{E}$ and 
       $E \cap P \times \{v\}$ must be dense in $P \times \{v\}$.
       This is not possible because $\mathrm{dim}(E_v) < m$. 
       
       If $Y \subset \overline{E}$ is an irreducible component 
       whose image via $p$ is a closed point $v \in V(F)$ then 
       we can identify $Y$ with a Zariski closed subset of 
       $P_v$ and write $T(Y) := Y \times V$. 
       If on the other hand $Y$ is an irreducible component of 
       $\overline{E}$ which surjects onto $V$ via $p$ then 
       we set $T(Y) := Y$. 
       If $Y_1,\ldots,Y_m$ are the irreducible components of $\overline{E}$ then
       let $T := \bigcup_i T(Y_i)$. 
       By construction $T$ satisfies assertions (1) and (2) when $U = V$. 
             
               \end{proof}

  \subsection{The inflation homotopy for families}  \label{inflation for families}
           
  \begin{lem} \label{lem : distance to closed sets}
  Let $V$ be a quasi-projective variety over a valued field $F$. 
  Let $m \colon V \times V \to \Gamma_\infty$ be a definable metric as constructed in the proof 
  of \cite[Lemma 3.10.1]{HL}. Let $D \subset V$ be a $v$-closed subset
  and $x \in V \smallsetminus D$. 
  Then the set $\{m(x,d) | d \in D\} \subset \Gamma$ has a supremum in $\Gamma$.   
  \end{lem}         
  \begin{proof}
   Observe that the definable metric as constructed in the proof of \cite[Lemma 3.10.1]{HL}
   is such that for any $z \in V$ and $\gamma \in \Gamma$,
    if $B(z,\gamma) := \{y \in V | m(y,z) \geq \gamma\}$ then 
    the family $\{B(z,\gamma)\}_\gamma$ is a fundamental system of $v$-open neighbourhoods 
    of $y$ in $V$.
    Hence, if $\{m(x,d) | d \in D\}$ does not have a supremum in $\Gamma$ then 
    we see that $x$ must be a limit point of $D$ for the $v$-topology. However, this is not 
    possible since $D$ is $v$-closed. 
      \end{proof}

   \begin{lem} \label{inflation for families}
 Let $F$ be a valued field and
    let $\phi \colon V' \to V$ be a morphism of quasi-projective $F$-varieties
    which satisfies the following properties. 
    \begin{enumerate}
   \item The fibres of the morphism $\phi$ 
    are pure of dimension $m$ for some $m \in \mathbb{N}$.   
 \item  There exists a proper closed subset
    $D' \subset V'$ such that we have a map 
    $f \colon V' \smallsetminus D' \to \mathbb{A}^m \times V$. 
 Furthermore, for every 
    $v \in V$ and $x \in \mathbb{A}^m \times \{v\}$, 
    if $x' \in f^{-1}(x)$ then there exists
   definable  
   sets
   $W \subseteq {\mathbb{A}^m \times \{v\}}$ containing $x$ and 
   $W' \subseteq V'_v$ containing $x'$ such that 
   $\widehat{W}$ and $\widehat{W'}$ are open in 
   $\widehat{\mathbb{A}^m \times \{v\}}$
    and $\widehat{V'_v}$ respectively and 
   $\widehat{f}_v$ restricts to a homeomorphism 
   from $\widehat{W'}$ onto $\widehat{W}$. 
 \item Let $\{\xi_i \colon V' \smallsetminus D' \to \Gamma_\infty\}_i$ be a finite family of $v$-continuous 
 functions which are definable over $F$. Furthermore, for every $i$ and every 
 $v \in V$, $\xi_i^{-1}(\infty) \cap V'_v$ is either empty or the union of irreducible components of 
 $V'_v$. 
 \item Let $G$ be a finite group acting on $V'$ such that
 $D'$ is $G$-invariant and
  the morphisms 
 $\phi$ and $f$ are $G$-equivariant 
 when we endow $V$ and $\mathbb{A}^m \times V$ with the trivial action.  
 \end{enumerate}  
    Then there exists a $G$-equivariant homotopy 
    $H'_{inff} \colon  [0,\infty] \times \widehat{V'}  \to \widehat{V'} $ with the following properties. 
    \begin{enumerate}[(a)]
 \item The homotopy $H'_{inff}$ restricts to well defined homotopies 
 along the fibres of the morphism $\widehat{\phi}$. 
 \item Let $v \in V$ and $X \subset {\phi}^{-1}(v)$ be 
 a Zariski closed subset of dimension strictly smaller than $\mathrm{dim}({\phi}^{-1}(v))$. 
    Then $\widehat{X} \cap H'_{inff}(0,\widehat{V'}) \subset \widehat{D'}$. 
    \item The homotopy $H'_{inff}$ can be taken to be $G$-invariant
     and respects the levels of the functions $\xi_i$. 
  \end{enumerate}
 \end{lem}           
 Our notation \emph{inff} is a concatenation of the short forms 
$inf$ for inflation and $f$ for families. 
\begin{proof}  
 We adapt the proof of \cite[Lemma 10.3.2]{HL} to prove the lemma. 
   Let $h_0 \colon [0,\infty] \times\mathbb{A}^m \to \widehat{\mathbb{A}^m}$ 
     be the standard homotopy which sends 
     $(t,x)$ to the generic type of the closed polydisk around $x$ of 
     valuative radius $(t,\ldots,t)$. 
  We abuse notation and use $h_0$ to denote the homotopy 
  $[0,\infty] \times (\mathbb{A}^m \times V) \to \widehat{(\mathbb{A}^m \times V)/V}$ 
 defined by $(t,(a,v)) \mapsto (h_0(t,a) \otimes v)$. 
 The fact that this is a well defined \emph{homotopy} follows from \cite[Lemma 9.8.3]{HL}.

 By 
   condition (3) of the lemma and \cite[Lemma 7.3.4]{HL}, for each 
  $u := (x,v) \in \mathbb{A}^m \times V$, 
   there exists a $\gamma_0(u) \in \Gamma$
   such that
   for any $v' \in f^{-1}(v)$,
    the path 
   $t \mapsto h_0(t,u)$ 
   for $t \in [\gamma_0(u),\infty]$ 
   lifts uniquely to a definable path
   $\widehat{V'_v}$ starting from $v'$.
  Furthermore, observe that for every $t \in [0,\infty]$,
  $\widehat{p}(h_0(t,u)) = v$. 
  It follows that if $v' \in V'$ 
  is such that $f(v') = u$ then any lift of the 
  path $t \mapsto h_0(t,u)$ starting from 
  $v'$ must belong to 
  $\widehat{V'_v}$. 
  Thus, for any $v' \in f^{-1}(v)$, the path   
  $t \mapsto h_0(t,u)$ 
   for $t \in [\gamma_0(u),\infty]$ 
   lifts uniquely to a path in 
   $\widehat{V'}$ starting from $v'$. 
   The remainder of the proof
   can be carried out using more or less the same arguments 
      as in the proof of \cite[Lemma 10.3.2]{HL}.
     Note that $V'$ is not a projective variety.      
   Hence instead of using \cite[Lemma 4.2.29]{HL} as 
   in the proof of \cite[Lemma 10.3.2]{HL}, we use 
   Lemma \ref{lem : distance to closed sets}.  
\end{proof}

    \section{Generic deformations} \label{finite obstruction}

  \subsection{The initial set-up} \label{section : initial set-up}
      The goal of this section is to summarize the constructions of Lemmas
      \ref{Local factorization} and \ref{existence of D} and introduce notation that will be used
      in the proofs of 
      Proposition \ref{generic relative cdr} and 
      Theorem \ref{generic cdr}.
      
      Let $\phi \colon V' \to V$ be a projective morphism 
      of 
      $K$-varieties such that the fibres of $\phi$ are pure of dimension $m$. 
      We suppose that $V$ is integral. 
      As before, let $G$ be a finite algebraic group acting on $V'$ such that the morphism 
      $\phi$ is $G$-equivariant when $V$ is endowed with the trivial action. 
      Let $\{\xi_i \colon V' \to \Gamma_\infty\}_i$ be a finite family of $v+g$-continuous $K$-definable functions. 
      
      By Lemma \ref{Local factorization}, there exists a Zariski open neighbourhood $U \subset V$ such that 
    the morphism $\phi \colon \phi^{-1}(U) \to U$ factors through a finite surjective
    $G$-equivariant morphism
    $g_U \colon \phi^{-1}(U) \to \mathbb{P}^m \times U$ over $U$  where 
    the $G$-action on $\mathbb{P}^m \times U$ is taken to be trivial. 
    Let $V'_U := \phi^{-1}(U)$. 
    We may shrink $U$ further and assume that all generic points of $V'_U$ belong to the
    fibre over the generic point of $U$. 
    We have the following commutative diagram. 
   $$
\begin{tikzcd}[row sep = large, column sep = large]
V'_U  \arrow[d,"\phi_U"] \arrow[r,"g_U"]  & \mathbb{P}^m \times U \arrow[dl]\\
U 
\end{tikzcd}
$$ 

We apply the following steps to choose a horizontal divisor in $V'_U$. 
\begin{enumerate}
\item  We apply Lemma \ref{existence of D} to the diagram above to obtain 
  a divisor $D_U \subset \mathbb{P}^m \times U$ which satisfies the conditions of 
  \ref{existence of D}
  and $(\mathbb{P}^m \times U) \smallsetminus D'_U \subseteq \mathbb{A}^m \times U$
  for some copy of $\mathbb{A}^m$ in $\mathbb{P}^m$. 
 Let $D'_U := g_U^{-1}(D_U)$. 
\item  For every $i$, let $Y_i := \xi_i^{-1}(\infty)$. 
  Since $\xi_i$ is $v+g$-continuous, $Y_i$ is a Zariski closed subset of $V'_U$. 
  We shrink $U$ further if necessary and suppose that for every $i$, the generic points 
  of $Y_i$ lie on the fibre over the generic point of $U$. 
  Furthermore, for every $i$, 
  and every irreducible component $Y_{ij}$ of $Y_i$, we can shrink $U$
   so that the map 
  $Y_{ij} \to U$ is flat. It follows from \cite[Corollary 9.6]{hart} that the 
  fibres of $\phi_{|Y_{ij}}$ are pure and equi-dimensional. 
  
  For every $i$, let $\mathcal{J}_{i}$ denote those indices $j$ such that 
  the irreducible component 
  $Y_{ij}$ of $Y_i$ is of dimension strictly less than $\mathrm{dim}(V'_U)$. 
  We enlarge $D'_U$ so that for every $i$ and $j \in \mathcal{J}_i$, 
  $Y_{ij} \subset D'_U$. 
\item We shrink $U$ further and enlarge $D'_U$ so that 
$D'_U$ remains a horizontal divisor and 
$V'_U \smallsetminus D'_U$ is the disjoint union of irreducible varieties 
whose fibres over $U$ are pure and equidimensional. This is possible by first choosing 
a closed subset $A$ of the generic fibre $V'_\eta$ such that $V'_\eta \smallsetminus A$ is the disjoint union 
of irreducible $k(\eta)$-varieties. Note that $\mathrm{dim}(A) < \mathrm{dim}(V'_\eta)$. 
We can shrink $U$ so that the Zariski closure $A'$ of $A$ in $V'_U$ will satisfy the required property.
 We then replace $D'_U$ with the union $D'_U \cup A'$. 
 \item We have $D'_U = g_U^{-1}(g_U(D'_U))$. Hence, $D'_U$ is $G$-invariant. 
\end{enumerate}

 By Lemma \ref{Local factorization},
there exists a 
$K$-point $z \in \mathbb{P}^m$ 
such that after shrinking $U$ if necessary
we can extend the diagram above to get the following commutative diagram.
          
   $$
\begin{tikzcd}[row sep = large, column sep = large]
V'_{1U} \arrow[r, "{g'_U}"] \arrow[d,"b_U'"] & E \times U \arrow[d, "b \times id"] \arrow[r,"p_U"] & \mathbb{P}^{m-1} \times U \\
V'_U  \arrow[d,"\phi_U"] \arrow[r,"g_U"]  & \mathbb{P}^m \times U \arrow[dl]\\
U 
\end{tikzcd}
$$ 
    where $b \colon E \to \mathbb{P}^m$ is the blow up at the point $z$, the restriction 
    of $(p_U \circ g'_U)$ to $D'_U$ is finite surjective onto $\mathbb{P}^{m-1} \times U$ and the 
    square in the diagram is cartesian. 
    
 Let $Z'_U \subset V'_U$ be the preimage for the morphism 
 $g_U$ of the closed subset $\{z\} \times U$. 
 Let 
 $\phi'_U := \phi_U \circ b'_U$, 
 $Z'_{1U} := {b'_U}^{-1}(Z'_U)$, 
 $D'_{1U} := {b'_U}^{-1}(D'_U)$ 
 and $D'_{11U} := D'_{1U} \cup Z'_{1U}$.    
 For every $i$, let $\xi_{1i} := \xi_i \circ b'_{U}$.

  In \S \ref{relative curves for families}, we adapt
 the construction in \cite[\S 11.3]{HL} 
       of the relative curve homotopy
     to the 
        fibration 
       $p'_{U} \colon V'_{1U} \to F \times U$
        where $F := \mathbb{P}^{m-1}$
        and $p'_{U} := p_{U} \circ g'_{U}$.
     
    \subsection{The relative curve homotopy for families} \label{relative curves for families} 
   
       Let the notation be as in
       \S \ref{section : initial set-up}.
               There exists an 
       open subset $W \subset F \times U$ such that 
       $p_{U}^{-1}(W) \subset E \times U$ is isomorphic to $\mathbb{P}^1 \times W$.
       Indeed, if $p \colon E \to F$ denotes the projection map then
       there exists a Zariski open subset $F_0 \subset F$ such that $p^{-1}(F_0) = F_0 \times \mathbb{P}^1$. 
       By definition, $p_{U} = p \times \mathrm{id} \colon E \times U \to F \times U$. Let $W := F_0 \times U$.
       By construction, $p'_{U}$ restricts to a finite morphism from
       $D'_{11U}$ onto $F \times U$. 
       
       Let $A := {p'_{U}}^{-1}(W) \subset V'_{1U}$ and 
       $B := p_{U}^{-1}(W) \subset E \times U$.  
                Furthermore, we can shrink $W$ if necessary so that
       the map $g'_{U} \colon A \to B$ factors through $A \to A' \to B$ where $A \to A'$ is radicial and 
       for every $w \in W$, $A'_w \to B_w = \mathbb{P}^1$ is generically étale. 
       This is possible because the above property is true over the generic point of $W$. 
       Using that the morphism $F \times U \to U$ is flat and hence open, we can
        shrink $U$ so that the 
        projection $W \to U$ is surjective.
     \\       
     
\noindent \emph{
      The homotopy $H'_{curvesf}$.}
\\ 

      We fix three points - $0,1$ and $\infty$ on $\mathbb{P}^1$. This is to 
      make sure that the notions of standard homotopy and 
      closed ball are well defined. Given a divisor $X$ on $B$, \cite[\S 10.2]{HL} implies 
      that
       there exists a well defined definable map
      $\psi_X \colon [0,\infty] \times B \to \widehat{B/W}$ 
      which fixes $\widehat{X}$ and if $w \in W$ then 
      $\psi_X$ restricts to a well defined homotopy on $\widehat{B_w}$.
      Furthermore, if $X \cap B_w$ is finite, then the image of the homotopy 
      restricted to the fibre $B_w$ is a $\Gamma$-internal subset of $\widehat{B_w}$.  
      We emphasize that this map is a priori not continuous unless for instance we 
      add additional hypothesis on $X$. 
      By definition, for $w \in W$, the fibres of the 
      map $\widehat{B/W} \to W$ are copies of $\widehat{\mathbb{P}^1}$. 
      The definable map $\psi_X$ is constructed using the 
      \emph{standard homotopy}
      (cf. \cite[\S 7.5, p.105]{HL})
       on $\mathbb{P}^1$ and then defining cut offs of this homotopy via the 
      divisor $X$.
      
%
\begin{lem} \label{curves for families}
   After shrinking the open set $U$ if necessary, 
   there exists a constructible set $C \subset V'_{1U}$ with the following properties. 
   \begin{enumerate} 
   \item The set $C$ maps surjectively onto $F \times U$ via $p'_{U}$.
   \item There exists an open subset $W' \subset F \times U$ such that 
   $V'_{1U}$ contains ${p'_U}^{-1}(W')$. 
  \item Let $H'_{inff}$ be the homotopy obtained from Lemma \ref{inflation for families} associated 
  to the divisor $D'_{11U}$ and the morphism $g'_U \colon V'_{1U} \smallsetminus D'_{11U} \to E \times U$. 
  (By construction, the image of $g'_U$ is contained in a copy of $\mathbb{A}^m \times U$).  
  The image $H'_{inff}(0,\widehat{V'_{1U}})$ is contained in $\widehat{C}$.
     \item  There exists a   
   deformation retraction $h'_{curvesf} \colon [0,\infty] \times C \to \widehat{C/F \times U}$. 
   The image $\Upsilon'_2 := h'_{curvesf}(0,C) \subset \widehat{V'_{1U}/F \times U}$ is 
   iso-definable and relatively $\Gamma$-internal over $F \times U$. 
   Furthermore, $h'_{curvesf}$ respects the levels of the functions $\xi_{1i}$ for every $i$ and 
   is $G$-equivariant. 
   \end{enumerate}    
   \end{lem} 
\begin{proof} 
    By \cite[Lemma 11.3.2]{HL},
 there exists a 
 divisor $X$ on $B$ such that for any 
 divisor $X'$ containing $X$, $\psi_{X'} \colon [0,\infty] \times B \to \widehat{B/W}$ lifts uniquely to a definable map
 $h \colon [0,\infty] \times A \to \widehat{A/W}$ which is fibrewise a homotopy. 
      We enlarge $X$ so that it contains the image 
      $g'_{U}(D'_{11U} \cap A)$ and the divisor $\infty \times W$. 
     We use 
     $$h'_{curvesf} \colon [0,\infty] \times A \to \widehat{A/W}$$ to denote the unique 
     lift of $\psi_X$. 
     By 
     Lemma 10.2.2 in loc.cit, we can enlarge $X$ so that 
     the lift $h'_{curvesf}$ preserves the levels of the 
     restrictions of the functions $\xi_{1i}$. As the lift is unique, it is $G$-invariant. 
     
     Note from the construction 
     in \cite[Lemma 11.3.2]{HL} that
      $\mathrm{dim}(X) = \mathrm{dim}(W)$.        
       Observe that $W$ is 
          an irreducible $K$-variety.
          It follows that the Zariski closure of the set $\{w \in W | \mathbb{P}^1 \times \{w\} \subset X\}$
          is of dimension strictly smaller than $\mathrm{dim}(W)$. 
%
                 Let $W' \subset W$ be a Zariski open subset over which $X$ is finite.
                 We replace $A$ with ${p'_{U}}^{-1}(W')$ and 
                 $B$ with $p_{U}^{-1}(W')$. 
                  It then follows 
          that $\psi_X \colon [0,\infty] \times B \to \widehat{B/W'}$ is
           a well defined \emph{homotopy} and its lift is a well defined homotopy 
          $h'_{curvesf} \colon [0,\infty] \times A \to \widehat{A/W'}$ which is
           $G$-invariant and in addition preserves the levels of the functions $\xi_{1i}$.
         Using that the morphism $F \times U \to U$ is flat and hence open, we can
        shrink $U$ so that the 
        projection $W' \to U$ is surjective.

                We extend $h'_{curvesf}$ to a definable map 
                $[0,\infty] \times A \cup D'_{11U} \to \widehat{A \cup D'_{11U}}$ 
          by setting $h'_{curvesf}(t,x) = x$ for every $x \in D'_{11U}$. 
          By Lemma 11.3.3 in loc.cit., the map $h'_{curvesf}$ is 
          a well defined homotopy with canonical extension 
           $H'_{curvesf} \colon [0,\infty] \times \widehat{A \cup D'_{11U}}
            \to \widehat{A \cup D'_{11U}}$. \emph{We set $C := A \cup D'_{11U}$}.
           
           It remains to verify the inequality 
           $H'_{inff}(e,\widehat{V'_{1U}}) \subset \widehat{C}$. Observe that 
           by construction the projection $W' \to U$ is surjective and hence
           the complement $S$ of $W'$ in $F \times U$ 
            cannot contain any subset of the form $F \times u$ where $u \in U$. As a result 
           for any $u \in U$, 
           ${p'}_{U}^{-1}(S) \cap {\phi'}_{U}^{-1}(u)$ is 
           a closed subset whose dimension is strictly smaller than that of 
           ${\phi'}_{U}^{-1}(u)$. The inflation property (cf.
            Lemma \ref{inflation for families} (2)) of the homotopy $H'_{inff}$ implies the result. 
           \end{proof} 
                      

We now prove the following proposition which
is related to \cite[Proposition 11.7.1]{HL}
in that we treat a family of quasi-projective varieties parametrized 
by a quasi-projective $K$-variety. While loc.cit. shows 
that there exists a family of deformation retractions uniform over the base, we show 
that there exists a global homotopy that restricts to a homotopy on each of the fibres but compromise by 
treating only a suitable open subset of the family.

\begin{prop} \label{generic relative cdr}
     Let $V$ be a pure quasi-projective $K$-variety. 
   Let $\phi \colon V' \to V$ be a morphism between quasi-projective varieties whose image is dense. Let $G$ be a finite algebraic group acting on $V'$
   which restricts to a well defined action along the fibres of the morphism $\phi$
   and 
   $\{\xi_i \colon V'\to \Gamma_{\infty}\}_i$ be a finite collection of $K$-definable functions. 
   There exists an open dense subset $U \subset V$ such that if $V'_U := \phi^{-1}(U)$ then 
    there exists a generalized interval $I$ and a homotopy $h'_{rel} \colon I \times V'_U \to \widehat{V'_U/U}$    
     which satisfies the following properties. 
      \begin{enumerate} 
      \item  The image of $h'_{rel}$ is a relatively $\Gamma$-internal subset of $\widehat{V'_U/U}$.
      \item The homotopy $h'_{rel}$ is invariant for the action of the group $G$ and
    respects the levels of the functions $\xi_i$.
        \item The homotopy $h'_{rel}$ is Zariski generalizing. 
     \end{enumerate} 
     \end{prop}
     \begin{proof}
     Lemma \ref{prove the simpler case}
     allows us to reduce to the case of a morphism 
     $\phi \colon V' \to V$ satisfying assertions 
     (1)-(4) of \ref{prove the simpler case}. 
%
     We make use of the notation and constructions
  introduced in \S \ref{section : initial set-up}.
     The simplifications in \S \ref{section : initial set-up} show that 
    it suffices to 
     prove the theorem for the 
     morphism $\phi'_{U} \colon V'_{1U} \to U$.
Indeed, any relative homotopy on $V'_{1U}$ that 
restricts to a well defined homotopy on $Z'_{1U}$ must descend to a homotopy on $V'_{U}$.
This is because for every $v \in U$, 
the image of $Z'_{1Uv}$ for the morphism $b'_{U}$
is $Z'_{Uv}$ which is the disjoint union of Zariski closed points and 
a relative homotopy that restricts to a well defined homotopy on $Z'_{1U}$ will also 
restrict to a well defined homotopy on each connected component of $Z'_{1Uv}$. 
By \cite[Lemma 3.9.4]{HL}, for every $v$ the restriction of the homotopy to $V'_{1Uv}$ descends to a homotopy
on $V'_{Uv}$, and likewise the homotopy on $V'_{1U}$ descends to a homotopy on $V'_{U}$.  
 
  We proceed to prove the proposition by induction on the dimension of the fibres 
  of the morphism 
  $\phi'_{U}$.
  Note that when the dimesion of the generic fibre is $0$, there is nothing to prove.    
  Let $H'_{inff}$ be the homotopy obtained from Lemma \ref{inflation for families} associated 
  to the divisor $D'_{11U}$ and the morphism $g'_U \colon V'_{1U} \smallsetminus D'_{11U} \to E \times U$. 
   By \ref{curves for families}, there exists homotopies 
   $H'_{inff} \colon I_1 \times \widehat{V'_{1U}} \to \widehat{V'_{1U}}$ and 
  $h'_{curvesf} \colon I_2 \times \widehat{C/F \times U} \to \widehat{C/F \times U}$ 
  where $C$ is a constructible subset of $V'_{1U}$ where $I_1 = I_2 = [0,\infty]$.
  Recall from the construction of the homotopy 
  $h'_{curvesf}$ that there exists a Zariski open subset 
  $W' \subset F \times U$ such that the image of $h'_{curvesf}$ when 
  restricted to those fibres over points in $W'$ is of dimension $1$. 
  Let $H'_{curvesf} \colon I_2 \times \widehat{C} \to \widehat{C}$ denote the 
  canonical extension of the homotopy $h'_{curvesf}$. 
    By construction, 
  the composition 
  $H'_{curvesf} \circ H'_{inff} \colon (I_2 + I_1) \times \widehat{V'_{1U}} \to \widehat{V'_{1U}}$ 
  is a well defined homotopy.  
  Let $\Upsilon'_2$ denote the image 
  $h'_{curvesf}(0,C) \subset \widehat{V'_{1U}/F \times U}$. 
  By construction, $\Upsilon'_2$ is iso-definable and relatively 
  $\Gamma$-internal over $F \times U$. 
  By Lemma 6.4.1 in \cite{HL}, $\widehat{\Upsilon'_2}$ can 
  be identifed with a pro-definable subset of $\widehat{V'_{1U}}$
such that over a point $u \in F \times U$, $\widehat{\Upsilon'_2}_u = \Upsilon'_{2u}$. 
The image of 
$H'_{curvesf} \circ H'_{inff}$ is contained in $\widehat{\Upsilon'_2}$. 
By construction, $H'_{curvesf} \circ H'_{inff}$ is compatible with the homotopy   
 on $\widehat{U}$ that fixes every point.
  
  \begin{lem} \label{getting around compactness} 
    There exists a 
    pseudo-Galois cover (cf. \cite[\S 2.12]{HL})
     $f \colon F' \to F \times U$
    and a morphism
    $\kappa \colon \Upsilon'_2 \times_{(F \times U)} F' \to F' \times \Gamma_\infty^M$ for 
    some $M \in \mathbb{N}$ such that the restriction
    $\widehat{\kappa}_{|\widehat{\Upsilon'_{2F'}}} \colon \widehat{\Upsilon'_{2F'}} \to \widehat{F'} \times \Gamma^M_\infty$ is a homeomorphism onto its image
    where $\Upsilon'_{2F'} :=  \Upsilon'_2 \times_{(F \times U)} F'$.
    \end{lem} 
    \begin{proof}
    Recall that $V$ is a projective variety. 
    Let $d_v \colon  V \to [0,\infty]$ denote the schematic distance 
    to $V_{bord}$ where $V_{bord} := V \smallsetminus U$. 
    By construction, the morphism $g'_{U} \colon V'_{1U} \to E \times U$ is finite and hence projective. 
   It follows that the morphism 
  $V'_{1U} \xrightarrow{g'_{U}} E \times U$ factors through an embedding 
  $j' \colon V'_{1U} \to \mathbb{P}_K^N \times E \times U$ for some $N \in \mathbb{N}$. 
  Let $V'_1$ be the closure in $\mathbb{P}_K^N \times E \times V$ of $V'_{1U}$.
  Let $p' \colon V'_1 \to V$ be the composition of the morphisms 
  $V'_1 \to E \times V$ which is the restriction of a projection and $E \times V \to V$. 
  Let $d'_v := d_v \circ p'$.
   
  Similarly,  let $\rho \colon V'_{1} \to [0,\infty]$ denote the schematic distance from the closure of 
        $D'_{11U}$ in $V'_1$ and $\eta \colon V'_{1} \to [0,\infty]$ denote the schematic distance 
        to the preimage of $(F \times V) \smallsetminus W'$ where $W' \subset F \times U$ is as in the statement of 
        Lemma \ref{curves for families}. 
            

   For $\gamma \in [0,\infty)$, let 
   $V_{\gamma} := \{x \in \overline{V} | d_v(x) \leq \gamma\}$ and let
   $V'_{1V_\gamma}$ denote its preimage in $V'_1$ i.e. the set $\{x \in \overline{V'_1} | d'_v(x) \leq \gamma\}$. 
      Observe that $\widehat{V_{\gamma}}$ 
      is 
      definably compact since 
      it is a closed subset of the definably compact space $\widehat{V}$. Furthermore, 
      $V_{\gamma} \subset U$.
      For similar reasons,        
      $\widehat{V'_{1V_\gamma}}$ is definably compact and also $V'_{1V_\gamma} \subset V'_{1U}$. 
      Observe that $\widehat{\Upsilon'_{2 \gamma}}$ is $\sigma$-compact with respect to the 
      restrictions of $\rho$ and $\eta$. 
             
%
%
%

        For ease of notation, let $S := F \times U$ and for $\gamma \in \Gamma$, $S_\gamma := F \times V_{\gamma}$. 
        By \cite[Lemma 6.4.2]{HL}, there exists a 
    pseudo-Galois cover $f \colon F' \to S$
    and a morphism
    $\kappa \colon \Upsilon'_2 \times_{S} F' \to F' \times \Gamma_\infty^M$ over $F'$ for 
    some $M \in \mathbb{N}$ such that 
    $\widehat{\kappa}_{|\widehat{\Upsilon'_{2F'}}} \colon \widehat{\Upsilon'_{2F'}} \to \widehat{F'} \times \Gamma^M_\infty$ is a continuous injection
    where $\Upsilon'_{2F'} := \Upsilon'_2 \times_S F'$. 
    

         We fix $\gamma \in \Gamma$. We claim that 
        $\widehat{\kappa}$ restricts
         to a homeomorphism on the closed subspace $\widehat{\Upsilon'_{2F'_\gamma}}$ where 
        $F'_\gamma$ is the preimage of $S_\gamma$ via $f$ and 
        $\Upsilon'_{2F'_\gamma} := \Upsilon'_{2} \times_{S_\gamma} F'_\gamma$. 
                We apply \cite[Lemma 6.4.3]{HL}
        to $\Upsilon'_2 \subset \widehat{V'_{1V_\gamma}/S_{\gamma}}$
         to verify the claim. Note that
        strictly speaking, loc.cit. requires that $S_\gamma$ be a quasi-projective variety.
        However, one checks that the steps of the proof can be carried out in
        our situation. 
             

       Let $V_{o\gamma}:= \{x \in {V} | d_v(x) < \gamma\}$,
       $S_{o\gamma} := F \times V_{o\gamma}$ and 
       $F'_{o\gamma} := f^{-1}(S_{o\gamma})$. 
       We have shown that $\widehat{\kappa}$ restricts 
       to a homeomorphism on $\widehat{\Upsilon'_{2F'_{o\gamma}}}$. 
       To conclude that $\widehat{\kappa}$ 
       is a homeomorphism on $\widehat{\Upsilon'_{2F'}}$, 
       it suffices to verify that the image of $\widehat{\Upsilon'_{2F'_{o\gamma}}}$
       via $\widehat{\kappa}$ is open in 
       $\widehat{\kappa}(\widehat{\Upsilon'_{2F'}})$. 
       This follows from the fact that by
       construction, 
        $$\widehat{\kappa}(\widehat{\Upsilon'_{2F'_{o\gamma}}}) = \widehat{\kappa}(\widehat{\Upsilon'_{2F'}}) \cap (\widehat{F'_{o\gamma}} \times \Gamma^M_\infty)$$
        since $\kappa$ is over $F'$ and hence also over $S$ and $U$. 
        Note that $\widehat{F'_{o\gamma}} \times \Gamma^M_\infty$ is open in 
        $\widehat{F'} \times \Gamma^M_\infty$.
       
%
%
%
%
%
%
    \end{proof} 
    
   By Lemma \ref{getting around compactness}, there exists a 
    pseudo-Galois cover $f \colon F' \to F \times U$
    and a morphism
    $\kappa \colon \Upsilon'_2 \times_{(F \times U)} F' \to F' \times \Gamma_\infty^M$ for 
    some $M \in \mathbb{N}$ such that the restriction
    $\widehat{\kappa}_{|\widehat{\Upsilon'_{2F'}}} \colon \widehat{\Upsilon'_{2F'}} \to \widehat{F'} \times \Gamma^M_\infty$ is a homeomorphism onto its image
    where $\Upsilon'_{2F'} :=  \Upsilon'_2 \times_{(F \times U)} F'$.
    
    We shrink $U$ if necessary and assume that it is normal. 
    Using the arguments in \cite[6.4.4]{HL},
  there exists
  a finite collection of $K$-definable 
  functions $\mu_j \colon  F' \to \Gamma_{\infty}$ such that,
  if $O \subset U$ is a Zariski open set, 
  for $I$ a generalised interval, 
  a homotopy $a'_f \colon I \times F \times O \to \widehat{F \times O}$
  which lifts to a homotopy $a' \colon I \times f^{-1}(F \times O) \to \widehat{f^{-1}(F \times O)}$ that preserves the levels 
  of the functions $\mu_j$ also induces a homotopy 
  $a'' \colon I \times \widehat{\Upsilon'_2} \cap \widehat{{\phi'_{U}}^{-1}(O)} \to \widehat{\Upsilon'_2} \cap \widehat{{\phi'_{U}}^{-1}(O)}$
   that is $G$-invariant and 
  respects the levels of the functions $\xi_i$.
  
    Observe that the fibres of the morphism 
    $p_2 \circ f$ are pure over $U$. 
    Let $G' := \mathrm{Aut}(F'/F \times U)$.
     Let
    $p_2 \colon F \times U \to U$ denote
    the projection map onto the second coordinate. 
    Observe that the group $G'$ acts on $F'$ 
    along the fibres of the composition $p_2 \circ f$. 
    We apply the induction hypothesis 
    to the morphism $p_2 \circ f \colon F' \to U$
    along with the definable functions $\mu_j$ and the group $G'$. 
    Note that the morphism $p_2 \circ f$ is 
    projective.
         It follows that
         after shrinking $U$ if necessary, 
      we have a well defined homotopy 
      $$H'_{bf} \colon I_3 \times \widehat{\Upsilon'_2} \cap \widehat{{\phi'_{U}}^{-1}(U)} \to \widehat{\Upsilon'_2} \cap \widehat{{\phi'_{U}}^{-1}(U)}.$$
      Let $\Upsilon'_{2U} := \Upsilon'_2 \cap {\phi'_{U}}^{-1}(U)$.
      Observe that by construction, $H'_{bf}$
      restricts to a well defined homotopy 
       $I \times {\Upsilon'_{2U}} \to \widehat{\Upsilon'_{2U}/U}$ 
       whose image is relatively $\Gamma$-internal over $U$. 
       Furthermore, the composition, 
      $H'_{bf} \circ H'_{curvesf} \circ H'_{inff}$ restricts 
       to a well defined homotopy
       $(I_3 + I_2 + I_1) \times V'_{1U} \to \widehat{V'_{1U}/U}$ which fulfils 
       the assertions of the proposition. 
     \end{proof}

\subsection{Proof of Theorem \ref{generic cdr}} \label{proof of generic cdr}

We begin by verifying Theorem \ref{generic cdr} when the dimension of the generic fibre is 
zero.   
  
\begin{lem} \label{lem : generic cdr dim 0 case}
 Let $\phi \colon V' \to V$
 be a morphism
  satisfying assertions 
     (1)-(4) of \ref{prove the simpler case}.
 We assume in addition that the dimension of the generic fibre is $0$. 
 Then, the conclusion of Theorem \ref{generic cdr} is true. 
\end{lem}   
\begin{proof} 
  There exists a Zariski open subset $U \subseteq V$ such that 
 $\phi_U \colon V'_U \to U$ is finite. It follows that $\widehat{V'_U/U}$ is 
 relatively $\Gamma$-internal and iso-definable.  
 By our assumption in \ref{prove the simpler case}, 
 $V$ and hence $U$ is normal. 
 Let $d \colon V \to \Gamma_\infty$ be the schematic distance 
 to the closed subset $V \smallsetminus U$. 
 Observe that $\widehat{V'_U}$ is $\sigma$-compact with respect to 
 the function $\widehat{d_{|U}} \circ \widehat{\phi_U}$.
  By \cite[Theorem 6.4.4]{HL},
 there exists a finite pseudo-Galois cover $\pi \colon U' \to U$ and 
 finitely many definable functions 
 $\{\xi'_j\}_j$ such that any deformation retraction of $\widehat{U}$ that lifts to a deformation retraction on $\widehat{U'}$ that respects the definable functions $\xi'_j$ will lift
 to a deformation retraction of $\widehat{V'_U}$ that is equivariant for the action of the group 
 $G$ and respects the definable functions $\xi_i$. 
  We may hence assume $V'_U \to U$ is a pseudo-Galois cover and 
  $G = \mathrm{Gal}(K(V'_U)/K(U))$ where $K(U)$ is the function field of $U$. 
  
  We shrink $U$ and assume $V'_U$ is normal. We replace $V'$ with the normalization 
  of $V$ in $K(V'_U)$. We can extend the functions 
  $\xi'_j$ to definable functions on $V'$. 
  Let $D'$ denote the preimage of $V \smallsetminus U$ 
  for the morphism $V' \to V$. 
  Recall that in the proof of \cite[Theorem 11.1.1]{HL} applied to the variety
  $V'$, the group $G$ and the functions $\xi'_j$,
   we
  choose a divisor of $V'$ with suitable properties and the first 
  homotopy on $\widehat{V'}$ that we construct is an inflation homotopy 
  with respect to this divisor. 
  We shrink $U$ if necessary and hence enlarge $D'$ so that 
  we can take it to be the divisor in the proof of loc.cit. for $V'$.
  We add the schematic distance to $D'$ to the definable functions $\xi'_j$. 
    
  We apply loc.cit. to the variety $V'$, the group $G$ and 
  functions $\xi'_j$ to get a deformation retraction $H' \colon I \times \widehat{V'} \to \widehat{V'}$ 
  on $\widehat{V'}$ which is $G$-equivariant.
 We suppose in addition that the 
   image of $H'$ is pure of dimension $\mathrm{dim}(V')$.  
  The deformation retraction
  $H'$ descends to a deformation retraction on $\widehat{V}$ and 
  the
  restrictions of $H'$ and $H$ to 
  $\widehat{V'_U}$ and $\widehat{U}$ respectively,
   clearly satisfy (1) - (4) and (6) of
  Theorem \ref{generic cdr}.   
  To verify property (5) for $H$, it suffices to verify
  that if $Z \subset V'$ is a proper Zariski closed subset then 
  $\widehat{Z} \cap H'(e,\widehat{V'_U})$ is empty
  where $e$ is the end point of the generalized interval $I$.   
  Note by construction, $H'$ is the composition of homotopies 
  $H'_{\Gamma} \circ H'_b \circ H'_{curves} \circ H'_{inf}$ where 
  $H'_{inf}$ is an inflation homotopy with respect to $D'$.
  Furthermore, the image of $H'$ is contained in the image of $H'_{inf}$. 
  It follows that if $\widehat{Z}$
   intersects the image of $H'$ non-trivially then it must intersect the 
  image of $H'_{inf}$ non-trivially.
  By construction of $H'_{inf}$, this implies that
   $\widehat{Z} \cap H'(e,\widehat{V'}) \subseteq \widehat{D'}$.
  It follows that 
  $\widehat{Z} \cap H'(e,\widehat{V'_U})$ is empty.
   \end{proof}

 \begin{proof} (of Theorem \ref{generic cdr}) 
  Let the data be as given in Theorem \ref{generic cdr}. 
   It suffices to treat the case of a morphism 
     $\phi \colon V' \to V$ satisfying assertions 
     (1)-(4) of \ref{prove the simpler case}. 
  We proceed by induction on the dimension of the generic fibre. 
 When the generic fibre has dimension $0$, Theorem \ref{generic cdr} is true
 by Lemma \ref{lem : generic cdr dim 0 case}.

  
    We proceed to the general case. 
    Suppose the dimension of the generic fibre is greater than $0$.
          We make use of the notation and construction
  introduced in \S \ref{section : initial set-up}.
    By 
   Lemma \ref{curves for families}, there exists a homotopy
  $h'_{curvesf} \colon [0,\infty] \times \widehat{C/F \times U} \to \widehat{C/F \times U}$ 
  where $C$ is a certain constructible subset of $V'_{1U}$.
  Let $H'_{curvesf} \colon [0,\infty] \times \widehat{C} \to \widehat{C}$ denote the 
  canonical extension of the homotopy $h'_{curvesf}$. 
  Let $\Upsilon'_2$ denote the image 
  $h'_{curvesf}(0,C) \subset \widehat{V'_{1U}/F \times U}$. 
  By construction, $\Upsilon'_2$ is iso-definable and relatively 
  $\Gamma$-internal over $F \times U$. 
  By Lemma 6.4.1 in \cite{HL}, $\widehat{\Upsilon'_2}$ can 
  be identifed with a pro-definable subset of $\widehat{V'_{1U}}$
such that over a point $u \in F \times U$, $\widehat{\Upsilon'_2}_u = \Upsilon'_{2u}$. 
By construction, $H'_{curvesf}$
 is compatible with the homotopy   
  that acts trivially on $\widehat{U}$.
  
   By Lemma \ref{getting around compactness}, there exists a 
    pseudo-Galois cover $f \colon F' \to F \times U$
    and a morphism
    $\kappa \colon \Upsilon'_2 \times_{(F \times U)} F' \to F' \times \Gamma_\infty^M$ for 
    some $M \in \mathbb{N}$ such that the restriction
    $\widehat{\kappa}_{|\widehat{\Upsilon'_{2F'}}} \colon \widehat{\Upsilon'_{2F'}} \to \widehat{F'} \times \Gamma^M_\infty$ is a homeomorphism onto its image
    where $\Upsilon'_{2F'} :=  \Upsilon'_2 \times_{(F \times U)} F'$.
    
    Using the arguments in \cite[6.4.4]{HL},
  there exists
  a finite collection of $K$-definable 
  functions $\mu_j \colon  F' \to \Gamma_{\infty}$ such that,
  if $O \subset U$ is a Zariski open set, 
  for $I$ a generalised interval, 
  a homotopy $a'_f \colon I \times F \times O \to \widehat{F \times O}$
  which lifts to a homotopy $a' \colon I \times f^{-1}(F \times O) \to \widehat{f^{-1}(F \times O)}$ that preserves the levels 
  of the functions $\mu_j$ also induces a homotopy 
  $a'' \colon I \times \widehat{\Upsilon'_2} \cap \widehat{{\phi'_{U}}^{-1}(O)} \to \widehat{\Upsilon'_2} \cap \widehat{{\phi'_{U}}^{-1}(O)}$
   that is $G$-invariant and 
  respects the levels of the functions $\xi_i$.
  
  Let
    $p_2 \colon F \times U \to U$ denote
    the projection map onto the second coordinate. 
    Observe that the fibres of the morphism 
    $p_2 \circ f$ are pure over $U$. 
    Let $G' := \mathrm{Aut}(F'/F \times U)$.
    Observe that the group $G'$ acts on $F'$ 
    along the fibres of the composition $p_2 \circ f$. 
    We apply the induction hypothesis 
    to the morphism $p_2 \circ f \colon F' \to U$
    along with the definable functions $\mu_j$ and the group $G'$. 
         It follows that 
         we can shrink $U$ 
      so that
      if $\Upsilon'_{2,U} := \Upsilon'_2 \cap {\phi'_{U}}^{-1}(U)$ then
       we have a well defined homotopy 
     $$H'_{bf} \colon I_{bf} \times \widehat{\Upsilon'_{2,U}} \to \widehat{\Upsilon'_{2,U}}$$ 
     whose image is an iso-definable $\Gamma$-internal subset of $\widehat{V'_{1U}}$.  
     By construction, the composition 
     $H'_{bf} \circ H'_{curvesf}$ 
      is a well defined homotopy on $\widehat{C}$.

\begin{rem} \label{rem : primary inflation remark}
\emph{Observe that in the proof of Proposition \ref{generic relative cdr}, the homotopy we constructed on 
$\widehat{V'_{1U}/U}$ automatically descends to a homotopy on $\widehat{V'_U/U}$. As explained before, this is 
because for every $u \in U$, $V'_{1u} \to V'_u$ is isomorphic outside the finite set of closed points
$Z'_u \subset V'_u$. Hence, as long as the induced homotopy on $\widehat{V'_{1u}}$ 
preserves $\widehat{Z'_{1u}}$ where $Z'_{1u}$ is the preimage of $Z'_u$ for the morphism 
$V'_{1u} \to V'_u$.}

     \emph{Note that such an argument does not hold in the case of Theorem \ref{generic cdr}
     since we must construct a homotopy on the base $U$ as well. 
     To resolve this issue, we construct an inflation homotopy 
     $H'_{inff-primary}$ 
     on $V'_U$ whose image does not contain $Z'_U$ and is hence contained in 
     $\widehat{V'_{1U}}$. 
     We then proceed to construct suitable homotopies on $\widehat{V'_{1U}}$.}
\end{rem}

\subsection{The relative tropical homotopy} \label{relative tropical homotopy section}
 
 Let $\Upsilon'_{bcf}$ denote the image of the 
 composition of homotopies $$ H'_{bf} \circ H'_{curvesf}.$$
 
  Our goal in this section runs parallel to \cite[\S 11.5]{HL}.
 We construct a homotopy
 $H'_{\Gamma f}$ on a subset of $\Upsilon'_{bcf}$
and
 homotopies $H'_{inff}$ and $H'_{inff-primary}$
by applying Lemma \ref{inflation for families}
 such that the composition
 $$H'_{\Gamma f} \circ H'_{bf} \circ H'_{curvesf} \circ H'_{inff} \circ H'_{inff-primary}$$
 is a well 
 defined \emph{deformation retraction} i.e. it fixes its image. 
 
An important fact to note is that we will no longer shrink the base $U$
to adapt the proofs in \cite[\S 11.5]{HL} to our relative setting.
We construct the homotopy $H'_{\Gamma f}$ such that its image outside of $\widehat{D'_{1U}}$
will be controlled completely by definable functions in $\Gamma$.
This enables us to choose the inflation homotopies $H'_{inff}$ and $H'_{inff-primary}$ so that 
they fix the image of $H'_{\Gamma f}$ as well.  
Furthermore, $H'_{\Gamma f}$ restricts to well defined homotopies along the 
 fibres of $\widehat{\phi_{U}}_{|\widehat{\Upsilon'_{bcf}}}$. It is hence
  compatible with $\widehat{\phi_{U}}_{|\widehat{\Upsilon'_{bcf}}}$.

 \subsubsection{Preliminaries}
 
    By our induction hypothesis, $H'_{bf}$ descends to a deformation retraction
    $H_{b} \colon I_{bf} \times \widehat{U} \to \widehat{U}$ whose image is a $\Gamma$-internal iso-definable set which we shall denote 
    $\Upsilon_b$. Furthermore, 
    $H_b$ satisfies properties (5) and (6) of 
    Theorem \ref{generic cdr}.
    For ease of notation,
    until the end of the proof of Lemma \ref{tropical homotopy part 1}, 
     we write $\Upsilon'$ in place of $\Upsilon'_{bcf}$ and 
    $\Upsilon$ in place of $\Upsilon_b$.

     We now choose continuous injective maps on $\Upsilon'$ 
     and $\Upsilon$ into the value group sort and construct a homotopy
     $H'^{\mathrm{trop}}_{\Gamma f}$ on the image of $\Upsilon'$ for this map. 
     However, since we are making use of two inflation homotopies (cf. Remark \ref{rem : primary inflation remark}),
      where 
     $H'_{inff-primary}$ acts on $V'_{U}$ and not $V'_{1U}$, we must take 
     this into consideration when choosing coordinates in the value group sort for
     $\Upsilon'$. 
     
     Recall the closed subsets 
     $Z'_{1U} \subset V'_{1U}$, $Z'_U \subset V'_U$ such that for every fibre over $u \in U$, 
    $Z'_{1u}$ is the exceptional divisor of the map $V'_{1u} \to V'_u$ and
    $V'_{1U} \smallsetminus Z'_{1U}$ is isomorphic to $V'_U \smallsetminus Z'_U$.
    We identify $V'_U \smallsetminus Z'_U$ with the subset $V'_{1U} \smallsetminus Z'_{1U} \subset V'_{1U}$.
    If 
    $\Upsilon'_1 := \Upsilon' \smallsetminus \widehat{Z'_{1U}}$ and
     $\Upsilon'_2 := \Upsilon' \cap \widehat{Z'_{1U}}$
    then 
    we have a decomposition,
    $\Upsilon' = \Upsilon'_1 \sqcup \Upsilon'_2$. 
    Observe that $\Upsilon'_1 \subset \widehat{V'_U}$. 
    By \cite[Theorem 6.2.8]{HL}, 
    there exists a 
    $K$-definable map 
    $\alpha_0 \colon V'_U \to \Gamma_\infty^{N_1}$ such that 
    $\widehat{\alpha_0}$ is continuous and its restriction to $\Upsilon'_1$ is injective.
%
%
 Similarly, there exists a 
    $K$-definable map 
    $\alpha_1 \colon V'_{1U} \to \Gamma_\infty^{N_2}$ such that 
    $\widehat{\alpha_1}$ is continuous and its restriction to $\Upsilon'_2$ is injective.
    We abuse notation and write $\alpha_0$ and $\alpha_1$ in place of $\widehat{\alpha_0}$ and 
    $\widehat{\alpha_1}$ respectively. 
    Let $\alpha'_0$ denote the 
    composition $\widehat{V'_{1U}} \xrightarrow{\widehat{b_U}} \widehat{V'_U} \xrightarrow{{\alpha_0}} \Gamma_\infty^{N_1}$.

      Once again, by 
      \cite[Theorem 6.2.8]{HL}, there exists a
      $K$-definable map $\alpha \colon U \to \Gamma_\infty^M$ such that the induced map 
      $\widehat{\alpha} \colon \widehat{U} \to \Gamma_\infty^M$ is continuous and restricts to an 
      injective map on $\Upsilon$. 
      We abuse notation and use $\alpha$ itself in place of $\widehat{\alpha}$. 
 The proof of \cite[Theorem 6.2.8]{HL} shows that
       if $x \colon \Gamma_\infty^{M} \to \Gamma_\infty$ is 
      a coordinate then 
      $(x \circ \alpha)^{-1}(\{z \in \Gamma_\infty^{M} | x(z) = \infty \})$
      is of the form $\widehat{Z}$ where $Z$ is some Zariski closed subset of $U$. 
      Since $U$ is irreducible, we can assume that for every coordinate
      $x$, $x \circ \alpha$ is not identically $\infty$ on $\widehat{U}$. 
       Hence the locus of points $u \in U$ such that 
       $x \circ \alpha (u) < \infty$ for every coordinate $x$ on $\Gamma_\infty^M$
       is a Zariski open dense subset of $U$.        
       By property (5), we deduce
       that we have a continuous
      injective definable map $\alpha \colon \Upsilon \to \Gamma^M$. 
                 
      As in \cite[\S 11.5]{HL}, we may assume that $G$ acts on the 
      coordinates of $\Gamma_\infty^{N_1}$ and on the coordinates of 
      $\Gamma_\infty^{N_2}$
       such that the map $x \mapsto (\alpha'_0(x),\alpha_1(x))$ is $G$-equivariant. 
      We simplify notation and write $f := \widehat{\phi'_{U}}_{|\Upsilon'}$
       and $N := N_1 + N_2$. 
      We define $\alpha' \colon \Upsilon' \hookrightarrow \Gamma_\infty^{N+M}$
      by $\alpha'(x) := (\alpha'_0(x),\alpha_1(x),\alpha(f(x)))$. 
      Observe that $\alpha'$ is a well defined, injective, continuous $K$-definable map such that
            $$f' \circ \alpha' = \alpha \circ f$$
      where $f'$ is the projection $\Gamma_{\infty}^{N+M} \to \Gamma_\infty^M$. 
           We abuse notation and write $\xi_i$ for the functions $\xi_i \circ \alpha'^{-1}$. 
           
                 Let $W'$ and $W$ denote the images of $\Upsilon'$ and $\Upsilon$ via the maps $\alpha'$ and $\alpha$ respectively. 
       Note that the action of $G$ on $W'$ restricts to a well defined action on the fibres of 
      $f'_{|W'}$. 
     Let $\rho_0 \colon V' \to [0,\infty]$ denote the schematic distance from
      the closure of $D'_{U}$ in $V'$.
 Let 
      $\eta$ and $\rho$
      be as in the proof of Lemma \ref{getting around compactness}.
           
     Let $d_{bord} \colon V \to \Gamma_{\infty}$ denote the schematic distance to the 
     closed subset $V \smallsetminus U$. 
      After modifying the injection $\alpha'$, we can assume that 
      there exists coordinates 
      $x_h$,$x_{hh}$,$x_v$ on $\Gamma_{\infty}^N$ 
      and $x_c$ on $\Gamma_{\infty}^M$
      such that $x_{hh} \circ \alpha'$,$x_h \circ \alpha'$, $x_v \circ \alpha'$ and $x_c \circ \alpha$ correspond  
      to $\rho$, $\rho_0 \circ b'_U$, $\eta$ and 
      $d_{bord}$ respectively where $b'_U$ is the map $V'_{1U} \to V_U$.     
      
     \begin{lem} \label{W is closed} 
      The space $W' \cap \Gamma^{N+M}$ is closed in $\Gamma^{N+M}$.       
      \end{lem} 
     \begin{proof}
     Let $x \in \Gamma^{N+M}$ be a limit point of the set $W' \cap \Gamma^{N+M}$
     and $\gamma := x_c(x)$.
     We show that $x \in W'$. 
     This is a consequence of the fact that $W' \cap [x_c \leq \gamma]$
     is $\sigma$-compact with respect to $x_{hh}$ and $x_v$. 
     \end{proof} 
      
   \begin{lem} \label{tropical homotopy part 1} 
   Let $$W'_1 := (W' \cap \Gamma^{N+M}) \bigcup [x_h = \infty].$$
   There exists a $K$-definable deformation retraction 
   $$H'^{\mathrm{trop}}_{\Gamma f} \colon [0,\infty] \times W'_1 \to W'_1$$
    which satisfies the following properties. 
    \begin{enumerate}
    \item The deformation $H'^{\mathrm{trop}}_{\Gamma f}$ leaves the
    functions
     $\xi_i$ invariant, is $G$-equivariant and 
    preserves the fibres of the morphism $f'$. 
    \item Let $W'_0$ denote the image of the deformation retraction $H'^{\mathrm{trop}}_{\Gamma f}$. There exists a
    $K$-definable open subset $W'_o$ of $W'$ 
    that contains $W'_0 \smallsetminus [x_h = \infty]$ and $m \in \mathbb{N}$ and $c \in \Gamma(K)$ 
    such that $x_i \leq (m+1)x_h + c$ for every $1 \leq i \leq N$.
    \item For every $x \in W$, if $W'_x$ is pure of dimension $n$ then 
    for every $x$, $W'_{0x}$ is pure of dimension $n$.
            \end{enumerate}
   \end{lem} 
  \begin{proof}
   We adapt the proof of \cite[11.5.1]{HL} to our setting. Before doing so, we give a rough outline of
   this proof. 
   The specific details can be found in loc.cit.
    This is the case when $\Upsilon$ is a point and $M = 0$. 
   In this situation, we choose a $G$-invariant $K$-definable cell decomposition 
   $\mathcal{D}$ 
   of $\Gamma^{N}$
   that respects the definable set
   $W' \cap \Gamma^{N}$ as well as all sets of the form $[x_a = x_b]$ or $[x_a = 0]$
   for all coordinates $x_a,x_b$. 
   We define the restriction of the deformation retraction $H'^{\mathrm{trop}}_{\Gamma f}$ to $\Gamma^{N}$ 
   by specifying its behaviour on each cell of $\mathcal{D}$. 
   Let $\mathcal{D}_0$ denote the sub collection of cells $C \in \mathcal{D}$ 
   such that every coordinate $x_i$ is \emph{$h$-bounded
    on $C$}. By this we mean that there exists $m \in \mathbb{N}$, $c \in \Gamma(K)$ and 
   $x_i(z) \leq mx_h(z) + c$ for every $z \in C$.   
   Furthermore, when every coordinate is $h$-bounded on $C$, we shall say that the 
   cell itself is $h$-bounded.
    
   The homotopy $H'^{\mathrm{trop}}_{\Gamma f}$ will fix every point in every element of $\mathcal{D}_0$.
   Let $C \notin \mathcal{D}_0$. 
   We construct an element $e_C \in \mathbb{Q}_{+}^N$ such that if a coordinate $x_i$
    is $h$-bounded on $C$ then $x_i(e_C) = 0$. 
    For $x \in C$ and $t \in [0,\infty)$, let 
    $H'^{\mathrm{trop}}_{\Gamma f}(t,x) := x - te_C$. 
   We verify that the definable function $H_\Gamma$ is continuous on $[0,\infty] \times C$ 
   by induction on the dimension of the cell $C$.
   An important observation which makes this possible is the following. 
    Since $x_i(te_C) \geq 0$ for $t \in [0,\infty)$ and any coordinate $x_i$, we must have that 
 $x - te_C \notin C$ for some $t$. Let $\tau(x)$ be the smallest 
    such $t$. Observe that $H'^{\mathrm{trop}}_{\Gamma f}(\tau(x),x)$ must lie in a lower dimensional cell. 
    It follows that the path $t \mapsto H'^{\mathrm{trop}}_{\Gamma f}(t,x)$ 
    begins at $x$ when $t = 0$ and traverses finitely many cells in decreasing dimensions till
    it finally ends up in $\mathcal{D}_0$. 
    
     The proof in the relative case follows the same steps as in the sketch above. 
     We highlight only those points which require more than just the obvious adaptation. 
     \\ 
     
    \noindent \textbf{Step 1.} \emph{Preliminaries.} 
    \\
    
      As in loc.cit., let $\mathbf{A}$ denote the convex subgroup of $\Gamma(\mathbb{U})$ which is generated
      by $\Gamma(K)$
       and define $B := \Gamma(\mathbb{U})/\mathbf{A}$. Given
        a definable subset $X \subset \Gamma^t$ for
      some $t \in \mathbb{N}$, we define $\beta X$ to be the image of $X$ in $B^t$. 
    We choose a cell decomposition $\mathcal{D}$
    of $\Gamma^{N+M}$ which respects the definable set $W' \cap \Gamma^{N+M}$, 
    all sets of the form $[x_a = x_b]$ and $[x_a = 0]$ for all coordinates $x_a,x_b$ and 
    is such that
    its push forward $f'(\mathcal{D})$ is a cell decomposition of $\Gamma^M$ which 
    respects $W \cap \Gamma^M$. By definition, for every $C' \in \mathcal{D}$, its image in
    $\Gamma^M$ for the projection $f'$ is a cell $C$. 
     By \cite[Proposition 3.3.5]{LVD}, if $a \in C$ then the fibre $C'_a$ is a cell in 
     $\Gamma^N \times \{a\}$. 
     We can suppose that the decomposition $\mathcal{D}$ is such that for every 
     $C \in \mathcal{D}$, $C$ is $h$-bounded iff 
     for every $a \in f'(C)$, $C_a$ is $h$-bounded which in turn will be 
     equivalent to saying that there exists $a \in f'(C)$ such that $C_a$ 
     is $h$-bounded.
     Indeed, we modify the existing decomposition $\mathcal{D}$
     so that it has the above property as follows.
     Let $C' \in \mathcal{D}$ and $C := f'(C')$. 
     The set $C_{h-bdd}$ of those points $x \in C$ such that 
     $C'_x$ is $h$-bounded is definable. 
     This is a consequence of the 
     cell decomposition. 
     We can now refine the decomposition $\mathcal{D}$ so that it respects the preimages of the 
     definable sets $C_{h-bdd}$ and $C \smallsetminus C_{h-bdd}$ for every $C$.
     
     For a cell $C' \in \mathcal{D}$, we now define
      the point $e_{C'}$. 
      Let $\mathfrak{M}$ denote the last $M$ coordinates of $\Gamma^{N+M}$ and 
      $\mathfrak{N}$ denote the first $N$ coordinates. 
      Let $$\beta'C' := \beta C' \bigcap [x_h = 0] \bigcap [x_i = 0]_{i \in \mathfrak{M}}$$
      and let $e_{C'}$ denote the barycentre of $\beta' C' \cap [\sum_i x_i = 1]$.
       By construction, if $x_i \in \mathfrak{M}$ or 
        if $x_i$ is $h$-bounded on $C'$ then $x_i(e_{C'}) = 0$. 
         It follows that if $C' \in \mathcal{D}_0$ then 
        $e_{C'} = 0$
        Also, if $x_i$ is not $h$-bounded on $C'$ then 
        $x_i(e_{C'}) > 0$. 
        
       For $x \in C'$, we define $H'^{\mathrm{trop}}_{\Gamma f}(x) := x - te_{C'}$. 
       Clearly, $H'^{\mathrm{trop}}_{\Gamma f}$ is well defined. Furthermore, 
       if $C = f'(C')$ then 
       by construction, we see that 
       for every $x \in C$, the homotopy $H'^{\mathrm{trop}}_{\Gamma f}$ restricts 
      to a well defined homotopy along the fibre over $x$. 
        \\ 
     
    \noindent \textbf{Step 2.} \emph{Continuity and end of the proof.} 
    \\
    
    The continuity of $H'^{\mathrm{trop}}_{\Gamma f}$
     and assertion (2) of the Lemma 
     can be shown by following the arguments in \cite[Lemma 11.5.1]{HL}  with little change.
     Note that the homotopy $H'^{\mathrm{trop}}_{\Gamma f}$ 
     preserves the closures of the cells. By 
     Lemma \ref{W is closed}, $H'^{\mathrm{trop}}_{\Gamma f}$
     preserves $W' \cap \Gamma^{N+M}$.
    Assertion (3) can be verified using the argument in loc.cit., 
    and considering the hyperplane $L \colon \sum_{i \in \mathfrak{N}} x_i = Mx_h + K$ where 
    $M = N(m + 1)$ and $K = Nc$.

%
     
  \end{proof} 
 
 It remains to treat the case when for some $x \in W$, $W'_x \cap (\Gamma^{N} \times \{x\})$ is empty. 
 We attempt a relative version of \cite[Lemma 11.5.2]{HL}.

%
 
 Let $\eta$ denote the generic point of $U$ and $V'_{1\eta}$ be the 
 generic fibre of the morphism $\phi'_U$. 
 Recall that we assumed that every 
 generic point of every irreducible component of $V'_{1U}$ 
 is contained in the fibre over $\eta$. 
 Let $V'_{10\eta},\ldots,V'_{1r\eta}$ denote the irreducible components of 
 $V'_{1\eta}$ and for every $j$, 
 we set $V'_{1j}$ to be the Zariski closure of $V'_{1j\eta}$ in $V'_{1U}$. 
 By construction, we have that the $V'_{1j\eta}$ are equidimensional and 
 there exists a Zariski open subset $U_0 \subseteq U$ such that
  the maps
 $V'_{1j} \to U_0$ are flat. By \cite[Corollary 9.6]{hart}, we get that the fibres of
 $V'_{1j}$ are pure of dimension $\mathrm{dim}(V'_{1j\eta})$. 
 Recall by construction that $D'_{1U}$ contains the intersections 
 $V'_{1j} \cap V'_{1j'}$ for every $j,j'$ such that $j \neq j'$. 
 For every $j$, 
 let $W'_j$ be the image of $\Upsilon'_{bcf} \cap [\widehat{V'_{1j}} \smallsetminus \bigcup_{j' \neq j}  \widehat{V'_{1j'}}]$
for the map $\alpha'$. 
 Observe that $W'_j$ is open and 
  $$W' = \bigsqcup_j W'_j \bigsqcup ([x_h = \infty] \cap W').$$
  For every $x \in W$, $W'_{jx}$ is pure of dimension 
  $\mathrm{dim}(V'_{1jx})$.
  
  Let $x_1,\ldots,x_N$ denote the first $N$ coordinates of $\Gamma_\infty^{N+M}$.
  Recall that the preimage of the locus $[x_i = \infty]$ via the function 
  $x_i \circ \alpha'$ is of the form $\widehat{Z_i}$ where $Z_i$ is a Zariski closed subset 
  of $V'_{1U}$. 
  
  Let $j \in \{1,\ldots,r\}$. 
  Let $\mathfrak{P}_j \subset \{1,\ldots,N\}$ be the subset of indices such that 
  for $p \in \mathfrak{P}_j$, 
  $Z_p \cap V'_{1j\eta}$ is of dimension strictly less than $\mathrm{dim}(V'_{1j\eta})$.
  We shrink $U_0$ further to get that for every such $p \in \{1,\ldots,N\}$,
  $Z_p \cap V'_{1j}$ is flat over $U_0$ or empty. It follows that for every 
  $x \in U_0$ and $p \in \mathfrak{P}_j$, 
  $Z_p \cap V'_{1jx}$ is of dimension strictly less than $\mathrm{dim}(V'_{1j\eta})$. 
  
   Since for every $x \in W$, $W'_{jx}$ is pure of dimension 
  $\mathrm{dim}(V'_{1jx})$, we deduce that for every $p \in \mathfrak{P}_j$, 
  $x_p$ is not identically $\infty$ on $W'_{jx}$. 
  Indeed, our choice of $p$ implies that 
  $(V'_{1j} \smallsetminus (\bigcup_{j \neq j'} V'_{1j'}  \bigcup Z_p)) \bigcap V'_{1U_0}$
  is a non-empty Zariski open subset of $V'_{1U}$. 
  By construction, if 
  $A := (V'_{1j} \smallsetminus (\bigcup_{j \neq j'} V'_{1j'}  \bigcup Z_p))  \bigcap V'_{1U_0}$ then
  $\phi'_U(A) = U_0$ and hence $\widehat{\phi'_U}(\widehat{A}) = \widehat{U_0}$. 
  Note that if $x' \in \Upsilon_b$ is such that $x' \mapsto x$
  then 
  the homotopies $H'_{bf}$ and $H'_{curves f}$ restrict 
  to well defined homotopies on $\widehat{\phi'_U}^{-1}(x')$. 
  As $H_b$ satisfies property (5) of Theorem \ref{generic cdr}, 
  we see that $x' \in \widehat{U_0}$.
  Since $H'_{bf}$ and $H'_{curves f}$ are also Zariski generalizing, we must have that 
  $\Upsilon_{bcf} \cap (\widehat{\phi'_U})^{-1}(x')$ intersects $V'_{1j} \smallsetminus (\bigcup_{j \neq j'} V'_{1j'}  \bigcup Z_p)$
  non-trivially.
   It follows that $x_p$ is not identically $\infty$ on 
  $W'_{jx}$.
  
       Let $n_j$ be the cardinality of the set $\mathfrak{P}_j$.     
       Consider the embedding $\alpha'_j \colon W'_{j} \to \Gamma_\infty^{n_j + M}$
       given by 
       $z \mapsto ((x_i(z))_{i \in \mathfrak{P}_j},(x_i(z))_{i \in \mathfrak{M}})$.
       Note that ${\alpha'_j}$ is a homeomorphism onto its image. 
       Let $W'^\circ_{j}$ be the preimage of the open set
       $\alpha'_j(W'_{j}) \cap \Gamma^{n_j + M}$. Observe that 
       $W'^\circ_{j}$ is $z$-open and $z$-dense in $W'_{j}$. 
       Let $H'^{\mathrm{trop}}_{\Gamma j}$
       be the deformation retraction induced on $W'^\circ_{j}$ from
        Lemma \ref{tropical homotopy part 1}. 
        Let 
      $W^\circ := \bigcup_{j} W'^\circ_j$
      and $W'' := (W^\circ \smallsetminus [x_v = \infty]) \cup [x_h = \infty]$.
      
       Let $$H'^{\mathrm{trop}}_{\Gamma f} \colon [0,\infty] \times W'' \to W''$$ denote the deformation retraction 
       whose restriction to $W'^\circ_j$ for every $j$ is induced by
       $H'^{\mathrm{trop}}_{\Gamma j}$. 
       Let $W'_0$ denote the image of $H'^{\mathrm{trop}}_{\Gamma f}$.
 Observe that $H'^{\mathrm{trop}}_{\Gamma f}$
 satisfies the following properties.  
      \begin{enumerate}
 \item The deformation $H'^{\mathrm{trop}}_{\Gamma f}$ leaves the $\xi_i$ invariant, is $G$-equivariant and 
    preserves the fibres of the morphism $f'$. 
    \item There exists a
    $K$-definable open subset $W'_o$ of $W'$ 
    that contains $W'_0 \smallsetminus [x_h = \infty]$, $m \in \mathbb{N}$ and $c \in \Gamma(K)$ 
    such that    
     on $W'_o \cap W'_{j}$, $x_i \leq (m+1)x_h + c$ for every $i \in \mathfrak{P}_j$.
  \end{enumerate} 
 
\subsection{Completing the proof of Theorem \ref{generic cdr}} \label{completing the proof}

        In order to complete the proof of Theorem \ref{generic cdr}, 
        we 
        choose inflation homotopies $H'_{inff}$
        and $H'_{inff-primary}$ 
        using 
        Lemma \ref{inflation for families} 
        such that the image of the composition
        $H'_{bf} \circ H'_{curvesf} \circ H'_{inff} \circ H'_{inff-primary}$
        is contained in a subspace $P$ of $\Upsilon'_{bcf}$ such that 
        $\alpha'_{|P}$ is a homeomorphism. 
        We can then define a homotopy $H'_{\Gamma f}$
        on $P$ via the tropical homotopy $H'^{\mathrm{trop}}_{\Gamma f}$
        so that the composition
        $$H'_{\Gamma f} \circ H'_{bf} \circ H'_{curvesf} \circ H'_{inff} \circ H'_{inff-primary}$$
        is well defined and fixes its image. 
        Here we abuse notation and write $H'_{\Gamma f}$ for the homotopy on $P$.
        
        For every 
        $i \in \mathfrak{N}$, 
        let $y_i := \mathrm{min}\{x_i, (m+1)x_h + c\}$.  
        Observe that for every $i$, 
        $(y_i \circ \alpha')^{-1}(\infty)$ is contained in $\widehat{D'_{1U}}$ and hence also in $\widehat{D'_{11U}}$. 
        When there is no ambiguity, we simplify notation and write 
        $y_i$ in place of the composition $y_i \circ \alpha'$. 
        
        We now construct the homotopy 
        $H'_{inff}$ using Lemma \ref{inflation for families}. 
        We deduce from (1) of \S \ref{section : initial set-up} that for some copy of $\mathbb{A}^m \subset E$, the 
        morphism $g'_U$
        restricts to a map 
        $V'_{1U} \smallsetminus D'_{11U} \to \mathbb{A}^m \times U$. 
        By our choice of $D'_{11U}$ in \S \ref{section : initial set-up} and 
        Lemma \ref{no need for etale}, this map satisfies requirement (2) of 
        Lemma \ref{inflation for families}. 
        Let $$H'_{inff} \colon [0,\infty] \times \widehat{V'_{1U}} \to \widehat{V'_{1U}}$$ be the deformation retraction constructed in 
        Lemma \ref{inflation for families} with respect to the morphisms $\phi'_{U}$ and 
        ${g'_U}_{|V'_{1U} \smallsetminus D'_{11U}}$, the divisor $D'_{11U}$, the family 
        $\{\xi_i\}_i \cup \{y_j \circ \alpha'\}_j$ and the group $G$. 
      Note that $H'_{inff}$
       is compatible with the 
      morphism $\phi'_U \colon V'_{1U} \to U$.
      Furthermore, $H'_{bf} \circ H'_{curvesf} \circ H'_{inff}$ is a well defined homotopy
      which respects the levels of the functions $\xi_i$ and is equivariant for the action of the group $G$.
      Note that at this stage we cannot say that 
      the image of 
      $H'_{bf} \circ H'_{curvesf} \circ H'_{inff}$ will
      be contained in $W''$ - the domain of the homotopy 
      $H'_{\Gamma f}$.
            
      Let $Z'_U \subset V'_U$ be as introduced in \S \ref{section : initial set-up}. 
      Recall that by definition, 
      $\rho_0 \colon V'_U \to \Gamma_\infty$ was defined to be the schematic distance to $D'_U$
      and $x_h \circ \alpha' = \rho_0 \circ b_U$.  
      By construction $V'_U \smallsetminus Z'_U \subset V'_{1U} \smallsetminus Z'_{1U}$.  
      Let $v_i \colon V'_U \to \Gamma_\infty$ be 
      the coordinates of the map 
      $\alpha_0 \colon \widehat{V'_U} \to \Gamma_\infty^{N_1}$ i.e.
      $v_i$ is the composition of $\alpha_0$ and the $i$-th projection 
        $\Gamma_\infty^{N_1} \to \Gamma_\infty$.    
        For every $i$, let $w_i := \mathrm{min}\{v_i, (m+1)\rho_0 + c\}$.
        Observe that $\rho_0^{-1}(\infty) = D'_U$. 
      Let $H'_{inff-primary}$ be the homotopy 
      as provided by Lemma \ref{inflation for families}
      for the morphism $V'_U \to U$, 
      the morphism $g_U \colon V'_U \to \mathbb{P}^m \times U$, the divisor 
      $D'_U$, the functions $\{w_j\}_j \cup \{\xi_i\}_i$ and the group $G$. 
      
      Observe that the composition 
      $H'_{bf} \circ H'_{curvesf} \circ H'_{inff} \circ H'_{inff-primary}$ is well defined. 
      Let $\beta \colon \widehat{V'_U} \to \Upsilon'$ be the retraction associated to 
      this composition and $B := \beta(\widehat{V'_{U}})$. 
      
      \begin{lem} \label{lem : B is homeomorphic under alpha'}
      The morphism $\alpha' \colon \Upsilon'_{bcf} \to W'$ restricts to a homeomorphism 
      from $\beta(\widehat{V'_U})$ onto its image.
      \end{lem}
      \begin{proof} 
          The morphism $\alpha'$ is injective and continuous. It suffices to show that 
       the restriction $\alpha'_{|B}$ is a closed map from 
       $B$ to $\alpha'(B)$. 
       Note that $\alpha'(B)$ is an iso-definable subset of $\Upsilon'_{bcf}$.
       This is a consequence of the fact that $\beta(V'_U) \subset \Upsilon'_{bcf}$ is 
       iso-definable and $\Gamma$-internal and hence $\beta(V'_U) = \beta(\widehat{V'_U})$.
       Let $Z \subset B$ be an iso-definable closed subset. 
       We claim that $\alpha'(Z)$ is closed in $\alpha'(B)$. 
       Let $a$ belong to the closure of $\alpha'(Z)$ in $\alpha'(B)$. 
       By \cite[Proposition 4.2.13]{HL}, there exists a definable type $q$ that concentrates on 
       $\alpha'(Z)$ with limit point $a$. 
       Let $a' \in B$ be the preimage of $a$
        and $q'$ be the preimage of the definable type $q$.
      We show that $a'$ is a limit point of $q'$. 
      Recall the coordinate $x_c$ on $\Gamma_\infty^{M}$ which corresponds 
      to the schematic distance $d_{bord} \colon V \to \Gamma_\infty$ 
      to $V \smallsetminus U$. 
      Let $\lambda \in \Gamma$ be such that 
      $x_c \circ f'(a) < \lambda$ where $f'$ is the projection 
      $\Gamma_\infty^{N+M} \to \Gamma_{\infty}^M$. 
      Note that 
      if $U_\lambda := \{x \in V | d_{bord}(x) \leq \lambda\}$ and  
      $V'_{1U_\lambda}  := ({\phi'_U})^{-1}(U_\lambda)$
      then $\widehat{V'_{1U_\lambda}}$ 
      is definably compact. 
      It follows that 
      $\beta(\widehat{V'_{1U_\lambda}}) \subset B$ is definably compact and 
      hence 
      $\alpha'$ restricts to a homeomorphism 
      from $\beta(\widehat{V'_{1U_\lambda}})$ 
      onto its image.
      Since $H_b$ preserves $d_{bord}$, 
      we deduce that 
      $B_\lambda := B \cap \widehat{V'_{1U_\lambda}} = \beta(\widehat{V'_{1U_\lambda}})$.
      By construction, 
      we have that $f' \circ \alpha' = \alpha \circ f$ where 
      $f = \widehat{\phi'_U}_{|\Upsilon_{bcf}}$. 
      Hence, $\alpha'$ restricts to a homeomorphism 
      from $B_\lambda$ onto $[x_c \leq \lambda] \cap \alpha'(B)$. 
      Note that $B_\lambda$ and 
      $[x_c \leq \lambda] \cap \alpha'(B)$ are closed
      neighbourhoods of $a'$ and $a$ respectively. 
      It follows that $a'$ is a limit point of $q'$. 
      Since $Z$ is closed, $a' \in Z$. 
      It follows that $a \in \alpha'(Z)$.
      This concludes the proof. 
%
%
%
             
      \end{proof} 
      
      We define 
      $H'_{\Gamma f}$ as follows. 
      Observe firstly that 
      $\alpha'(B) \subset W''$. 
      This is a consequence
      of the inflation properties of 
      $H'_{inff}$ and $H'_{inff-primary}$. Firstly, 
      note that the image of $H'_{inff-primary}$ is 
      contained in $\widehat{V'_{1U}} \smallsetminus \widehat{Z'_{1U}}$ which in turn 
      implies that the image of $H_{inff} \circ H'_{inff-primary}$ 
      is contained $\widehat{V'_{1U}} \smallsetminus \widehat{Z'_{1U}}$. 
      Furthermore,  
      for any $u \in U$, 
      if $Z \subset V'_{u}$ is a closed sub variety of 
      dimension strictly smaller than $\mathrm{dim}(V'_{u})$
      then the intersection of $\widehat{Z}$
      with the image of $H'_{inff} \circ H'_{inff-primary}$
      will be contained in 
      $\widehat{D'_{1u}}$. 
      Let $x \in \beta(\widehat{V'_{U}})$ 
      and $t \in [0,\infty]$. 
      We set $H'_{\Gamma f}(t,x) := {\alpha'}^{-1}(H'^{\mathrm{trop}}(t,\alpha'(x))$. 
      By Lemma \ref{lem : B is homeomorphic under alpha'}, 
      $H'_{\Gamma f}$ is a well defined homotopy. 
       We have thus shown that 
       $H'_{\Gamma f} \circ H'_{bf} \circ H'_{curvesf} \circ H'_{inff} \circ H'_{inff-primary}$
      is a well defined homotopy.
      
      We now show that 
      $H'_{\Gamma f} \circ H'_{bf} \circ H'_{curvesf} \circ H'_{inff} \circ H'_{inff-primary}$
      fixes its image. 
      From the construction, we see that it suffices 
      to verify 
      that 
      $H'_{inff}$ and 
      $H'_{inff-primary}$ fix this image.
      Recall that we used $W'_0$ to denote the 
      image of $H'^{\mathrm{trop}}_{\Gamma f}$.
      Let $\Upsilon'_0$ denote the image of 
      $H'_{\Gamma f}$. 
      Let $w \in \Upsilon_b \subset \widehat{U}$. 
      We check that $H'_{inff}$ and $H'_{inff-primary}$ 
      both fix $W'_{0w}$. 
      To do so we borrow the notation from the previous section. 
      Recall that we have a decomposition 
       $$W' = \bigsqcup_j W'_j \bigsqcup ([x_h = \infty] \cap W').$$
       On $W'_j \cap W'_0$, 
       for every $i \in \mathfrak{P}_j$, we must have that 
       $y_i = x_i$.
       Furthermore, if $i \notin \mathfrak{P}_j$
       then $x_i$ is identically $\infty$ on 
       $V'_{1j}$.
       Since $H'_{inff}$ respects the levels of the functions $y_i$ for all $i$,
       by \cite[Lemma 8.3.1(2)]{HL}, 
       it fixes $W'_j \cap W'_0$.  
     
       We now show that $H'_{inff-primary}$ fixes the image of the composition 
       $H'_{\Gamma f} \circ H'_{bf} \circ H'_{curvesf} \circ H'_{inff} \circ H'_{inff-primary}$.
       The argument is identical to the one made above but makes use of the fact 
       that the image of the composition must lie in the complement of 
       $Z'_{1U}$ and is hence 
       controlled completely by the coordinates $v_i$.
       
       One observes that each of the homotopies in the composition
       respects the levels of the functions $\xi_i$ and the action of the group $G$. 
       Furthermore, the induction hypothesis that gives rise 
       to $H'_{bf}$ and (3) of Lemma \ref{tropical homotopy part 1} guarantee 
       that for every $z \in \Upsilon_b$, the fibre over $z$ is pure of dimension $n$ 
       where $n$ is the dimension of the generic fibre for the map $V'_{1U} \to U$. 
        We now show how to ensure that the composition is Zariski generalizing. 
        The argument is similar to 
       the one that appears in \cite[\S 11.6]{HL}.
       By Lemma \ref{lem : B is homeomorphic under alpha'}, 
       $\alpha'$ restricts to a homeomorphism from
       $\beta(\widehat{V'_U})$ onto its image. Hence, 
       by \cite[Corollary 10.4.6]{HL}, it suffices to verify 
       that $H'^{\mathrm{trop}}_\Gamma$ is Zariski generalizing. 
       This can be done as in \S 11.6 of loc.cit.
       Lastly, our construction of $H_b$ ensures that it satisfies 
       properties (5) and (6) of Theorem \ref{generic cdr}. 
      This concludes the proof.         
%
%
%
%
%
%
      
\end{proof}

 \section{When the base is a curve} \label{base is a curve}
 
         Theorem \ref{generic cdr} asserts the existence of 
        compatible deformations generically over the base. 
        In the proof of this theorem, at several stages, 
        we shrunk the base so as to obtain that the family behaved in a tame 
        manner. 
        When the base is a curve, we do not need to 
        shrink the base constantly to obtain tame properties of the family. 
       This allows us to prove the following theorem.
       
  \begin{thm} \label{weak cdr for curves} 
    Let $S$ be a smooth connected $K$-curve and $X$ be a quasi-projective $K$-variety. 
    Let $\phi \colon X \to S$ be a surjective morphism such that every irreducible component of 
    $X$ dominates $S$. 
 Let $\{\xi_i \colon X \to \Gamma_{\infty}\}$ be a finite collection of 
   $K$-definable functions.
  Recall that the functions $\xi_i$ extend to functions $\xi_i \colon \widehat{X} \to \Gamma_\infty$.
   Let $G$ be a finite algebraic group acting on $X$ such that the action of $G$ respects the fibres 
   of the morphism $\phi$. 
   Let $s \in S(K)$. There exists a Zariski open subset $U \subset S$ containing $s$ and
    compatible homotopies $(H,\Upsilon)$ of $\widehat{U}$ and $(H',\Upsilon')$ of $\widehat{X_U}$ 
    \footnote{$X_U := X \times_S U$}
    such that the following hold. 
    \begin{enumerate}
    \item The homotopy $H$ is in fact a deformation retraction. 
    \item The images $\Upsilon \subset \widehat{U}$ and $\Upsilon' \subset \widehat{X_U}$ are $\Gamma$-internal. 
    \item The homotopy $H'$ respects the functions $\xi_i$ i.e. $\xi_i(H'(t,p)) = \xi_i(p)$ for every $p \in \widehat{X}$ and 
     $t \in I$. 
    \item  The action of the group $G$ on $X$ extends to an action on $\widehat{X}$.  
      The deformation $H'$ can be taken to be $G$-invariant i.e. for every $g \in G'$, $H'(t,g(p)) = g(H'(t,p))$. 
      \item The homotopy $H'$ is Zariski generalizing i.e. if $W \subset X_U$ is a Zariski open subset 
      then $H'$ restricts to a well defined homotopy on $\widehat{W}$. 
    \end{enumerate} 
   \end{thm}

 \subsection{Initial reductions}

\begin{rem} \label{prove the simpler case wbc}  
 \begin{enumerate} 
 \item \emph{By arguments similar to those that appear in the proof 
  of Lemma \ref{prove the simpler case} and using Lemma \ref{simplifications}, one shows that if
    Theorem
      \ref{weak cdr for curves} is true for $G$-equivariant morphisms \footnote{The action of $G$ on $S$ is assumed to be trivial.} 
       $\phi \colon X \to S$ between projective $K$-varieties whose fibres are 
       pure and where $S$ is a smooth projective connected $K$-curve 
        then the theorem is true in general. Henceforth, unless otherwise stated, we will assume $\phi \colon X \to S$ is a 
        flat morphism between projective $K$-varieties 
        such that the fibres of $\phi$ are pure and $S$ is a smooth, connected curve.  
        }  
        \item \emph{The functions $\xi_i \colon V' \to \Gamma_{\infty}$ 
can be taken to be $v + g$-continuous \cite[\S 11.2]{HL}. 
It follows by Lemma 10.4.3 in loc.cit. that for every $i$, $\xi_i^{-1}(\infty)$ 
is a sub-variety of $V'$.  For every $i$, let $Z_i := \xi_i^{-1}(\infty)$.
Since $S$ is a curve, we can 
  realize $Z_i$ as the union of a horizontal divisor $Z_{ih}$ and 
  a vertical divisor $Z_{iv}$. Indeed, $Z_{ih}$ is the closure of those irreducible components 
  of $Z_i$
  that dominate $S$ while $Z_{iv}$ is the union of those irreducible components that map to 
  a $K$-point in $S$.}
\end{enumerate}
   \end{rem}

\subsection{The inflation homotopy}

\begin{rem} \label{putting stuff together}
  
  \emph{Let $s \in S(K)$ be as in the statement of Theorem \ref{weak cdr for curves}. 
  We apply the first part of Lemma \ref{Local factorization} to obtain a 
  Zariski open affine neighbourhood $U$ of $s$ such that the
  morphism $\phi$ factors through 
  a finite $G$-equivariant morphism $f \colon X_U \to \mathbb{P}^m \times U$ 
  where $G$ acts trivially on $\mathbb{P}^m \times U$ and $m = \mathrm{dim}(X_s)$.
  Let $T \subset \mathbb{P}^m \times U$ be as given by 
  Lemma \ref{existence of D}.
  Let $H := \mathbb{P}^m \smallsetminus \mathbb{A}^m$. 
  We enlarge $T$ so that it contains $H \times U$. 
  Let $D := f^{-1}(T)$. 
  We enlarge $T$ so that $D$ contains the closed sub-varieties $Z_{ih}$ introduced above and remains $G$-invariant.}
  
       \emph{We apply the second part of Lemma \ref{Local factorization} to obtain 
       a variety $X'_U$ that fits into a commutative diagram}
  
                  $$
\begin{tikzcd}[row sep = large, column sep = large]
X'_{U} \arrow[r, "{f'}"] \arrow[d,"b'"] & E \times U \arrow[d, "b \times id"] \arrow[r,"p"] & \mathbb{P}^{m-1} \times U \\
X_U  \arrow[d,"\phi"] \arrow[r,"f"]  & \mathbb{P}^m \times U \\
U 
\end{tikzcd}
$$ 
\emph{such that 
 $b \colon E \to \mathbb{P}^m$ is the blow up at a $K$-point $z$, the restriction 
    of $(p \circ f')$ to $b'^{-1}(D)$
     is finite surjective onto $\mathbb{P}^{m-1} \times U$ and the 
    square in the diagram is cartesian.
   We define $D' \subset X'_{U}$ to be the union of 
   $b'^{-1}(D)$ and the preimage via 
   $(b \times id) \circ f'$ of $\{z\} \times \mathbb{P}^m$.  
   Let $\phi' := \phi \circ b'$. 
  Lastly, for every $i$, we set 
  $\xi'_i := \xi_i \circ b'$. 
    }
\end{rem}

\subsection{Relative curves homotopy} 

       Our goal in this section is to construct a homotopy on a 
       family of curves similar to the construction 
       in \S \ref{relative curves for families}. The difference between the results 
       presented here and those before is that we no longer have the freedom to shrink the base $U$ arbitrarily. 
       As a result, we proceed a little differently and make use crucially of the fact that 
       the base is of dimension $1$. 
       
       We adapt the construction in \cite[\S 11.3]{HL} 
       of the relative curve homotopy
         to the 
        fibration 
       $p' := p \circ f' :  X'_U  \to F \times U$ where $F = \mathbb{P}^{m-1}$.  
        As in \S \ref{relative curves for families}, there exists an affine 
       open subset $F_0 \subset F$ such that 
       $p^{-1}(F_0 \times U) \subset E \times U$ is isomorphic to $\mathbb{P}^1 \times (F_0 \times U)$.
      Let $W := F_0 \times U$, $A := {p'}^{-1}(W) \subset X'_U$ and 
       $B := p^{-1}(W) \subset E \times U$.  
     \\       
     
\noindent \emph{
      The homotopy $H'_{curvesf}$.}
\\ 

      We fix three points - $0,1$ and $\infty$ on $\mathbb{P}^1$. 
      Recall that 
      given a divisor $P$ on $B$, \cite[\S 10.2]{HL} implies 
      that
       there exists a well defined definable map
      $\psi_P \colon [0,\infty] \times B \to \widehat{B/W}$ 
      which fixes $\widehat{P}$.
      These homotopies played a crucial role in \S \ref{relative curves for families}.  
      
\begin{lem} \label{the divisor in the lift}
  There exists a divisor $P \subset B$ which satisfies the following properties. 
  \begin{enumerate} 
  \item For any divisor $P' \subset B$ that contains $P$, 
  the definable map 
  $\psi_{P'} \colon  [0,\infty] \times B \to \widehat{B/W}$ lifts uniquely to a  
  definable map $h \colon [0,\infty] \times A \to \widehat{A/W}$ such that :
  \begin{enumerate} 
  \item For every $i$, the function $h$ preserves the levels of the functions $\xi_i'$. 
  \item The function $h$ is $G$-invariant.  
  \end{enumerate} 
  \item Recall the divisor $D' \subset X'_U$ from Remark \ref{putting stuff together} which is 
  finite over $F \times U$. 
      We have that $f'(D' \cap A) \subseteq P$. 
  \item Let $W' \subset W$ be the open subspace over which the 
  divisor $P$ is finite. The restriction $W' \to U$ is surjective. 
  \end{enumerate}
\end{lem} 
\begin{proof} 
    We use arguments as in the proof of Lemma \ref{existence of D} to show that 
    there exists a constructible set $P_1 \subset B$ such that 
    for every $w \in W$, $P_{1w} := P_1 \cap B_w$ is a finite set
    of those points 
    $c \in B_w$ such that 
    $\mathbf{card}\{f'^{-1}(c)\} < \mathrm{sup}_{b \in B_w} \{\mathbf{card}\{f'^{-1}(b)\}$.
%
    We can further enlarge $P_1$ so that for every $w \in W$, 
    if $FB_w \subset B_w$ denotes the forward branching points 
    \cite[Definition 7.4.2]{HL} of the morphism $f'_w := f'_{|A_w}$ then
    the convex hull of $P_{1w}$ contains $FB_w$. 
   One sees from the proof of \cite[Lemma 11.3.1]{HL} that 
   $P_1$ can be chosen so that it continues to be constructible and finite over $W$ i.e. for every 
   $w \in W$, $P_{1w}$ is finite. 
   
     Let $\overline{P_1} \subset B$ 
    denote the Zariski closure of $P_1$ in $B$. We claim that 
    $\psi_{\overline{P_1}}$ lifts uniquely to a definable map 
    $h_{\overline{P_1}} \colon [0,\infty] \times A \to \widehat{A/W}$. 
    It suffices to show that if given $w \in W$, 
    then the restriction 
    ${\psi_{\overline{P_1}}}_{|[0,\infty] \times B_w}$ 
    lifts to a definable map
    $h_{\overline{P_1},w} \colon [0,\infty] \times A_w \to \widehat{A_w}$. 
    This can be accomplished by the arguments in 
    \cite[Proposition 7.4.6]{HL} and Lemma \ref{no need for etale}. 
    Note that 
    in \cite[Proposition 7.4.6]{HL}, one constructs a path in $\widehat{A_w}$ (which is a lift of a path in $\widehat{B_w}$)
    starting from a Zariski closed point  which is distinct from a point of ramification upto a forward branching point. 
    However, the only reason to have the hypothesis that 
     the Zariski closed point be distinct from the ramification locus
     is to use 
    Lemma 7.3.1 in loc.cit. which we accomplish by Lemma \ref{no need for etale}. 
    Let $W_1 \subset W$ be the open subspace over which 
    $\overline{P_1} \to W$ is finite. We claim that the projection $W_1 \to U$
    is surjective. If $W_1$ was not surjective then we deduce that 
    $\overline{P_1}$ must contain a subspace of the form 
    $E \times \{u\}$ for some $u \in U$ using the notation from 
    Remark \ref{putting stuff together}. 
    Observe that 
    $\mathrm{dim}(\overline{P_1}) = \mathrm{dim}(E) = m$. Hence we must have
    that $E \times \{u\}$ is an irreducible component of $\overline{P_1}$ which implies that 
    $P_1 \cap (E \times \{u\})$ is dense in $E \times \{u\}$. By construction of $P_1$, this is not possible. 
    We have thus verified the claim.       
   
     We now enlarge $\overline{P_1}$ to a divisor $P$ so that the lift $h$ 
     of $\psi_P$ respects the levels of the functions $\xi'_i$ .
     We proceed as follows.
     We enlarge $\overline{P_1}$ so that it contains the divisor $\{\infty\} \times W \subset B$.
      Observe that $B_0 := B \smallsetminus (\{\infty\} \times W) = \mathbb{A}^1 \times W$ is affine.
      It follows that $A_0 := f'^{-1}(B_0)$ is affine as well. 
      By the proof of \cite[Lemma 10.2.3]{HL}, we see that 
     for every $i$,
     there exists a finite family $\{\epsilon_{ji} \colon B \to \Gamma_\infty\}$
     of definable functions such that if a definable function 
     $[0,\infty] \times B \to \widehat{B/W}$ preserves 
     $\epsilon_{ji}$ for every $j$ then 
     any lift $[0,\infty] \times A \to \widehat{A/W}$ must preserve 
     $\xi'_i$. 
      
      Since $B_0$ is affine, 
     for every $i,j$, 
     the restriction $(\epsilon_{ij})_{|B_0}$ factorizes through functions of the form $\mathrm{val}(g)$ 
     where $g$ is a regular function on $B_0$.  
     Hence 
     there exists finitely many regular functions $g_1,\ldots,g_r$ on $B_0$ such that if 
     $Z_1,\ldots,Z_r$ denotes the zeroes of the functions and
     $P_2 = \bigcup_t Z_t$ then 
      the homotopy $\psi_{P_2}$ 
      respects the levels of the functions $\epsilon_{ij}$ for every $i,j$. 
      We modify $P_2$ slightly as follows. 
     For every $t$, 
     let $Z_{ht}$ be the union of those components of $Z_t$ which are generically finite over $W$. 
      If $W_2 \subset W$ denotes the locus over which the Zariski closure
      $\overline{Z_{ht}}$ in $B$ is finite then 
      we claim that $W_2 \to U$ is surjective.   
     Indeed, the Zariski closure $\overline{Z_{ht}}$
     is such that it cannot contain a subset of the form 
     $E \times \{u\}$ for dimension reasons. 
     Let
     $P_3 := \bigcup_t \overline{Z_{ht}}$. 
     
     Let $P = \overline{P_1} \bigcup P_3 \bigcup f'(D' \cap A)$. 
     By construction, $\psi_P$ lifts to a homotopy on $A$ and 
     respects the levels of the functions $\xi'_i$. 
     To conclude a proof of the lemma, we must 
     show that the lift $h$ is $G$-invariant. This follows from the uniqueness of the lift $h$. 
 
\end{proof}
 
  Using Lemma \ref{no need for etale}, 
  we deduce that our choice of $D' \subset X'_U$ implies that 
 we can apply Lemma \ref{inflation for families} 
 to the morphism $X'_U \to U$, the map
 $X'_U \smallsetminus D' \to \mathbb{A}^m \times U$,
 the functions $\xi_i \colon X'_U \smallsetminus D' \to \Gamma_\infty$ and the 
 group $G$.

%

%
\begin{lem} \label{curves for families wbc}
   There exists a constructible set $C \subset X'_{U}$ with the following properties. 
   \begin{enumerate} 
   \item The set $C$ maps surjectively onto $\mathbb{P}^{m-1} \times U$ via $p'$
   \item  Let $H'_{inff}$ be the homotopy obtained by applying
    Lemma \ref{inflation for families} 
 to the morphism $X'_U \to U$, the map
 $X'_U \smallsetminus D' \to \mathbb{A}^m \times U$,
 the functions $\xi_i \colon X'_U \smallsetminus D' \to \Gamma_\infty$ and the 
 group $G$.
   The image $H'_{inff}(e,\widehat{X'_{U}})$ is contained in $\widehat{C}$.  
   \item  Let $I_2 := [0,\infty]$. 
    There exists a $v+$g-continuous
   homotopy $h'_{curvesf} \colon I_2 \times C \to \widehat{C/\mathbb{P}^m \times U}$. 
   The image $\Upsilon'_2 := h'_{curvesf}(0,C) \subset \widehat{X'_{U}/\mathbb{P}^{m-1} \times U}$ is 
   iso-definable and relatively $\Gamma$-internal over $\mathbb{P}^{m-1} \times U$. 
   \end{enumerate}    
   \end{lem} 
   \begin{proof} 
  We verify that $C := A \cup D'$ satisfies the assertions of the lemma by  
  identical constructions and arguments as in
    Lemma \ref{curves for families} using 
    Lemma \ref{the divisor in the lift} in place of \cite[Lemma 11.3.2]{HL}
       \end{proof} 

 \subsection{Theorem \ref{weak cdr for curves} and consequences} 
 
 \begin{proof} (Theorem \ref{weak cdr for curves})
 The proof is identical to that given in \S \ref{proof of generic cdr} of Theorem \ref{generic cdr}. 
 Note that we do not claim that the homotopy $H'$ on $\widehat{X_U}$ fixes its image i.e. that 
 it is a deformation retraction. Hence, we do not require the relative tropical homotopy
 from \S \ref{relative tropical homotopy section}. 
 
   \end{proof} 
   
    In the case that the morphism $\phi \colon X \to S$ in 
    Remark \ref{prove the simpler case wbc}
     is of relative dimension $1$ 
    i.e. for every $s \in S(K)$, $X_s$ is a $K$-curve, we can verify that 
    there exists deformation retractions of $\widehat{X}$ and $\widehat{S}$  
    which are compatible with $\widehat{\phi}$ and whose images 
    are $\Gamma$-internal.

       The method of proof is to first show the result locally around an arbitrary point.
      The following lemma then allows us to glue the 
  various relative homotopies to obtain a relative homotopy of the total family whose image is 
  relatively $\Gamma$-internal. 
  
         We employ the following notation in the lemma below. 
         Given a projective variety $V$, 
         recall from  \cite[\S 3.10]{HL} the notion of a definable metric 
         $m \colon V \times V \to \Gamma_\infty$. 
        Let $D \subset V$ be a closed sub-variety
        of $V$. 
         As in the proof of Lemma 10.3.2 in loc.cit., we define $\rho_D \colon V \to \Gamma_\infty$ as follows. 
         For $x \in V$, set $\rho_D(x) := \mathrm{sup}\{m(x,d) | d \in D\}$. 
         When there is no ambiguity about the divisor $D$ chosen, we simplify notation 
         and write $\rho_D$ in place of $\rho$.         
  
  \begin{lem} \label{foreshortening}
  Let $f \colon V' \to V$ be a morphism of projective $K$-varieties. 
  Let $D \subset V$ be a closed sub-variety and $D' := f^{-1}(D)$.
  Let $U := V \smallsetminus D$ and $U' := V' \smallsetminus D'$.  
   Let 
  $h \colon [0,\infty] \times U' \to \widehat{U'/U}$ be a homotopy. 
%
    Let $\epsilon \in [0,\infty)$ be $K$-definable. 
  There exists a $K$-definable 
   homotopy $g_{h,\epsilon} \colon [0,\infty] \times V' \to \widehat{V'/V}$ such that 
  if $x \in V$ and $\rho_D(x) \leq \epsilon$ then 
  $g_{h,\epsilon}(0,V'_{x}) = h(0,V'_x)$. 
    \end{lem}      
    
  \begin{proof} 
   We begin by verifying a certain technical condition that is necessary for the proof. 
    Let $m'$ be a definable metric on $V'$. 
    Let $m$ be a definable metric on 
  $V$.
   Let $\gamma \colon U \to [0,\infty]$ be defined as
     $\gamma(u) := \mathrm{sup}\{\gamma'(u') | u' \in f^{-1}(u)\}$ 
     where $\gamma'(u')$ is the smallest element  
     in $[0,\infty]$
     such that 
     $h(\gamma'(u'),u')$ belongs to the ball 
     $$B(u',m',\rho_D(f(u'))) := \{x \in {U'} | m'(u',x) \geq \rho_D(f(u')) \}.$$
We claim that 
     $\gamma$ is locally bounded on $U$ in the sense of 
     \cite[\S 10.1]{HL}.  
    
     Let $u' \in U'$ and $t_0 \in \Gamma$. 
     By continuity, 
     $h^{-1}(\widehat{B(u',m',\rho_D(f(u'))})$ is a $v+g$-closed neighbourhood of the point $(\infty,u') \in [0,\infty] \times U'$. 
     The $v+g$-topology of $[0,\infty] \times U'$ is the product of the $v+g$-topology on $[0,\infty]$ and the $v+g$-topology
     on $U'$.
      Hence
     we can assume that the closed neighbourhood of $(\infty,u')$ above contains 
     an open neighbourhood $O_1 \times O_2$ where $O_1$ is of the form $(s,\infty]$ and 
     $O_2$ is a $v$-open \footnote{It is $g$-open as well but we will use only that it is $v$-open.}
      neighbourhood of $u'$.
     We can further shrink $O_2$ so that for every 
     $o \in O_2$, $\rho_{D}(f(o))  = \rho_D(f(u'))$. 
     This is a consequence of the fact that $\rho_D$ is a $v+g$-continuous function 
     and hence $\rho_D^{-1}(\rho_D(f(u')))$ is a $v$-open neighbourhood of $f(u')$.
     We can hence replace $O_2$ with $O_2 \cap f^{-1}(\rho_D^{-1}(\rho_D(f(u'))))$.    
     It follows that for every $o \in O_2$, 
     $\gamma'(o) \leq s$. 
     We have thus shown that $\gamma' \colon U' \to \Gamma$ is locally bounded. 
     We now deduce that as a consequence $\gamma$ is locally bounded. 
     Indeed, let $u \in U$. 
     There exists $\beta \in \Gamma$ large enough so that 
     $B(u,m,\beta) \subset {U}$. 
     Since $f$ is a projective morphism, 
     we see that 
     $f^{-1}(B(u,m,\epsilon))$ is bounded and definable. 
     By \cite[Lemma 10.1.7]{HL}, we get that 
     $\gamma'_{|f^{-1}(B(u,m,\epsilon))}$ is bounded.

     
    For any $\delta$, 
   let $U_\delta := \{u \in U | \rho_D(u)  = \delta\}$
   and $\gamma_1(\delta) := \mathrm{sup}\{\gamma(u) | u \in U_\delta\}$.
   Since $U_\delta$ is bounded, by \cite[Lemma 10.1.7]{HL} and the locally boundedness 
   on the function $\gamma$, $\gamma_1(\delta) \in \Gamma$. 
   Observe that since $\gamma_1$ is piece-wise affine, there exists 
   $m \in \mathbb{N}$ and $c_0 \in \Gamma$ such that if $\delta \geq 0$ then 
   $\gamma_1(\delta) \leq m\delta + c_0$. 
   Let $\gamma_{10}(\delta) := m\delta + c_0$.
   Let $\epsilon' \in \Gamma(K)$ be such that $\epsilon' > \epsilon$. 
   Let $\gamma_2$ be a continuous function $\Gamma \to \Gamma$ which 
   is defined as follows. For every $x \leq \epsilon$, 
   $\gamma_2(x) = 0$. 
   For $x \in [\epsilon,\epsilon']$, 
  $$\gamma_2(x) := (\gamma_{10}(\epsilon')/(\epsilon'-\epsilon))x - (\gamma_{10}(\epsilon')/(\epsilon'-\epsilon))\epsilon$$   
  and 
  for $x \geq  \epsilon'$, $\gamma_2(x) := \gamma_{10}(x)$. 
  By construction, $\gamma_2$ is continuous. 
  Let $\gamma_3 \colon U' \to \Gamma$ be continuous function 
  $x \mapsto  \gamma_2(\rho_D(f(x)))$.

      Let $g_{h,\epsilon} \colon [0,\infty] \times V' \to \widehat{V'/V}$ be defined 
    as follows. 
    We set $$(g_{h,\epsilon})_{|[0,\infty] \times U'} := h'[\gamma_3]$$
     and for every $t \in [0,\infty],
     d \in D'$, $$g_{h,\epsilon}(t,d) := d.$$
       
    We now show that the canonical extension $G_{h,\epsilon}$ is continuous.
    We proceed as in \cite[Lemma 10.3.2]{HL}. 
    It suffices to verify that $G_{h,\epsilon}$ is 
    continuous at a point $(t,d)$ with $t \in [0,\infty]$ and $d \in \widehat{D'}$. 
    By definition, $G_{h,\epsilon}(t,d) = d$.  
    Let $\mathbf{O} \subset \widehat{V'}$ be a neighbourhood of $d$.  
    We must show that there exists an open neighbourhood 
    $W$ of $(t,d)$ in $[0,\infty] \times \widehat{V'}$ such that $G_{h,\epsilon}$ maps the simple points 
    of $W$ to $\mathbf{O}$.    
    
    
      Let $M$ be a model of ACVF such that $\mathbf{O}$ is pro-definable over $M$ and 
      $d$ is an $M$-definable type. 
      Let $z$ be a realization of $d_{|M}$. 
      Let $\alpha \in [0,\infty)$ be the smallest value such that 
      $B(z,m',\alpha)^- := \{y \in V' | m'(y,z) > \epsilon \} \subset \mathbf{O}$. 
      Observe that $\alpha \in  \Gamma(M(z))$.
      Since $d$ is stably dominated, we get that $\alpha \in \Gamma(M)$.
      Let $W_0$ be the set of all $x \in V'$ such that 
      $B(x,m',\alpha)^-$ is contained in $\mathbf{O}$. 
       The set $W_0$ is $v+g$-open and definable with parameters in $M$. 
      Since $z \in W_0$, we see that $d \in \widehat{W_0}$. 
      
    Suppose the open set $W$ does not exist. 
    As the simple points are dense, 
   there exists sequences $t_i \mapsto t$, 
   $v_i \mapsto d$ but 
    $G_{h,\epsilon}(t_i,v_i) \notin \mathbf{O}$. 
    Since $v_i \mapsto d$ 
    and $\widehat{W_0}$ is an open neighbourhood of $d$, we get that 
       there exists $i_0$ 
    such that if $i \geq i_0$ then 
    $v_i \in W_0$.
    It follows that
     $B(v_i,m',\alpha)^- \subset \mathbf{O}$. 
     The assumption that $v_i \mapsto d$ implies that $f(v_i) \mapsto f(d)$ which in turn 
     implies 
      that 
    $\rho_D(f(v_i)) \mapsto \infty$. Hence after increasing $i_0$ suitably,
    we get that if $i \geq i_0$ then
    $G_{h,\epsilon}(t_i,v_i) \in \widehat{B(v_i,m',\alpha)^-}$. 
   This implies that
   $G_{h,\epsilon}(t_i,v_i) \in \mathbf{O}$ which gives a contradiction.   
   Clearly, from the construction, 
   $g_{h,\epsilon}(0,V'_x) = h(0,V'_x)$ for 
   every $x \in U_\epsilon$. 
                \end{proof} 
                
 \subsubsection{Deformation retractions for curves}
 
 As stated above, our goal is to show that when $\phi \colon X \to S$ is of relative dimension $1$, there exist
 deformation retractions of $\widehat{X}$ and $\widehat{S}$ which are compatible with the morphism 
 $\widehat{\phi}$. 
 We require the following lemmas to ensure that the homotopies we construct do indeed fix their image.


We introduce the following notation to simplify the statement and proof of
 Lemma \ref{lem : to ensure image is fixed}. 
 Let $F$ be a valued field and $C$ be an $F$-curve. Let $f \colon C \to \mathbb{P}^1_F$
 be a finite morphism. Let $D \subset \mathbb{P}^1$ be a divisor. We say that
 the pair $(f,D)$ is \emph{homotopy liftable} if the standard homotopy 
 with stopping divisor $\psi_{D} \colon [0,\infty] \times \mathbb{P}^1 \to \widehat{\mathbb{P}^1}$ 
 lifts to a homotopy $h \colon [0,\infty] \times C \to \widehat{C}$ via the morphism $\widehat{f}$.
 Given a $v+g$-continuous function $\gamma \colon \mathbb{P}^1 \to [0,\infty]$, 
 recall the cut-off homotopy $h[\gamma \circ f]$ from 
 \cite[Lemma 10.4.6]{HL}.

\begin{lem} \label{lem : to ensure image is fixed}
 Let $F$ be an algebraically closed valued field and let $C$ be an $F$-curve. 
 Let $f_1 \colon C \to \mathbb{P}^1_F$ and $f_2 \colon C \to \mathbb{P}^1_F$ be finite morphisms. 
 Let $D_1,D_2,D_{11},D_{22} \subset \mathbb{P}^1$ be finite $F$-definable sets with the following properties.   
 \begin{enumerate} 
 \item The pairs $(f_1,D_1)$, $(f_1,D_{11})$, $(f_2,D_2)$ and $(f_2,D_{22})$ are homotopy liftable. 
 \item We have the inclusions $D_{11} \subseteq D_1$ and $D_{22} \subseteq D_2$. 
 \item Let $D'_1 := f_1^{-1}(D_1)$, $D'_{11} := f_1^{-1}(D_{11})$,
 $D'_{2} = f_2^{-1}(D_2)$ and $D'_{22} := f_2^{-1}(D_{22})$.  
 Assume there exists a finite $F$-definable set $D' \subseteq D'_1 \cap D'_2$ 
 containing $D'_{11}$ and $D'_{22}$ such that $f_1(D') \cap f_2(D')$ contains $\{0,\infty\}$. 
 \end{enumerate} 
 Let $h_1$ and $h_{11}$ ($h_2$ and $h_{22}$)
 lift the standard homotopies $\psi_{D_1}$ and $\psi_{D_{11}}$ 
 ($\psi_{D_1}$ and $\psi_{D_{11}}$) respectively via $\widehat{f_1}$ ($\widehat{f_2}$). 
 Let $\gamma_1,\gamma_2 \colon \mathbb{P}^1 \to [0,\infty]$ 
 be $v+g$-continuous functions. 
 We then have that 
 the image of the composition 
 $h_2[\gamma_2 \circ f_2] \circ h_1[\gamma_1 \circ f_1]$ is the intersection of the images of the homotopies
 $h_2[\gamma_2 \circ f_2]$ and $h_1[\gamma_1 \circ f_1]$. 
 \end{lem} 
  \begin{proof}
  We simplify notation and write $h_1[\gamma_1]$ in place of $h_1[\gamma_1 \circ f_1]$. 
  Let $\Upsilon'_1$, $\Upsilon'_2$, $\Upsilon'_{11}$ and $\Upsilon'_{22}$ denote the images of the homotopies
  $h_1$,$h_2$, $h_{11}$, and $h_{22}$ respectively.
  Let $\Upsilon'$ denote the convex hull in $\widehat{C}$ of the union of the finite set $D'$ and 
  $\Upsilon'_{11} \cup \Upsilon'_{22}$.
  Furthermore, let $\Upsilon'_1[\gamma_1]$ and $\Upsilon'_2[\gamma_2]$ denote the 
  images of the homotopies $h_1[\gamma_1]$ and $h_2[\gamma_2]$.
  Since $h_1$ and $h_2$ are deformation retractions, we see that it suffices to check that if
  $x \in \widehat{C}$ lies in the 
  image of the
  composition
  $h_2[\gamma_2] \circ h_1[\gamma_1]$
  then $x$ belongs to  
  $\Upsilon'' := \Upsilon'_1[\gamma_1] \cap \Upsilon'_2[\gamma_2]$.
  
     By construction of the standard homotopies and their lifts via finite morphisms
     in \cite[\S 7.5]{HL}, we see that it suffices to show that if 
     $x \in C$ then $h_2[\gamma_2](0,h_1[\gamma_1](0,x))$ belongs to 
      $\Upsilon''$.
      We prove the following lemma. 
      \begin{lem} \label{lem : same retraction point}
     Observe that $h_{11}(0,x) \in \Upsilon'_{11} \subset \Upsilon'$.
      Let $t$ be the largest element in $[0,\infty]$ such that 
      $h_{11}(t,x) \in \Upsilon'$ and $p_{11} := h_{11}(t,x)$. 
      Likewise, let $t'$ be the largest element in $[0,\infty]$ such that 
      $h_{22}(t',x) \in \Upsilon'$ and $p_{22} := h_{22}(t',x)$. 
      Then $p_{22} = p_{11}$. 
       \end{lem} 
       \begin{proof} 
        By construction,
        $\Upsilon'$ is such that 
        for every $a,b \in \Upsilon'$
       there exists no path in $\widehat{C}$ from $a$ to $b$ that does not intersect
       $\Upsilon'$ outside of the points $a$ and $b$.
              If $p_{11} \neq p_{22}$ then we deduce that there exists a path from 
       $p_{11}$ to $p_{22}$ which does not lie in $\Upsilon'$ outside of $\{p_{11},p_{22}\}$. 
       This is not possible. Hence $p_{11} = p_{22}$.       
       \end{proof} 
       
            Let $p := p_{11} = p_{22}$. 
       Let $P_1$ be the path from $x$ to $p$ given by
      $r \mapsto h_{11}(r,x)$ for $r \in [t,\infty]$.
      Let $P_2$ denote the path 
      $r \mapsto h_{22}(r,x)$ from 
      $x$ to $p$ for $r \in [t',\infty]$. 
       Observe that 
      the morphism $f'_1$ maps the paths
      $P'_1$ and $P'_2$ to paths from
      $f_1(x)$ to $f_1(p)$. 
      Since there is exactly one injective path up to re-parametrization
      from $f_1(x)$ to $f_1(p)$ in $\widehat{\mathbb{P}^1}$
      and this path lifts uniquely to a path in
      $\widehat{C}$, we deduce that 
      the images of $P'_1$ and $P'_2$ must coincide. 
      
      We identify the path $P_1$ with the closed interval $[t,\infty]$.
      Our discussion above implies that the path $r \mapsto h'_{22}(r,x)$ moves along $P_1$ for $r \in [t',\infty]$. 
      Let $t'_1$ be the largest point on $[t,\infty]$ that belongs 
      to $\Upsilon'_1[\gamma_1]$ and similarly, let $t'_2$ be the largest 
      element that belongs to $\Upsilon'_2[\gamma_2]$. 
      Recall that $h'_1$ and $h'_2$ are cut-offs of the homotopies 
      $h'_{11}$ and $h'_{22}$ respectively.
      Hence, we see that 
      $h'_1[\gamma_1]$ defines an injective path from $\infty$ to $t'_1$ and 
      fixes $t'_1$. Likewise, 
      $h'_2[\gamma_2]$ defines an injective path from $\infty$ to $t'_2$ and 
      fixes $t'_2$. 
    
      Suppose $t'_1 \leq t'_2$. In this case, $t'_1 \in \Upsilon'_2[\gamma_2]$ and 
     $$h'_2[\gamma_2] \circ h'_1[\gamma_1] (0,x) = t'_1$$
      which belongs to the intersection $\Upsilon'_1[\gamma_1] \cap \Upsilon'_2[\gamma_2]$.
      Now, suppose $t'_1 > t'_2$. 
      This implies that $t'_2 \in \Upsilon'_1[\gamma_1]$ and we see that 
      the image of the composition is $t'_2$ which lies in the 
      intersection 
      $\Upsilon'_1[\gamma_1] \cap \Upsilon'_2[\gamma_2]$.
      This completes the proof. 
    \end{proof} 
    
   \begin{cor} \label{statement 1 wbc}
    Let $S$ be a smooth connected $K$-curve and $X$ be a quasi-projective $K$-variety. 
    Let $\phi \colon X \to S$ be a surjective morphism such that every irreducible component of 
    $X$ dominates $S$. We assume in addition that the fibres of $\phi$ are of dimension $1$. 
 Let $\{\xi_i \colon X \to \Gamma_{\infty}\}$ be a finite collection of 
   $K$-definable functions.
  Recall that the functions $\xi_i$ extend to functions $\xi_i \colon \widehat{X} \to \Gamma_\infty$.
   Let $G$ be a finite algebraic group acting on $X$ such that the action of $G$ respects the fibres 
   of the morphism $\phi$. 
    There exists compatible deformation retractions
     $(H,\Upsilon)$ of $\widehat{S}$ and $(H',\Upsilon')$ of $\widehat{X}$ such that
    \begin{enumerate}
    \item The images $\Upsilon \subset \widehat{S}$ and $\Upsilon' \subset \widehat{X}$ are $\Gamma$-internal. 
    \item  
     The homotopy $H'$ respects the functions $\xi_i$ i.e. $\xi_i(H'(t,p)) = \xi_i(p)$ for every $p \in \widehat{X}$ and 
     $t \in I$. 
    \item  The action of the group $G$ on $X$ extends to an action on $\widehat{X}$.  
      The homotopy $H'$ can be taken to be $G$-equivariant i.e. for every 
      $g \in G'$ and $p \in \widehat{X}$, $H'(t,g(p)) = g(H'(t,p))$. 
      \item The homotopy $H'$ is Zariski generalizing i.e. if $U \subset X$ is a Zariski open subset 
      then $H'$ restricts to a well defined homotopy on $\widehat{U}$. 
    \end{enumerate} 
 \end{cor} 
 \begin{proof} 
 Similar to Remark \ref{prove the simpler case wbc}, we 
 can assume at the outset that $X$ and $S$ are projective, $S$ is a smooth connected $K$-curve, 
 and the fibres of the morphism $\phi \colon X \to S$ are pure.
 
   Let $s \in S(K)$. 
   We show that there exists a Zariski open neighbourhood $U$ of $s$ and a
    homotopy 
   $H' \colon [0,\infty] \times X_U \to \widehat{X_U/U}$ whose image 
   $H(0,X_U)$ is relatively $\Gamma$-internal over $U$. 
   Note that this is slightly different from the assertion in 
   Theorem \ref{weak cdr for curves} since
   we are asking for a single homotopy and not a chain of homotopies as would appear if we 
   went through the construction in the proof of loc.cit. 
   
  By the first part of Lemma \ref{Local factorization}, there exists a Zariski open set $U_0$ that contains 
  $s$ and is such that the morphism 
   $\phi \colon X \to S$ factors through a finite morphism $f_0 \colon X_{U_0} \to \mathbb{P}^1 \times U_0$.
   We have the following commutative diagram. 
    
           $$
\begin{tikzcd}[row sep = large, column sep = large]
X_{U_0} \arrow[r, "{f_0}"] \arrow[d,"\phi"] & \mathbb{P}^{1} \times U_0 \arrow[ld,"p_2"] \\
U_0 
\end{tikzcd}
$$ 
 
  We fix coordinates on $\mathbb{P}^1$. 
  Let $B_0 := \mathbb{P}^1 \times U_0$. 
  Observe that any closed subset of $B_0$ that is generically finite over $U_0$ will be finite over $U_0$. 
     We apply Lemma \ref{the divisor in the lift} to obtain a divisor $P_{00} \subset B_0$ 
      that satisfies the following properties
      \begin{enumerate} 
      \item $P_{00}$ is finite over $U_0$ and contains the closed subset $\{0,\infty\} \times U_0$.
      \item  For any divisor $T \subset B_0$ that contains $P_{00}$ and is finite over $U_0$, 
     the homotopy 
     $\psi_T \colon [0,\infty] \times B_0 \to \widehat{B_0/U_0}$ (cf. \cite[\S 10.2]{HL}) 
     lifts 
     uniquely to a  
  homotopy $h_{T} \colon [0,\infty] \times A \to \widehat{A/U}$.
   Furthermore, for every $j$, $h_{T}$ preserves the levels of the function $\xi_j$ and 
  is $G$-invariant.  
 \end{enumerate} 
 The fact that $\psi_T$ is a homotopy follows from the fact that 
 $T$ is finite over $U_0$ which then implies that the lift $h_{T}$ is a homotopy 
 by \cite[Lemma 10.1.1]{HL}. Observe that the image of $h_{T}$ is 
 relatively $\Gamma$-internal. 
      
        Let $D := \{x_1,\ldots,x_n\}$ be the complement of the set $U_0 \subset S$. 
      For every $1 \leq i \leq n$, there exists a Zariski open neighbourhood $U_i$ 
      of $x_i$ and a divisor $P_{ii} \subset \mathbb{P}^1 \times {U_i}$ 
      that satisfies the analogous properties as fulfilled by $P_{00}$ which are listed above.
       
        Let $W := \bigcap_{0 \leq i \leq n} U_i$.
        For every $i$, let $P'_{ii} := f_i^{-1}(P_{ii})$. 
        Let $P'$ denote the Zariski closure in $X$ 
        of the closed subset $\bigcup_i P'_{ii}$. 
        For every $i$, define 
        $P_i := f_i(P' \cap X_{U_i})$ and 
        $P'_i := f_i^{-1}(f_i(P' \cap X_{U_i}))$. 
         Observe that $P_{i}$ also satisfies the analogous versions of the points (1) - (2) 
         as fulfilled by $P_{00}$ above. 
        Hence, we see that for every $i$, the homotopies $\psi_{P_{i}}$
        lift to homotopies $h_i$ which respect the definable function 
        $\xi_j$ and the action of the group $G$. 
         
      Let $D_i := S \smallsetminus U_i$. Let $m$ be a metric on 
      $S$. Recall the function $\rho_D \colon S \to [0,\infty]$ defined as
      $x \mapsto \mathrm{sup}\{m(x,d) | d \in D\}$.  
      Let $\epsilon \in [0,\infty](K)$ and $\delta \in [0,\infty](K)$ be such that the following is satisfied. 
      For every $i$, 
      $B(x_i,m,\epsilon) \subset \{x \in S | \rho_{D_i}(x) \leq \delta\}$ where 
      $B(x_i,m,\epsilon) := \{y \in S | m(x_i,y) \geq \epsilon\}$. 
       
      For every $i$, we extend the homotopy $h_i$ to 
      the whole of $X$ such that the image of the composition 
      of the extended $h_i$ will be relatively $\Gamma$-internal over $S$. 
      We proceed as follows. 
      
       Note
       that 
      $$\{x \in S | \rho_{D_0}(x) \geq \epsilon\} = \bigcup_{x \in D_0} B(x,m,\epsilon).$$
     By Lemma \ref{foreshortening}, 
     we can extend $h_0$ to a homotopy 
     $h_{f0} \colon [0,\infty] \times X \to \widehat{X/S}$ such that for every 
     $x \in S$ with $\rho_{D_0}(x) \leq \epsilon$, 
     $h_{f0}(0,X'_x) = h_0(0,X'_x)$. This implies in particular that the image 
     of $h$ is relatively $\Gamma$-internal over
     $\{x \in S | \rho_{D_0}(x) \leq \epsilon\}$. Likewise,
     for every $i \geq 1$, we apply Lemma \ref{foreshortening} 
     to extend $h_i$ to a homotopy
     $h_{fi} \colon [0,\infty] \times X \to \widehat{X/S}$ such that
     the image 
     of $h_{fi}$ is relatively $\Gamma$-internal over $\{x \in S | \rho_{D_i}(x) \leq \delta\}$. Hence
     the image of $h_{fi}$ is relatively $\Gamma$-internal over 
     $B(x_i,m,\epsilon)$ since $B(x_i,m,\epsilon) \subset \{x \in S | \rho_{D_i}(x) \leq \delta\}$.     
     
     Let $h := h_{fn} \circ \ldots h_{f0} \colon [0,\infty] \cup \ldots \cup [0,\infty] \times X \to \widehat{X/S}$
     be the composition of the homotopies described above. Observe from the construction that the 
     image of $h$ is relatively $\Gamma$-internal over $S$.
     Let $\Upsilon \subset \widehat{X/S}$ denote the image of the homotopy 
     $h$. By Lemma \ref{lem : to ensure image is fixed}, 
     the homotopy $h$ fixes its image. 
     
  
  By \cite[Theorem 6.4.4]{HL},
  there exists a 
 pseudo-Galois cover $\alpha \colon S' \to S$
and
  a finite collection of $K$-definable 
  functions $\mu_j \colon  S' \to \Gamma_{\infty}$ such that,
  for $I$ a generalised interval, 
  a homotopy $\beta \colon I \times S \to \widehat{S}$
  which lifts to a homotopy $a' \colon I \times S' \to \widehat{S'}$ that preserves the levels 
  of the functions $\mu_j$ also induces a homotopy 
  $a \colon I \times \widehat{\Upsilon} \to \widehat{\Upsilon}$
   that is $G$-invariant and 
  respects the levels of the functions $\xi_i$.
  The homotopy $a$ is compatible with the homotopy 
  $\beta$ for the morphism $\widehat{\phi}$. 
  We apply \cite[Remark 11.1.3 (2)]{HL} to complete the proof while noting that 
  since each of the deformations
   in the composition are Zariski generalizing, the composition is Zariski generalizing as well. 
\end{proof}

\section{Locally trivial morphisms} \label{rank 1 projective bundles} 
 
            In this section, we show that compatible homotopies
            can be constructed 
            for locally trivially morphisms. Unfortunately, our method doesn't guarantee
            that the homotopy on the source fixes its image and hence it might not be a deformation retraction. 
            It is possible that in this situation an additional hypothesis is required. We 
            demonstrate this 
            explicitly for $\mathbb{P}^1$-bundles
             which satisfy an additional finiteness requirement. 
             The method presented below adapts the proof of the relative curve 
             homotopy in \cite{HL}. 
\\

\subsection{Locally trivial morphisms} \label{section : locally trivial morphisms}

Let $m$ be a definable metric on a $K$-variety $S$. Suppose $D \subset S$ is a Zariski closed subset.
Recall the $v+g$-continuous function, $\rho_{D} \colon  S \to \Gamma_\infty$ given by 
$x \mapsto \mathrm{sup} \{m(x,d) | d \in D\}$. 

\begin{prop}\label{non-Archimedean distance problem} 
  Let $S$ be a projective $K$-variety. 
  Let $m$ be a definable metric on $S$. 
  Suppose that for every $x \in S(K)$, we are given a Zariski open neighbourhood $U_x \subset S$
  of $x$
  defined over $K$.
Let $D_x := S \smallsetminus D_x$. 
   We have that there exists finitely many points 
 $\{x_1,\ldots,x_n\} \subset  S(K)$ 
 and $\{\epsilon_1,\ldots,\epsilon_n\} \subset \Gamma(K)$ such that 
 $$S = \bigcup_{1 \leq i \leq n} \{x \in U_{x_i} | \rho_{D_{x_i}}(x) \leq \epsilon_i\}.$$
\end{prop} 
\begin{proof} 
%
 Let $\langle a_i \rangle$ be a sequence in $\Gamma(K)$. 
      We define a sequence $\langle x_i \rangle_i \subset S(K)$
      as follows. 
  Let $x_1 \in S(K)$. Given $x_i$, let $x_{i+1}$ be a point in $F_i := S \smallsetminus (\bigcup_{1 \leq j \leq i} U_{x_j})$ if 
   $F_i$ is not empty. 
   Observe that the sequence must terminate. Let us assume that it terminates at $n \in \mathbb{N}$ i.e. $F_n = \emptyset$. 
   
   We choose an $n-1$-tuple
   $(\delta_1,\ldots,\delta_{n-1}) \in \prod_{1 \leq i \leq n-1} [a_i,\infty)(K)$ 
   which satisfies the following properties. 
   Firstly, $\delta_{n-1}$ is such that 
   $\{y \in S| \rho_{F_{n-1}}(y) \geq \delta_{n-1}\} \cap D_{x_{n}} = \emptyset$.
   This is possible because $F_{n-1}$ is disjoint from $D_{x_n}$ and 
   $\rho_{F_{n-1}}$ is $v+g$-continuous and hence bounded in $\Gamma$ on $D_{x_n}$. 
   
   Suppose 
   $1 \leq i \leq n-1$. Having chosen $\delta_i$, we choose $\delta_{i-1}$ such that
    the following is true. 
   \begin{enumerate}
   \item $\delta_{i-1} \in [a_{i-1},\infty)$ and $\delta_{i-1} > \delta_i$. 
   \item If $$F_{i-1,\delta_{i}} := \{x \in F_{i-1}|\rho_{F_{i}}(x) \leq \delta_{i}\}$$
    and $$D_{x_{i},\delta_{i}} := \{x \in D_{x_{i}}|\rho_{F_{i}}(x) \leq \delta_{i}\}$$
   then 
   $\{y \in S| \rho_{F_{i-1}}(y) \geq \delta_{i-1}\} \cap D_{x_{i},\delta_{i}} = \emptyset$. 
   \end{enumerate} 
    The existence of $\delta_{i-1}$
    can be deduced as follows. 
    Observe that $F_i = D_{x_i} \cap F_{i-1}$. 
    Hence, we see that 
    $D_{x_{i},\delta_{i}}$ is disjoint from $F_{i-1}$. We now use the fact that
    $\rho_{F_{i-1}}$ is $v+g$-continuous and hence bounded on 
    $D_{x_{i},\delta_{i}}$ .  
  
      We choose an $n-1$-tuple $(\epsilon_1,\ldots,\epsilon_{n-1}) \in [0,\infty)(K)$ such that
      the following holds.
    For $1 \leq i \leq n-1$, if 
       $$A_i := S \smallsetminus (\bigcup_{1 \leq j \leq i} \{x \in U_{x_j} | \rho_{D_{x_j}}(x) \leq \epsilon_j\})$$
    then $A_i \subseteq \{x \in S | \rho_{F_i}(x) \geq \delta_i\}$.
     We construct $(\epsilon_1,\ldots,\epsilon_{n-1})$ as follows. 
     We set $\epsilon_1 = \delta_1$. Suppose, we have chosen 
     $(\epsilon_1,\ldots,\epsilon_i)$ with $i \leq {n-2}$ appropriately. 
    Let $\epsilon_{i+1} \in \Gamma(K)$ be such that $\epsilon_{i+1} > \delta_i$.
    We claim that 
     $A_{i+1} \subseteq \{x \in S | \rho_{F_{i+1}}(x) \geq \delta_{i +1}\}$.
     
     Let $z \in A_{i+1}$ and suppose that $z \notin \{x \in S |\rho_{F_{i+1}}(x) \geq \delta_{i+1}\}$ i.e. $\rho_{F_{i+1}}(z) < \delta_{i+1}$. 
 Since $z \in A_{i+1}$, we have that
    $z \in A_{i}$ and hence $z \in  \{x \in S | \rho_{F_i}(x) \geq \delta_i\}$.
    Let $y \in F_i$ be such that 
    $\rho_{F_i}(z) = m(z,y)$. We must have that 
    $y \in F_{i,\delta_{i+1}}$. Indeed, 
    suppose that for some 
    $y' \in F_{i+1}$, $m(y',y) \geq \delta_{i+1}$. 
    It follows that 
    \begin{align*} 
    m(y',z) &\geq \mathrm{inf}\{m(y',y),m(y,z)\} \\
                                          & \geq \mathrm{inf}\{\delta_{i+1},\delta_i\} \\
                                          & \geq \delta_i &\mbox{ (by our choice of } \delta_i). 
                                          \end{align*} 
     This contradicts our initial assumption that 
     $\rho_{F_{i+1}}(z) < \delta_{i+1}$. 
     
     By assumption, $\rho_{D_{x_{i+1}}}(z) > \epsilon_{i+1}$. 
      Using arguments as above and the fact that $\epsilon_{i+1} > \delta_i > \delta_{i+1}$, we get 
     that 
     if $x \in D_{x_{i+1}}$ is such that $\rho_{D_{x_{i+1}}}(z) = m(x,z)$ then 
     $x \in D_{x_{i+1},\delta_{i+1}}$. 
     
     Observe that 
     \begin{align*} 
      m(x,y)  &\geq \mathrm{inf}\{m(x,z),m(y,z)\} \\
                  &\geq \mathrm{inf}\{\epsilon_{i+1},\delta_i\} \\
                  & \geq \delta_i .     
     \end{align*}   
        However this is not possible by our choice of $\delta_i$.

     Finally, we choose $\epsilon_n \in \Gamma(K)$ such that 
    the following holds. 
    By construction, $F_{n} := D_{x_n} \cap F_{n-1}$  
    is empty.
    Recall that we chose $\delta_{n-1}$ such that 
     $\{y \in S| \rho_{F_{n-1}}(y) \geq \delta_{n-1}\} \cap D_{x_{n}} = \emptyset$.
     It follows that $\rho_{D_{x_n}}$ is bounded in $\Gamma$ when restricted to 
     $\{y \in S| \rho_{F_{n-1}}(y) \geq \delta_{n-1}\}$. 
     Let $\epsilon_n$ be such that 
     $\{a \in S | \rho_{D_{x_n}}(a) \geq \epsilon_n\} \cap \{y \in S| \rho_{F_{n-1}}(y) \geq \delta_{n-1}\} = \infty$.  
     One checks easily from the construction that 
     $(\epsilon_1,\ldots,\epsilon_n)$ satisfies the required property. 
     
\end{proof}

 Proposition \ref{non-Archimedean distance problem} used in conjunction with Lemma 
 \ref{foreshortening} allows us to construct compatible deformations for certain locally trivial 
 morphisms. 
 
\begin{cor} \label{locally trivial morphisms}
  Let $V$ be a projective $K$-variety and 
  $\psi \colon [0,\infty] \times V \to \widehat{V}$ be a homotopy whose image is $\Gamma$-internal. 
    Let $\phi \colon X \to S$ be a morphism of projective varieties such that for every 
  $s \in S(K)$, there exists a Zariski open neighbourhood $U_s$ of $s$ and 
  a $U_s$-isomorphism 
  $f_{U_s} \colon X_{U_s} \to U_s \times V$. 
  Let $\{\xi_i \colon S \to \Gamma_\infty\}$ be finitely many $K$-definable functions. 
    We then have that there exists a homotopy
  $H' \colon I \times \widehat{X} \to \widehat{X}$ and a deformation retraction
  $H \colon I \times \widehat{S} \to \widehat{S}$ which are compatible for the morphism 
  $\widehat{\phi}$. Furthermore, 
  the deformation $H$ respects the levels of the functions $\xi_i$ for every $i$. 
\end{cor} 
\begin{proof}
    Recall from \cite[Remark 9.8.4]{HL} that if 
    $V_1$ and $V_2$ are varieties over a valued field and 
    $\psi \colon I \times  {V_1} \to \widehat{V_1}$ is $v+g$-continuous then 
    the map 
     $(t,u,v) \mapsto \psi(t,u) \otimes v$ defines a $v+g$-continuous map  
     $\psi' \colon I \times (V_1 \times V_2) \to \widehat{(V_1 \times V_2)/V_2}$.  
     It follows from this and 
     our assumptions on $\phi$ and $V$ 
      that for every 
     $s \in S(K)$ there exists a 
     Zariski open neighbourhood $U_s \subseteq S$ of $s$ and a
     homotopy 
     $h'_{U_s} \colon [0,\infty] \times X_{U_s}  \to \widehat{X_{U_s}/U_s}$ such that the image is relatively $\Gamma$-internal. 
     
     Let $m$ be a definable metric on $V$. 
     We apply Proposition \ref{non-Archimedean distance problem} to the given data and obtain finite sets 
   $\{s_1,\ldots,s_m\} \subset S(K)$ and $\{\epsilon_1,\ldots,\epsilon_m\} \subset \Gamma(K)$ such that  
   $S = \bigcup_{1 \leq i \leq m} \{x \in S|\rho_{D_{s_i}}(x) \leq \epsilon_i\}$ where 
   $D_{s_i}:= S \smallsetminus U_{s_i}$. 
   For every $1 \leq i \leq m$, we simplify notation and write $U_i := U_{s_i}, D_{i} := D_{s_i}$ and $h'_{U_{s_i}} := h'_i$.
   
    By Lemma \ref{foreshortening}, 
    we can extend 
    the homotopy $h'_i : [0,\infty] \times X_i \to \widehat{X_i/U_i}$ to a homotopy 
    $h'_{fi} \colon [0,\infty] \times X \to \widehat{X/S}$ such that for every 
    $s$ with $\rho_{D_i}(s) \leq \epsilon_i$, 
    $h'_{fi}(0,X_{is}) = h'_i(0,X_{is})$. In particular 
    the image of $h'_{fi}$ is relatively 
    $\Gamma$-internal over 
    $\{s \in S | \rho_{D_i}(s) \leq \epsilon_i\}$. 
    Let $h'_f:= \Sigma_{1 \leq i \leq m} h'_{fi} \colon [0,\infty] \cup \ldots [0,\infty] \times X \to \widehat{X/S}$ be the 
    composition of the homotopies $h'_{fi}$. 
    Let $I_1$ be the glueing of the $m$-copies of the intervals $[0,\infty]$. 
    By construction, $h'_f \colon I_1 \times X \to \widehat{X/S}$ and its image $\Upsilon'$ is relatively $\Gamma$-internal. 
    
       By \cite[Theorem 6.4.4]{HL}
        there exists a finite pseudo-Galois cover 
       $f \colon S' \to S$ and a
      a finite number of $K$-definable functions 
     $\{\mu_j \colon S \to \Gamma_\infty\}_j$ such that a homotopy 
     $H \colon I_2 \times S \to \widehat{S}$ that lifts to a
     homotopy $I_2 \times S' \to \widehat{S'}$ and respects the level of the functions 
     $\{\mu_j\}_j$ must also give a homotopy 
     $H'_1 \colon I_2 \times \Upsilon \to \widehat{\Upsilon}$ that is compatible 
     with the map $\Upsilon \to S$. 
     By \cite[Theorem 11.1.1]{HL}, there exists a deformation $H$ 
     that satisfies these properties, has a $\Gamma$-internal image and also 
     respects the definable functions $\xi_i$.  
     By construction, the composition
     $H' := H'_1 \circ H'_f \colon I_2 + I_1  \times X \to \widehat{X}$ is a homotopy 
     compatible with $H$ 
     and has $\Gamma$-internal image.        
\end{proof} 

\begin{cor} \label{locally trivial morphisms and curves} 
Let $C$ be a projective $K$-curve and $S$ be a projective $K$-variety. 
Let $\phi \colon X \to S$ be a projective morphism such that for every 
$s \in S$, there exists a Zariski open neighbourhood $U$ of $s$ and a 
$U$-isomorphism $f_U \colon X_U \to C \times U$. 
  Let $\{\xi_i \colon S \to \Gamma_\infty\}$ be a  finite family of $K$-definable functions. 

We then have that there exists a homotopy 
  $H' \colon I \times \widehat{X} \to \widehat{X}$ and a deformation
  $H \colon I \times \widehat{S} \to \widehat{S}$ which are compatible for the morphism 
  $\widehat{\phi}$. Furthermore, 
  the deformation $H$ respects the levels of the functions $\xi_i$ for every $i$. 
\end{cor} 
\begin{proof}
  Let $C$ be as in the statement of the corollary. 
  \cite[Theorem 7.5.1]{HL} implies the existence of a deformation 
  $\psi \colon [0,\infty] \times C \to \widehat{C}$ whose image is $\Gamma$-internal. 
  Applying Corollary \ref{locally trivial morphisms} proves the result. 
\end{proof} 

\subsubsection{Compatible deformations for $\mathbb{P}^1$-bundles}  
             
   Note that Corollary \ref{locally trivial morphisms and curves} 
   doesn't guarantee that the existence of compatible deformation retractions. 
   It is possible that this is true only after an additional hypothesis. 
   We present such an instance of adding a hypothesis to guarantee the existence 
   of compatible deformations in the case of  $\mathbb{P}^1$-bundles.

  Let $S$ be a quasi-projective $K$-variety.
  \begin{defi}
 \emph{A morphism $\phi \colon X \to S$ is a} $\mathbb{P}^1$-bundle
  \emph{if there exists a
  finite Zariski open covering $\{U_i\}_{1 \leq i \leq m}$ of $S$ such that 
  for every $i$, there exists an isomorphism $g_i \colon X_i := \phi^{-1}(U_i) \to U_i \times \mathbb{P}^1_k$ over $U_i$. We encode the data of the $\mathbb{P}^1$-bundle using the 
  tuple $(\phi \colon X \to S,  \{U_i\}_i, \{g_i\}_i)$.}      
 \end{defi}
 
 \begin{rem}
  \emph{For $i,j$, the transition maps
   \begin{align*}
  g_{ij} := [g_j \circ g_i^{-1}]_{U_i \cap U_j} \colon (U_i \cap U_j) \times \mathbb{P}^1_k \to (U_i \cap U_j) \times \mathbb{P}^1_k
  \end{align*}
     are isomorphisms over $U_i \cap U_j$ i.e. 
     for every $u \in U_i \cap U_j$, 
     $g_{ij}^u := {g_{ij}}_{|\phi^{-1}(u)} \colon \mathbb{P}^1_K \to \mathbb{P}^1_K$  
   is a well defined automorphism.  Let $X_{ij} := X_i \cap X_j$. } 
 \end{rem} 
 
      The projective bundles we consider will satisfy the following finiteness hypothesis.
       For the remainder of this section, we fix a system of coordinates on 
     $\mathbb{P}^1_K$. 
   \\ 
   
    ($\mathbf{F}$) \emph{There exists a divisor $D \subset X$ such that 
       \begin{enumerate}
       \item The morphism $\phi$ restricts to a finite 
    map from $D$ onto $S$. 
       \item For every $1 \leq i \leq m$, $g_i(D \cap X_i)$ contains $U_i \times \{0,\infty\}$.    
       \end{enumerate}  }  

  \begin{rem} \label{conditions for glueing}
   \emph{Let $(\phi \colon X \to S, \{U_i\}_i, \{g_i\}_i\}$ be a rank-$1$ projective bundle. 
   Let $X_i = \phi^{-1}(U_i)$.  
   To construct a homotopy $H' \colon I \times X \to \widehat{X/S}$ 
it suffices to construct 
    a family of homotopies
   $H'_i \colon I \times Y_i \to \widehat{Y_i/U_i}$
    which satisfy the obvious glueing conditions. 
  Precisely, if $u \in U_i \cap U_j$ then $H'_i$ restricts to a homotopy on the fibre 
  $\widehat{Y'_{iu}}= \widehat{\mathbb{P}^1_K}$ such 
  that for every $t \in I$ and $x \in \widehat{Y'_{iu}}$, $H'_i(t,x) = \widehat{g_{ji}}(H'_j(t,\widehat{g_{ij}}(x)))$.}    
  \end{rem}

 \begin{thm} \label{horizontal divisor implies deformation}
      Let $\{\phi \colon X \to S,\{U_i\}_i, \{g_i\}_i\}$ be a $\mathbb{P}^1$-bundle which satisfies the hypothesis 
      $(\mathbf{F})$. There exists a pair 
      of deformation retractions  
      $(H' \colon I \times \widehat{X} \to \widehat{X}, \Upsilon')$ and 
      $(H \colon I \times \widehat{S} \to \widehat{S}, \Upsilon)$ 
      which are compatible for the morphism $\widehat{\phi}$ and whose images
      $\Upsilon' \subset \widehat{X},\Upsilon \subset \widehat{S}$
       are $\Gamma$-internal subsets. 
  \end{thm}

\begin{proof} 
For every $i$, let $Y_i := U_i \times \mathbb{P}^1_K$. 
    By assumption, there exists a closed subset $D \subset X$ such that $\phi_{|D} \colon D \to S$
     is finite and for every $i \in \{1,\ldots,n\}$,
     $g_i(D \cap X_i) \subset Y_i$ contains
     $U_i \times \{0,\infty\}$. Let 
     $E_i := g_i(D \cap X_i)$. 
      Let $p_i \colon Y_i \to U_i$ be the projection map. 
     We have that ${p_i}_{|E_i} \colon E_i \to U_i$ is finite. 
     Recall from the paragraph above \cite[Lemma 9.5.3]{HL}, the construction of the 
     homotopy
      $\psi_{E_i}$. 
      The homotopy $\psi_{E_i}$ is such that 
     for every $u \in U_i$, $\psi_{E_i}$ restricts to the homotopy on the 
     fibre $Y_{iu} = \mathbb{P}^1_K$ defined by the standard homotopy with cut-off 
     determined by $E_i \cap Y_{iu}$.  
    By \cite[Lemma 10.2.1]{HL}, the homotopy   
     $\psi_{E_i} \colon [0,\infty] \times Y_i \to \widehat{Y_i/U_i}$ is 
     well defined and continuous.

     By pulling back via the isomorphism
     $g_i$, we have 
     for every 
     $i$, a homotopy 
     $$\psi'_{i} \colon [0,\infty] \times X_i \to \widehat{X_i/U_i}$$ whose image is relatively
     $\Gamma$-internal. We claim that these homotopies glue to give a homotopy on
     $\widehat{X/S}$. We verify the claim as follows. 
     
           Let $i,j \in \{1,\ldots,n\}$. 
             Let $u \in U_i \cap U_j$. The homotopies $\psi'_{i}$ and $\psi'_{j}$ restrict 
             to define homotopies 
             $\psi'_{i,u},\psi'_{j,u} \colon [0,\infty] \times X_u \to \widehat{X_u}$ where 
             $X_u := \phi^{-1}(u)$. 
             We show that these homotopies coincide. 
           Since, $i$ and $j$ were chosen arbitrarily, this will imply that the homotopies
           $\psi'_{i}$ glue together. 
           Since $\psi'_{i}$ and $\psi'_{j}$ are definable maps, 
           it suffices to consider the case when $u$ is defined over $K$.

              The isomorphisms $g_i$ and $g_j$ imply a definable 
              automorphism $g_{iju} := g_{j,u} \circ g_{i,u}^{-1} \colon \mathbb{P}_{K(u)}^1 \to \mathbb{P}_{K(u)}^1$ 
              where 
              $g_{i,u} := {g_i}_{|X_u}$. 
              Let $\psi_{E_i,u}, \psi_{E_j,u} \colon [0,\infty] \times \mathbb{P}^1 \to \widehat{\mathbb{P}^1}$
              be the homotopies induced by $\psi_{E_i}$ and $\psi_{E_j}$. 
              As in Remark \ref{conditions for glueing}, it suffices to verify that 
              for every $x \in \mathbb{P}_{K(u)}^1$ and $t \in [0,\infty]$, 
              $\widehat{g_{iju}}(\psi_{E_i,u}(t,x)) = \psi_{E_j,u}(t,g_{iju}(x))$. 
              Since, $\widehat{\mathbb{P}^1}$ is definable, we reduce to showing the above equality 
              when $x \in {\mathbb{P}^1_{K(u)}}(K^{max})$ and $t \in [0,\infty](K^{max}) = [0,\infty](\mathbb{R})$.
              In this situation, we make use of the fact that 
              $\widehat{\mathbb{P}^1_{K(u)}}(K^{max})$ is homeomorphic to the Berkovich space 
              $B_{K^{max}}(\mathbb{P}^1)$. 
              
              If $x \in E_i$ then for every $t$, 
               $\widehat{g_{iju}}(\psi'_{E_i,u}(t,x)) = \psi'_{E_j,u}(t,g_{iju}(x))$.
              Let $\Upsilon_i \subset \widehat{\mathbb{P}^1}$ denote the convex hull of 
               $E_{iu}$ and likewise, $\Upsilon_j \subset \widehat{\mathbb{P}^1}$ denote the 
              convex hull of $E_{ju}$.
             Note that 
              $\widehat{g_{iju}}^{-1}(\Upsilon_j) = \Upsilon_i$. 
              Let $x \in \mathbb{P}^1(K^{max}) \smallsetminus \Upsilon_i$ and $O$ be the connected 
              component of $\widehat{\mathbb{P}^1}(K^{max}) \smallsetminus \Upsilon_i$ that contains $x$.  
               As $\widehat{g_{iju}}$ is a homeomorphism and 
               $\widehat{g_{iju}}(\Upsilon_i) = \Upsilon_j$, 
               there exists a connected component $O'$ of 
              $\widehat{\mathbb{P}^1} \smallsetminus \Upsilon_j$ such that
               $\widehat{g_{iju}}(O) = O'$. Note that these connected components 
               are Berkovich open balls.
              Since the morphism $g_{iju}$ is algebraic,
              the restriction ${g_{iju}}_{|O}$ at the level of algebras must be of the form $x \mapsto ax + b$ where 
              $a \in K^*$ and $b \in K$. It can then be checked by explicit calculation that   
                $\widehat{g_{iju}}(\psi'_{E_i,u}(t,x)) = \psi'_{E_j,u}(t,g_{iju}(x))$.
                
                We have thus shown that there exists a definable homotopy 
                $\psi \colon [0,\infty] \times X \to \widehat{X/S}$ whose image is relatively
                $\Gamma$-internal. Let $\Upsilon$ denote the image of $\psi$. 
      Let $\overline{S}$ be a projective $K$-variety that contains 
      $S$ as an open dense subvariety. 
      Let $d_{bord} \colon \overline{S} \to \Gamma_{\infty}$ be the 
      schematic distance to $S_{bord} := \overline{S} \smallsetminus S$. 
      Observe that $\Upsilon$ is $\sigma$-compact with respect to the function 
      $d_{bord} \circ \widehat{\phi}$. 
 By \cite[Theorem 6.4.4]{HL}, there exists a 
    pseudo-Galois cover $f \colon S' \to S$
    and a morphism
    $\kappa \colon  \Upsilon' := \Upsilon \times_{S} S' \to S' \times \Gamma_\infty^M$ for 
    some $M \in \mathbb{N}$ such that the restriction
    $\widehat{\kappa}_{|\widehat{\Upsilon'}} \colon \widehat{\Upsilon'} \to \widehat{S'} \times \Gamma^M_\infty$ is a homeomorphism onto its image.
     By loc.cit., we see that 
     there exists a finite number of $K$-definable functions 
     $\{\mu_j \colon S \to \Gamma_\infty\}_j$ such that a homotopy 
     $H \colon I \times S \to \widehat{S}$ that lifts to a
     homotopy $I \times S' \to \widehat{S'}$ and respects the level of the functions 
     $\{\mu_j\}_j$ must also give a homotopy 
     $H'_1 \colon I \times \Upsilon \to \widehat{\Upsilon}$ that is compatible 
     with the map $\Upsilon \to S$. 
     By \cite[Theorem 11.1.1]{HL}, the homotopy $H$ can be 
     chosen so that it satisfies these properties and in addition has a $\Gamma$-internal image. 
     By construction, the composition
     $H'_1 \circ \psi \colon I + [0,\infty] \times X \to \widehat{X}$ is compatible with $H$ 
     and has $\Gamma$-internal image.  
     Clearly, $H'_1 \circ \psi$ is a deformation retraction. 
                          \end{proof}

           \bibliographystyle{plain}
\bibliography{library}

\end{document}